\documentclass[reqno,12pt]{amsart}
\usepackage[a4paper,height=20cm,top=26mm,bottom=26mm,left=26mm,right=26mm]{geometry}

\usepackage{subfiles}
\usepackage{cleveref}

\usepackage{etex}
\usepackage{epic,eepic}
\usepackage{xypic}
\usepackage{amsmath,amssymb}

\usepackage{here}
\usepackage{bm}
\usepackage[all]{xy}
\usepackage{tikz}
\usetikzlibrary{calc}

\numberwithin{equation}{section}

\newtheorem{thm}{Theorem}[section]

\newtheorem{lem}[thm]{Lemma}
\newtheorem{prop}[thm]{Proposition}

\theoremstyle{definition}

\newtheorem{example}[thm]{Example}

\theoremstyle{remark}
\newtheorem{remark}[thm]{Remark}

\crefname{thm}{Theorem}{Theorems}
\crefname{cor}{Corollary}{Corollaries}
\crefname{lem}{Lemma}{Lemmas}
\crefname{prop}{Proposition}{Propositions}
\crefname{definition}{Definition}{Definitions}
\crefname{example}{Example}{Examples}
\crefname{claim}{Claim}{Claims}
\crefname{conjecture}{Conjecture}{Conjectures}
\crefname{remark}{Remark}{Remarks}
\crefname{figure}{Figure}{Figures}
\crefname{section}{Section}{Sections}
\crefname{subsection}{Section}{Sections}

\crefname{introthm}{Theorem}{Theorems}
\crefname{introcor}{Corollary}{Corollaries}
\crefname{introconj}{Conjecture}{Conjectures}

\def\e{e}
\def\C{{\mathbb C}}
\def\Z{{\mathbb Z}}
\def\ve{{\varepsilon}}

\newcommand\R{{\mathbb{R}}}
\newcommand\trop{{\mathrm{trop}}}

\newcommand\mg{{\mathfrak{g}}}

\newcommand\rY{{\mathsf{Y}}}
\newcommand\Ad{{\mathrm{Ad}}}
\newcommand{\bb}{{\mathfrak b}}

\newcommand{\maru}[1]{\raise0.2ex\hbox{\textcircled{\scriptsize{#1}}}}

\newfont{\bg}{cmr9 scaled\magstep4}%BIGZERO
\newcommand{\bigzerol}{\smash{\lower1.0ex\hbox{\bg 0}}}

\newcommand\qarrow[2]{\draw[->,shorten >=2pt,shorten <=2pt] (#1) -- (#2) [thick];} %arrow
\newcommand\qsarrow[2]{\draw[->,shorten >=4pt,shorten <=4pt] (#1) -- (#2) [thick];}
\newcommand\qdarrow[2]{\draw[->,dashed,shorten >=2pt,shorten <=2pt] (#1) -- (#2) [thick];} %dashed arrow
\newcommand\qarrowsa[2]{\draw[->,shorten >=2pt,shorten <=4pt] (#1) -- (#2) [thick];} 
\newcommand\qarrowsb[2]{\draw[->,shorten >=4pt,shorten <=2pt] (#1) -- (#2) [thick];} 

\tikzset{
  mid arrow/.style={postaction={decorate,decoration={
        markings,
        mark=at position .5 with {\arrow[#1]{stealth}}
      }}},
}

\begin{document}
\title[Quantum cluster algebras and 3D integrability]
{Quantum cluster algebras and 3D integrability: Tetrahedron and 3D reflection equations}

\author[Rei Inoue]{Rei Inoue}
\address{Rei Inoue, Department of Mathematics and Informatics,
   Faculty of Science, Chiba University,
   Chiba 263-8522, Japan.}
\email{reiiy@math.s.chiba-u.ac.jp}

\author[Atsuo Kuniba]{Atsuo Kuniba}
\address{Atsuo Kuniba, Institute of Physics, Graduate School
of Arts and Sciences, University of Tokyo, Komaba, Tokyo, Japan.}
\email{atsuo.s.kuniba@gmail.com}

\author[Yuji Terashima]{Yuji Terashima}
\address{Yuji Terashima, Graduate school of science, Tohoku University,
6-3, Aoba, Aramaki-aza, Aoba-ku, Sendai, 980-8578, Japan}
\email{yujiterashima@tohoku.ac.jp}

%\date{\today}
\date{October 21, 2023}

%%%%%%%%%%%%%%%%%%%%%%%%%%%

\begin{abstract}
We construct a new solution to the tetrahedron equation 
and the three-dimensional (3D) reflection equation
by extending the quantum cluster algebra approach by Sun and Yagi concerning the former. 
We consider the Fock-Goncharov quivers associated 
with the longest elements of the Weyl groups of type $A$ and $C$, 
and investigate the cluster transformations  
corresponding to changing a reduced expression into a `most distant' one. 
By devising a new realization of the quantum $y$-variables in terms of $q$-Weyl algebra, 
the solutions are extracted as the operators 
whose adjoint actions yield the cluster transformations of the quantum $y$-variables.
Explicit formulas of their matrix elements are also derived for some typical representations.
\end{abstract}

\keywords{}

%\subjclass[2000]{}

\maketitle

%\tableofcontents

\section{Introduction}

The tetrahedron equation and the 3D reflection equation, originally proposed in \cite{Z80} and \cite{IK97} respectively, 
are key to integrability in three dimensional (3D)  systems in the bulk and at the boundary. 
In this paper, we focus on specific formulations of these equations presented as follows:
\begin{align*}
&R_{456} R_{236} R_{135} R_{124} = R_{124} R_{135} R_{236} R_{456},
\\
&R_{457} K_{4689} K_{2379} R_{258} R_{178} K_{1356}R_{124} 
= R_{124} K_{1356} R_{178} R_{258} K_{2379} K_{4689} R_{457}.
\end{align*}
Here $R \in \mathrm{End}(V^{\otimes 3}) $ and 
$K \in \mathrm{End}(V\otimes W \otimes V \otimes W)$ for some vector spaces 
$V$ and $W$, which can be either finite or infinite-dimensional.
They are linear operators designed to encode the 
``scattering amplitudes" or ``Boltzmann weights" of integrable 3D systems. 
The indices of these operators specify the tensor components on which they have non-trivial actions. 
The above equations are to hold in $\mathrm{End}(V^{\otimes 6})$ and in 
$\mathrm{End}(V \otimes V \otimes W \otimes V \otimes V \otimes W \otimes V \otimes V \otimes W)$, 
respectively.
A solution to the 3D reflection equation means a pair $(R,K)$ in which 
$R$ is also required to satisfy the tetrahedron equation by itself.

As for the tetrahedron equation, several interesting solutions have been obtained up to now.
See for example
\cite{Z81, B83, BB92, KMS93,KV94, SMS96, MS97, BS06, S08, BMS08, BV15, KMY23} and the references therein.
On the other hand, as of today, with regard to the 3D reflection equation, 
the only solutions known are those constructed in \cite{KO12,KO13,Y21}.
Some of these solutions have originated from intriguing applications of 
the representation theory of quantum groups.
In fact, those in \cite{KV94, S08, KO12, KO13, Y21}
have been constructed systematically by 
invoking the quantized coordinate rings and the PBW basis of a nilpotent subalgebra of quantized 
universal enveloping algebras.  
Further details can be found in the survey monograph \cite{K22}. 

This paper is dedicated to another fascinating approach based on quantum cluster algebras, 
recently proposed by Sun and Yagi for the tetrahedron equation \cite{SY22}.
Quantum cluster algebras were introduced by Fock and Goncharov \cite{FG09}, 
as a $q$-deformation of cluster $y$-variables in the cluster algebras by Fomin and Zelevinsky \cite{FZ07}.
One of the remarkable features of the quantum cluster algebra is the quantum mutation and the quantum dilogarithm 
inherently incorporated within it.
In \cite{SY22}, three types of quivers, namely triangle, square, and butterfly quivers, are introduced, and 
a scheme for producing solutions to the tetrahedron equation is formulated using quantum dilogarithms. 
It poses intriguing challenges because, in practice, 
there exist nontrivial degrees of freedom 
in implementing the noncommuting quantum $y$-variables in a selected quiver, 
upon which the resulting solutions would critically depend.
Given such considerations,  in this paper, we embark on a study of the tetrahedron 
and the 3D reflection equations 
with a focus on the Fock-Goncharov quivers in place of the triangle quivers.
Fock-Goncharov quivers (FG quivers, for short) \cite{FG06} are defined for reduced expressions of elements 
in the Weyl group $W(\mg)$ for a finite dimensional simple Lie algebra $\mg$, and has a beautiful application to the higher Teichm\"uller theory introduced in \cite{FG03}. 

The first step in the cluster algebra approach is to find a sequence of mutations corresponding to the Coxeter moves 
of reduced expressions of the longest element of  $W(\mg)$.
This is done for the transformations $R_{ijk}$ and $K_{ijkl}$ with the choices $\mg =A_2$ and $C_2$, respectively.
The tetrahedron and the 3D reflection equations are associated to $\mg =A_3$ and $C_3$. 
It is known that quantum mutation is decomposed into two parts: 
the monomial part and the automorphism part \cite{FG09} in two ways \cite{Ke11},  
and, in general, only the latter part admits a description as an adjoint action by quantum dilogarithms.
The situation is similar also in the cluster transformations, i.e., those for the quantum $y$-variables
 induced by the mutation sequences corresponding to $R_{ijk}$ 
 and $K_{ijkl}$.
The essential point of our approach lies in realizing these cluster transformations entirely as adjoint actions,
including the monomial part as $\mathrm{Ad}(\mathcal{R}_{ijk})$ and $\mathrm{Ad}(\mathcal{K}_{ijkl})$.
This is achieved through an elaborate embedding of the noncommuting algebra of quantum 
 $y$-variables into the $q$-Weyl algebra.
 See (\ref{eq:phi_kappa}), (\ref{eq:phi-K}), Propositions \ref{prop:R-ad} and \ref{prop:K-ad}.
 Here $\mathcal{R}_{ijk}=\mathcal{R}(\lambda_i, \lambda_j,\lambda_k)_{ijk}$ and 
 $\mathcal{K}_{ijkl}= \mathcal{K}(\lambda_i, \lambda_j,\lambda_k,\lambda_l)_{ijkl}$ are 
 dependent on the spectral parameters $\lambda_i$'s, and are 
 expressed in terms of the quantum dilogarithm  $\Psi_q(z)$ 
 (\ref{eq:Psiq}) as follows (see \S \ref{subsec:qWeyl-R} and 
 \S \ref{sb:cK} for notations):
\begin{align*}
\mathcal{R}_{ijk} &= \Psi_q( \e^{p_i+u_i+p_k-u_k-p_j+\lambda_{ik}}) 
\rho_{jk}\, e^{\frac{1}{\hbar}p_i(u_k-u_j)} 
e^{\frac{\lambda_{jk}}{\hbar}(u_k-u_i)},
\\
\mathcal{K}_{ijkl} &= \Psi_{q^2}(e^{p_j+u_j+p_l-u_l-2p_k+ \lambda_{jl}}) 
\Psi_q(e^{p_i+u_i+p_k-u_k-p_j+ \lambda_{ik}}) 
\Psi_{q^2}(e^{p_j+u_j+p_l-u_l-2p_k+ \lambda_{jl}})^{-1}  
\\
& \qquad   \times  \rho_{jl} \,e^{\frac{1}{\hbar}p_i(u_l-u_j)}
e^{\frac{\lambda_{jl}}{2\hbar}(2u_k-2u_i+u_l-u_j)}.
\end{align*}
Our main results are Theorem \ref{thm:R-tetra} and \ref{thm:K-reflection}
stating that they satisfy the tetrahedron and the 3D reflection equations
involving the spectral parameters.
The fact that the objects $\mathcal{R}_{ijk}$ and  $\mathcal{K}_{ijkl}$ that were within $\mathrm{Ad}(\;\,)$ 
yield the solutions by themselves is quite non-trivial, and this constitutes an essential and original aspect of 
our work building upon \cite{SY22}.
A similar result related to the square quiver will be presented 
in a separate paper \cite{IKT23}\footnote{The conventions used in that paper and in this one are slightly different.}.

\bigskip

This paper is organized as follows. 
In \S 2, we review basic facts on quantum cluster algebra used in this paper. 
In \S 3, we explain a relation between wiring diagrams and the FG quivers, 
and introduce sequences of mutations corresponding to 
the transformations $R_{ijk}$ and $K_{ijkl}$ of the wiring diagrams. 
Further, the associated cluster transformations are realized as the adjoint action 
of the operators $\mathcal{R}_{ijk}$ and $\mathcal{K}_{ijkl}$
by embedding the noncommuting torus algebra of $y$-variables into $q$-Weyl algebra.
In \S 4 and \S 5, we present our main results 
that these operators yield a new solution to the tetrahedron equation and the 3D reflection equation.
Our explanation of the origin of the spectral parameters can be found in Remark \ref{re:sp}.
We study the representations of the operators $\mathcal{R}_{ijk}$ and $\mathcal{K}_{ijkl}$
 in an infinite dimensional vector space in \S 6, and in the modular double setting in \S 7.  
Our $\mathcal{R}_{ijk}$ satisfies the tetrahedron equation with spectral parameters 
as in Theorem \ref{th:teu} and \ref{th:teR}. 
Although they follow from Theorem \ref{thm:R-tetra}, 
we provide an independent proof based on direct calculations at the level of matrix elements.
It reveals an intriguing duality. See Remark \ref{re:dual}.  
In appendix \ref{s:ncq}, we recall the definition and useful formulas for 
the noncompact quantum dilogarithm used in \S 7.

\subsection*{Acknowledgement}
The authors thank Junya Yagi and Akihito Yoneyama for discussions.
RI is supported by  JSPS KAKENHI Grant Number 19K03440 and 23K03048.   
YT is supported by  JSPS KAKENHI Grant Number 21K03240 and 22H01117.

%%%%%%%%%%%%%%%%%%%%%%%%%%%%%%
\section{Basic facts on quantum cluster algebras}

\subsection{Quantum cluster mutation}\label{subsec:mutation}

We recall the definition of quantum cluster mutation by Fock and Goncharov \cite{FG06}. 

For a finite set $I$, let $B = (b_{ij})_{i,j \in I}$ be a skew-symmetrizable matrix with values in $\frac{1}{2}\Z$: there exists an integral diagonal matrix $d =\mathrm{diag}(d_j)_{j \in I}$ such that $\widehat{B} = (\widehat{b}_{ij})_{i,j \in I}:= B\,d = (b_{ij} d_j)_{i,j \in I}$ is skew-symmetric. We assume that $\gcd(d_j \mid j \in I)=1$.  Define a subset $I_0$ of $I$ by $I_0 := \{i \in I; ~b_{ij} \notin \Z \text{ or } b_{ji} \notin \Z \}$. We call $B$ the {\it exchange matrix}.
Let $\mathcal{Y}(B)$ be a skew field generated by $q$-commuting variables $Y = (Y_i)_{i \in I}$ with the relations 
\begin{align}\label{eq:q-Y}
  Y_i Y_j = q^{2 \widehat{b}_{ij}} Y_j Y_i.  
\end{align}
We write $q_i$ for $q^{d_i}$.
We call the data $(B, d, Y)$ a quantum $y$-seed, and $Y_i$ a quantum $y$-variable.    

For $(B, d, Y)$ and $k \in I \setminus I_0$, the quantum mutation $\mu_k$ transforms $(B, d, Y)$ to $(B', d',Y') := \mu_k (B, d, Y)$ as 
\begin{align}\label{eq:q-mutation}
  &b_{ij}' = 
  \begin{cases}
    -b_{ij} & i=k \text{ or } j=k,
    \\
    \displaystyle{b_{ij} + \frac{|b_{ik}| b_{kj} + b_{ik} |b_{kj}|}{2}}
    & \text{otherwise},
  \end{cases}
  \\ \label{eq:d-mutation}
  &d_i' = d_i,
  \\ \label{eq:X-mutation}
  &Y_{i}' = 
  \begin{cases}
  Y_k^{-1} & i=k,
  \\
  \displaystyle{Y_i \prod_{j=1}^{|b_{ik}|}(1 + q_k^{2j-1} Y_k^{-\mathrm{sgn}(b_{ik})})^{-\mathrm{sgn}(b_{ik})}} & i \neq k.
  \end{cases}
\end{align}
The mutations are involutive, $\mu_k \mu_k = \mathrm{id.}$, and commutative, $\mu_k \mu_j = \mu_j \mu_k$ if $b_{jk}=b_{kj}=0$.
We may abbreviate $(B,d,Y)$ to $(B,Y)$, since $d$ is uniquely determined from $B$ and invariant under mutations.
The mutation $\mu_k$ induces an isomorphism of skew fields $\mu_k^\ast: \mathcal{Y}(B') \to \mathcal{Y}(B)$, where $\mathcal{Y}(B')$ is a skew field generated by $Y'_i$ with relations $Y'_i Y'_j = q^{2 \widehat{b}'_{ij}} Y'_j Y'_i$. 

The quantum mutation is decomposed into two parts, a monomial part and an automorphism part \cite{FG09}, in two ways \cite{Ke11}. 
For $\ve \in \{+,-\}$, we define an isomorphism $\tau_{k,\ve}$ of the skew fields by 
\begin{align}\label{eq:mono-iso}
\tau_{k,\ve} : ~\mathcal{Y}(B')  \to \mathcal{Y}(B)
; ~ Y'_i \mapsto 
\begin{cases} 
  Y_k^{-1} & i= k, 
  \\ 
  q^{-b_{ik}[\ve b_{ik}]_+}Y_i Y_k^{[\ve b_{ik}]_+} & i \neq k,
\end{cases}
\end{align} 
with $[a]_+ := \max[0,a]$, and adjoint actions $\mathrm{Ad}_{k,\ve}$ on $\mathcal{Y}(B)$ by 
% := \mathrm{Ad} (\Psi_{q_k}(Y_k'))$ and $\mathrm{Ad}_{k,-}:= \mathrm{Ad}(\Psi_{q_k}(Y_k^{'-1})^{-1})$ by
\begin{align}
&\mathrm{Ad}_{k,+}(Y_i) := \mathrm{Ad}(\Psi_{q_k}(Y_k))(Y_i) =  \Psi_{q_k}(Y_k) Y_i \Psi_{q_k}(Y_k)^{-1},
\\
&\mathrm{Ad}_{k,-}(Y_i) := \mathrm{Ad}(\Psi_{q_k}(Y_k^{-1})^{-1})(Y_i) = \Psi_{q_k}(Y_k^{-1})^{-1} Y_i \Psi_{q_k}(Y_k^{-1}).
\end{align}
Here, $\Psi_{q}(U)$ denotes the quantum dilogarithm defined as
\begin{align}\label{eq:Psiq}
\Psi_q(U) = \frac{1}{(-qU; q^2)_\infty}, \quad  (z;q)_\infty = \prod_{k=0}^{\infty}(1-zq^k).
\end{align}
For $n \in \Z$ we define $(z;q)_n$ by
\begin{align}\label{eq:qq_n}
(z;q)_n = \frac{(z;q)_\infty}{(zq^n;q)_\infty},
\end{align}
and express the expansion of $\Psi_{q}(U)$ as
\begin{align}\label{eq:q-dilog-sum}
\Psi_q(U) = \sum_{n = 0}^\infty \frac{(-qU)^n}{(q^2;q^2)_n},
\qquad 
\Psi_q(U)^{-1} = \sum_{n = 0}^\infty \frac{q^{n^2} U^n}{(q^2;q^2)_n}.
\end{align}
The fundamental properties of $\Psi_q(U)$ are 
%Its fundamental property is the pentagon identity 
\begin{align}
\label{eq:reculsion}
&\Psi_q(q^2 U) \Psi_q(U)^{-1} = 1+qU,
\\
\label{eq:pentagon}
&\Psi_q(U)\Psi_q(W) = \Psi_q(W) \Psi_q(q^{-1}UW) \Psi_q(U) ~~\text{ if } UW = q^2WU,
\end{align}
where the second one is called the pentagon identity.

\begin{lem}[\cite{FG09,Ke11}]
For $k \in I \setminus I_0$ and $i \in I$ it holds that 
$$
\mathrm{Ad}_{k,+}\circ \tau_{k,+}(Y'_i)
= \mathrm{Ad}_{k,-} \circ \tau_{k,-}(Y'_i),
%\mathrm{Ad}(\Psi_{q_k}(Y_k)) \circ \tau_{k,+}(Y'_i)
%= \mathrm{Ad}(\Psi_{q_k}(Y_k^{-1})^{-1})\circ \tau_{k,-}(Y'_i),
$$
and the following diagram is commutative for $\ve = +, -$:
\begin{align*}
\xymatrix{
\mathcal{Y}(B') \ar[r]^{\mu_k^\ast} \ar[dr]_{\tau_{k,\ve}} & \mathcal{Y}(B) 
\\
& \mathcal{Y}(B) \ar[u]_{\mathrm{Ad}_{k,\ve}} 
}
\end{align*}
\end{lem}

We introduce the quantum torus algebra $\mathcal{T}(B)$ associated to $B$ for later use. This is the $\mathbb{Q}(q)$-algebra generated by noncommutative variables $\rY^\alpha~(\alpha \in \Z^{I})$ with the relations
\begin{align}
q^{\langle \alpha,\beta \rangle} \rY^\alpha \rY^\beta 
= \rY^{\alpha + \beta},
\end{align}
where $\langle ~~,~~ \rangle$ is a skewsymmetric form given by 
$\langle \alpha,\beta \rangle = - \langle \beta,\alpha \rangle =  \alpha \cdot \widehat{B} \beta$.
Let $e_i$ be the standard unit vector of $\Z^I$, and write $\rY_i$ for $\rY^{e_i}$. Then we have $\rY_i \rY_j = q^{2 \widehat{b}_{ij}} \rY_j \rY_i$.
We identify $\rY_i$ with $Y_i$, and obtain \eqref{eq:q-Y}.
The monomial part \eqref{eq:mono-iso} of $\mu_k^\ast$ naturally induces the morphisms of the quantum torus algebras, which is written as
\begin{align}\label{eq:torus-iso} 
\tau_{k,\ve} : \mathcal{T}(B') \to \mathcal{T}(B); ~
\rY'_i \mapsto \begin{cases} 
\rY_k^{-1} & i= k, 
\\ 
\rY^{e_i + e_k [\ve b_{ik}]_+} & i \neq k.
\end{cases}
\end{align}
Following \cite{Ke11} we consider the subalgebra $\mathbb{A}(B)$ of $\mathcal{T}(B)$ generated by $\rY^\alpha~(\alpha \in \Z_{\geq 0}^{I})$, and the completion $\hat{\mathbb{A}}(B)$ of $\mathbb{A}(B)$ with respect to the ideal generated by $\rY_i ~(i \in I)$\footnote{In \cite{Ke11}, the algebras $\mathbb{A}(B)$ and $\hat{\mathbb{A}}(B)$ are respectively called the quantum affine space and the formal quantum affine space.}. 
 Further, let $\hat{\mathbb{A}}^{\! \times}(B)$ be a subset of $\hat{\mathbb{A}}(B)$ generated by all invertible elements in $\hat{\mathbb{A}}(B)$.
Note that the quantum dilogarithm $\Psi_q(\rY^\alpha)$ with $\alpha \in \Z_{\geq 0}^{I}$ belongs to $\hat{\mathbb{A}}^{\! \times}(B)$.
 
\subsection{Weighted quiver}
For the pair $(B,d)$, we define the skew-symmetric matrix $\sigma = (\sigma_{ij})_{i,j \in I}$ by $\sigma_{ij} = b_{ij} \gcd(d_i,d_j) / d_i$. In this paper, we only have the cases that $\sigma_{ij}$ is integral or $\pm 1/2$.  
We determine the weighted quiver $Q =(\sigma,d)$ without one-loop or two cycle as follows. The vertex set of $Q$ is $I$, where each vertex $i \in I$ has a weight $d_i$. When $\sigma_{ij}$ is integral, we draw ordinary arrows $\longrightarrow$ in such a way that $\sigma_{ij} = \#\{\text{arrows from $i$ to $j$}\}  - \#\{\text{arrows from $j$ to $i$}\}$ is satisfied. When $\sigma_{ij} = 1/2$ (resp. $\sigma_{ij} = -1/2$), we draw a dashed arrow $\dashrightarrow$ from $i$ to $j$ (resp. $j$ to $i$). 

The following lemma on the rule of the quiver mutation is useful (cf. \cite{IIO21}):

\begin{lem}
The mutation of the weighted quiver $Q = (\sigma,d)$ by $\mu_k$, $(\sigma',d') := \mu_k(\sigma,d)$, is given by
\eqref{eq:d-mutation} and 
$$
  \sigma_{ij}' = 
  \begin{cases}
    -\sigma_{ij} & i=k \text{ or } j=k,
    \\
    \displaystyle{\sigma_{ij} 
    + \frac{|\sigma_{ik}| \sigma_{kj} + \sigma_{ik} |\sigma_{kj}|}{2} 
      \,\alpha_{ij}^k }
    & \text{otherwise},
  \end{cases} 
$$
where 
$$
  \alpha_{ij}^k 
  = 
  d_k \frac{\gcd(d_i, d_j)}{\gcd(d_k, d_i)\gcd(d_k, d_j)}.
$$
In particular, if $d_k \in \{d_i,d_j\}$, then we have $\alpha_{ij}^k = 1$. 
If $d_i = d_j$ and $\gcd(d_k,d_i)=1$, then we have $\alpha_{ij}^k = d_k d_i$.
\end{lem}
\begin{example}
In Figure~\ref{fig:mutation-ex}, we denote a vertex of weight one by a circle, and a vertex of weight two by 
$\small{\raise0.2ex\hbox{\textcircled{\scriptsize{2}}}}$. 
By applying $\mu_3$ to the left quiver, we get the right one, 
where the arrow from $1$ to $2$ appears due to $\alpha^1_{2,2} = 2$. 
If all the three vertices have the same weight, this arrow does not appear.

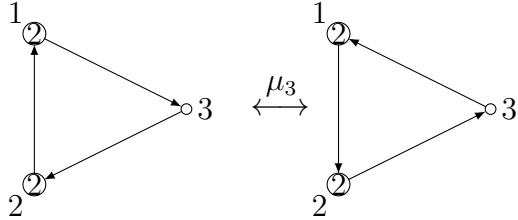
\begin{figure}[ht]
\[
\begin{tikzpicture}
\begin{scope}[>=latex]
\path (0,1) node[circle]{2} coordinate(C1) node[above left]{$1$};
\draw (0,1) circle[radius=0.15];
\path (0,-1) node[circle]{2} coordinate(C2) node[below left]{$2$};
\draw (0,-1) circle[radius=0.15];
\draw(2,0) circle(2pt) coordinate (C3) node[right]{$3$};

\draw[->,shorten >=4pt,shorten <=4pt] (C2)--(C1);
\draw[->,shorten >=2pt,shorten <=4pt] (C1)--(C3);
\draw[->,shorten >=4pt,shorten <=2pt] (C3)--(C2);

\draw(2.7,0) node[right]{$\longleftrightarrow$};
\draw(2.9,0.3) node[right]{$\mu_3$};

\path (4,1) node[circle]{2} coordinate(D1) node[above left]{$1$};
\draw (4,1) circle[radius=0.15];
\path (4,-1) node[circle]{2} coordinate(D2) node[below left]{$2$};
\draw (4,-1) circle[radius=0.15];
\draw(6,0) circle(2pt) coordinate (D3) node[right]{$3$};

\draw[->,shorten >=4pt,shorten <=4pt] (D1)--(D2);
\draw[->,shorten >=4pt,shorten <=2pt] (D3)--(D1);
\draw[->,shorten >=2pt,shorten <=4pt] (D2)--(D3);

\end{scope}
\end{tikzpicture}
\]
\caption{Mutation of weighted quiver}
\label{fig:mutation-ex}
\end{figure}

\end{example}

\subsection{Tropical $y$-variables and tropical sign}

Let $\mathbb{P}(u)=\mathbb{P}_\trop(u_1,u_2,\ldots,u_p) := \{\prod_{i=1}^{p} u_i^{a_i}; ~a_i \in \Z \}$ be the tropical semifield of rank $p$, equipped with the addition $\oplus$ and multiplication $\cdot$ as
$$
  \prod_{i=1}^{p} u_i^{a_i} \oplus \prod_{i=1}^{p} u_i^{b_i}
  = 
  \prod_{i=1}^{p} u_i^{\min(a_i,b_i)}, 
  \qquad
  \prod_{i=1}^{p} u_i^{a_i} \cdot \prod_{i=1}^{p} u_i^{b_i}
  =  
  \prod_{i=1}^{p} u_i^{a_i+b_i}.
$$
For $v = \prod_{i \in I} u_i^{a_i} \in \mathbb{P}(u)$, we write $v = u^\alpha$
with $\alpha = (a_i)_{i \in I} \in \Z^{I}$.
If $\alpha \in \Z_{\geq 0}^{I}$ (resp. $\alpha \in \Z_{\leq 0}^{I}$), we say $v$ is positive (resp. negative). 

For a quiver $Q$ with a vertex set $I$, let $\mathbb{P}(u)$ be a tropical semifield of rank $|I|$.
For a tropical $y$-seed $(B,d,y); y = (y_i)_{i \in I} \in \mathbb{P}(u)^{I}$, and $k \in I \setminus I_0$, the mutation $\mu_k (B,d, y) =: (B',d',y')$ is given by \eqref{eq:q-mutation}, \eqref{eq:d-mutation} and 
\begin{equation}\label{eq:trop-mutation}
y_{i}' = 
  \begin{cases}
  y_k^{-1} & i=k,
  \\
  y_i \cdot (1 \oplus y_k^{-\mathrm{sgn}(b_{ik})})^{-b_{ik}} & i \neq k.
  \end{cases}
\end{equation}
For a tropical $y$-variable $y_i' = u^{\alpha'}$, the vector $\alpha'$ is called the {\it $c$-vector} of $y_i'$.
The following theorem states the {\em sign coherence} of the $c$-vectors.

\begin{thm}[\cite{FZ07,GHKK14}]
\label{thm:sign-coherence}
For any sequence of mutations which transforms a tropical $y$-seed $(B,u)$ into $(B',y')$, 
it holds that each $y'_i \in \mathbb{P}(u)$ is either positive or negative.
\end{thm}

Based on this theorem, for any tropical $y$-seed $(B',y')$ obtained from $(B,u)$ by applying mutations, we define the {\it tropical sign} $\ve_i'$ of $y_i'$ to be $+1$ (resp. $-1$) if $y_i'$ is positive (resp. $y_i$ is negative).  We also write $\ve_i'= \pm$ for $\ve_i'=\pm 1$ for simplicity.

\begin{remark}
For the mutation $\mu_k (B,y) = (B',y')$ of $y$-seeds, let $c_i$, $c_i'$ and $c_k$ be the $c$-vectors of $y_i$, $y_i'$ and $y_k$, and $\ve_k$ be the tropical sign of $y_k$. 
Then the  tropical mutation \eqref{eq:trop-mutation} is written in terms of $c$-vectors as 
$$
  c_i' = 
  \begin{cases}
  -c_k & i=k,
  \\
  c_i + c_k [\ve_k b_{ik}]_+ & i \neq k.
\end{cases}
$$
This coincides with the transformation of quantum torus \eqref{eq:torus-iso} on $\Z^I$ (i.e. the power of \eqref{eq:torus-iso}), when $\ve = \ve_k$. 
\end{remark}

\subsection{Periodicity and quantum dilogarithm identity}

The symmetric group $\mathfrak{S}_I$ naturally acts on $y$-seeds.
% by permuting the variables and the rows and columns of the exchange matrix. 
The action of $\sigma \in \mathfrak{S}_I$ is given by
$$
  \sigma : (b_{ij},d_i,y_i) \mapsto (b_{\sigma^{-1}(i), \sigma^{-1}(j)}, d_{\sigma^{-1}(i)}, y_{\sigma^{-1}(i)}). 
$$
For a sequence $\mathbf{i} = (i_1, i_2, \ldots,i_L) \in I^L$, define a sequence of mutations $\mu_{\mathbf{i}}:= \mu_{i_L} \mu_{i_{L-1}} \cdots \mu_{i_2} \mu_{i_1}$, and consider the sequences of tropical $y$-seeds starting with $(B,u)$ and of quantum $y$-seeds starting with $(B,Y)$:
\begin{align}
\label{eq:seq-tropy}
&(B,u) =:(B(1),y(1)) \stackrel{\mu_{i_1}}{\longleftrightarrow}
(B(2),y(2)) \stackrel{\mu_{i_2}}{\longleftrightarrow}
\cdots
\stackrel{\mu_{i_L}}{\longleftrightarrow} (B(L+1),y(L+1)),
\\
\label{eq:seq-Y}
&(B,Y) =:(B(1),Y(1)) \stackrel{\mu_{i_1}}{\longleftrightarrow}
(B(2),Y(2)) \stackrel{\mu_{i_2}}{\longleftrightarrow}
\cdots
\stackrel{\mu_{i_L}}{\longleftrightarrow} (B(L+1),Y(L+1)).
\end{align}
For an element $\sigma \in \mathfrak{S}_I$, we say that the sequence $\mathbf{i}$ is a $\sigma$-period of $(B,u)$ if it holds that
$$
  b_{\sigma(i),\sigma(j)}(L+1) = b_{ij}(1), \qquad y_{\sigma(i)}(L+1) = y_i(1).
$$  
The $\sigma$-period of $(B,Y)$ is defined in the same manner.
For an exchange matrix $B$, a mutation sequence is a sequence of mutations for $B$ and permutations of $I$.
  
The following theorem is obtained by combining the synchronicity \cite{N21} among $x$-seeds, $y$-seeds and tropical $y$-seeds, and the synchronicity between classical and quantum seeds \cite[Lemma 2.22]{FG09b}, \cite[Proposition 3.4]{KN11}.
 
\begin{thm}\label{thm:period}
For an exchange matrix $B$ and a mutation sequence $\nu$ for $B$, the following two statements are equivalent:
\begin{itemize}

\item[(1)]
For a tropical $y$-seed $(B,y)$, it holds that $\nu (B,y) = (B,y)$. 

\item[(2)]
For a quantum $y$-seed $(B,Y)$, it holds that $\nu (B,Y) = (B,Y)$. 

\end{itemize}
\end{thm}

Proving 
the periodicity of a tropical $y$-seed is much easier than the quantum one. In this paper we use this theorem in such a way that (2) follows from (1).

For $t=1,\ldots, L+1$, let $\rY^{e_i}(t)(= \rY_i(t)) ~(i \in I)$ be the generators of the quantum torus $\mathcal{T}(B(t))$.
The quantum $y$-variables in \eqref{eq:seq-Y} are expressed as 
\begin{align}\label{eq:ad-tau-decomp}
\begin{split}
Y_i(t+1) 
&= \Ad(\Psi_{q_{i_1}}(\rY_{i_1}(1)^{\delta_1})^{\delta_1}) \tau_{i_1,\delta_1}\cdots \Ad(\Psi_{q_{i_{t}}}(\rY_{i_{t}}(t)^{\delta_{t}})^{\delta_{t}})  \tau_{i_{t},\delta_{t}}(\rY_i(t+1))
\\
&=\Ad(\Psi_{q_{i_1}}(\rY^{\delta_1\beta_1})^{\delta_1}) \cdots 
\Ad(\Psi_{q_{i_{t}}}(\rY^{\delta_{t} \beta_{t}})^{\delta_{t}}) \tau_{i_1,\delta_1}\cdots \tau_{i_{t},\delta_{t}}(\rY_i(t+1)),
\end{split}
\end{align}
where $\beta_r \in \Z^I$ is determined by $\rY^{\beta_r} = \tau_{i_1,\delta_1}\cdots \tau_{i_{r-1},\delta_{r-1}}(\rY_{i_r}(r))$. Especially we have $\beta_1 = e_{i_1}$.
Suppose that $\mathbf{i}=(i_1, i_2, \ldots,i_L)$ is a $\sigma$-period of $(B,Y)$. For any sign sequence $(\delta_t)_{t=1,\ldots, L}$ of $\delta_t \in \{ +, -\}$, we have 
\begin{align}
\begin{split}
\Ad(\Psi_{q_{i_1}}(\rY^{\delta_1 \beta_1})^{\delta_1}) \Ad(\Psi_{q_{i_2}}(\rY^{\delta_2 \beta_2})^{\delta_2}) \cdots & \Ad(\Psi_{q_{i_L}}(\rY^{\delta_L \beta_L})^{\delta_L}) 
\\
&\circ \tau_{i_1,\delta_1} \tau_{i_1,\delta_1} \cdots \tau_{i_L,\delta_L} \sigma = \mathrm{id}.
\end{split}
\end{align} 

Let $\alpha_t \in \Z^{I}$ be the $c$-vector of $y_{i_t}(t)$ in the sequence \eqref{eq:seq-tropy}. (Of course $\alpha_1 = e_{i_1}$.) From Theorem \ref{thm:sign-coherence}, for each $t = 1,2,\ldots,L$ the vector $\alpha_t$ is
either in $\Z_{\geq 0}^{I}$ or in $\Z_{\leq 0}^{I}$.
We write $\ve_{t}$ for the tropical sign of $y_{i_t}(t)$ in \eqref{eq:seq-tropy}, and call $(\ve_1,\ve_2,\ldots,\ve_L)$ the {\it tropical sign-sequence} of \eqref{eq:seq-tropy}. 
The following theorem is a generalization of \cite[Theorem 3.5]{KN11} for skew-symmetric exchange matrices to that for skew-symmetrizable ones. (We remark that \cite[Theorem 3.5]{KN11} is due to \cite[Theorem 5.16]{Ke11}. See also \cite{R09}.)

\begin{thm}\label{thm:id-mono-qdilog}
Suppose $\mathbf{i}=(i_1, i_2, \ldots,i_L)$ is a $\sigma$-period for $(B,Y)$,
and let $(\ve_1,\ldots,\ve_L)$ be the tropical sign-sequence of \eqref{eq:seq-tropy}. Then the following identities hold, 
\begin{align}
\label{eq:tau-id}
&\tau_{i_1,\ve_1} \tau_{i_1,\ve_1} \cdots \tau_{i_L,\ve_L} \sigma = \mathrm{id},\end{align}
as morphism of $\mathcal{Y}(B)$,  
and 
\begin{align}
\label{eq:dilog-id}
&\Psi_{q_{i_1}}(\rY^{\ve_1 \alpha_1})^{\ve_1} \Psi_{q_{i_2}}(\rY^{\ve_2 \alpha_2})^{\ve_2} \cdots \Psi_{q_{i_L}}(\rY^{\ve_L \alpha_L})^{\ve_L} = 1. 
\end{align}
in $\mathbb{A}^{\! \times}(B)$.
\end{thm}

\begin{proof}
Eq.~\eqref{eq:tau-id} follows from Theorem \ref{thm:period}.
Eq.~\eqref{eq:dilog-id} is proved in the same manner as \cite[Theorem 3.5]{KN11}.
\end{proof}

%%%%%%%%%%%%%%%%%%%%%%%%%%%%%%%%%%%%%%%%%%%%%%
\section{Wiring diagrams and Fock-Goncharov quivers}
%%%%%%%%%%%%%%%%%%%%%%%%%%%%%%%%%%%%%%%%%%%%%%

Let $\mathfrak{g}$ be a finite dimensional simple Lie algebra of rank $\ell$, 
and $W(\mathfrak{g})$ be the Weyl group for $\mg$ generated by 
the simple reflections $r_s ~(s = 1,2,\ldots,\ell)$. We write $w_0$ for the longest element in $W(\mathfrak{g})$.
A reduced expression $r_{s_1} r_{s_2}\cdots r_{s_p}$ of $w_0$ will be denoted by $s_1 s_2 \cdots s_p$ for simplicity.
This paper discusses cases involving $\mathfrak{g} = A_2, A_3, C_2,$ and $C_3$.
We first consider the quivers associated to the reduced expressions of $w_0$ 
for $W(A_2)$ and $W(C_2)$ introduced by Fock and Goncharov \cite{FG06}. 
(For the detail, see \cite[Section 4]{IIO21} for instance.)

\subsection{Pictorial transformation $R_{123}$ associated with $W(A_2)$}

For $W(A_2)$, we have the two reduced expressions $121$ and $212$ of $w_0$. 
Let $J_{121}$ and $J_{212}$ denote the corresponding Fock-Goncharov quivers (FG quivers, for short).
They are depicted as (\ref{quiver:A2}),
where the weight of all vertices are one, thus $B = \widehat{B} = \sigma$. 
The two quivers are transformed to each other by a single mutation $\mu_4$.
 
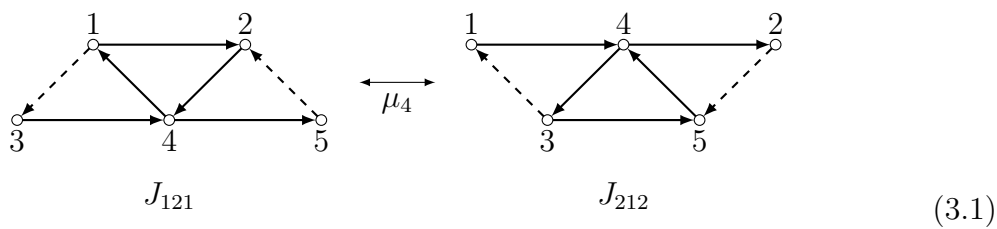
\begin{align}\label{quiver:A2}
\begin{tikzpicture}
\begin{scope}[>=latex,xshift=0pt]
\draw (1,0) circle(2pt) coordinate(B) node[below]{$3$};%{$\bar v^t_1$};
\draw (3,0) circle(2pt) coordinate(C) node[below]{$4$};%{$\bar v^t_2$};
\draw (5,0) circle(2pt) coordinate(F) node[below]{$5$};%{$\bar v^t_3$};
\draw (2,1) circle(2pt) coordinate(D) node[above]{$1$};%{$\bar v^s_1$};
\draw (4,1) circle(2pt) coordinate(E) node[above]{$2$};%{$\bar v^s_2$};
\qarrow{B}{C}
\qarrow{C}{D}
\qarrow{D}{E}
\qarrow{E}{C}
\qarrow{C}{F}
\qdarrow{D}{B}
\qdarrow{F}{E}
%
%{\color{red}
%\fill (2,0) circle(2pt) node[below]{$A$};
%\fill (3,1) circle(2pt) node[above]{$B$};
%\fill (4,0) circle(2pt) node[below]{$C$};
%}
%
\coordinate (P1) at (5.5,0.5);
\coordinate (P2) at (6.5,0.5);
\draw[<->] (P1) -- (P2);
\draw (6,0.5) circle(0pt) node[below]{$\mu_4$};
\draw (3,-1) node{$J_{121}$};  
\end{scope}
\begin{scope}[>=latex,xshift=170pt]
\draw (1,1) circle(2pt) coordinate(B) node[above]{$1$};%{$v^s_1$};
\draw (3,1) circle(2pt) coordinate(C) node[above]{$4$};%{$v^s_2$};
\draw (5,1) circle(2pt) coordinate(D) node[above]{$2$};%{$v^s_3$};
\draw (2,0) circle(2pt) coordinate(E) node[below]{$3$};%{$v^t_1$};
\draw (4,0) circle(2pt) coordinate(F) node[below]{$5$};%{$v^t_2$};
\qarrow{B}{C}
\qarrow{C}{D}
\qarrow{E}{F}
\qarrow{F}{C}
\qarrow{C}{E}
\qdarrow{E}{B}
\qdarrow{D}{F}
%{\color{red}
%\fill (2,1) circle(2pt) node[above]{$C$};
%\fill (3,0) circle(2pt) node[below]{$B$};
%\fill (4,1) circle(2pt) node[above]{$A$};
%}
\draw (3,-1) node{$J_{212}$};  
\end{scope}
\end{tikzpicture}
\end{align}

Elements of $W(A_2)$ are depicted by wiring diagram with three wires:  
$r_s \in W(A_2)$ is denoted by a crossing of the $s$th and the $s+1$st wires from the bottom.
The above quivers are related to the wiring diagrams \cite{BFZ96} in the following way: 
the quivers are `dual' to the wiring diagrams, where the vertices of the quivers correspond to the domains of the diagrams. 
We let $R_{123}$ denote the transformation of the wiring diagrams as follows:

\begin{align}
\label{quiver-d:A2}
\begin{tikzpicture}
\begin{scope}[>=latex,xshift=0pt]
{\color{red}
\fill (1,0.5) circle(2pt) coordinate(A) node[below]{$1$};
\fill (2,1.5) circle(2pt) coordinate(B) node[above]{$2$};
\fill (3,0.5) circle(2pt) coordinate(C) node[below]{$3$};
\draw [-] (0,2) node[left=2pt] {\footnotesize $3$} to [out = 0, in = 135] (B);
\draw [-] (B) -- (C); 
\draw [->] (C) to [out = -45, in = 180] (4,0);
\draw [-] (0,1) node[left=2pt] {\footnotesize $2$} to [out = 0, in = 135] (A); 
\draw [-] (A) to [out = -45, in = -135] (C);
\draw [->] (C) to [out = 45, in = 180] (4,1);
\draw [-] (0,0) node[left=2pt] {\footnotesize $1$} to [out = 0, in = -135] (A); 
\draw [-] (A) -- (B);
\draw [->] (B) to [out = 45, in = 180] (4,2);
}
\coordinate (P1) at (4.5,1);
\coordinate (P2) at (5.5,1);
\draw[->] (4.5,1) -- (5.5,1);
\draw (5,1) circle(0pt) node[below]{$R_{123}$};
%
%
%{\color{blue}
\draw (0,0.5) circle(2pt) coordinate(B1) node[below]{$3$};
\draw (2,0.5) circle(2pt) coordinate(C1) node[below]{$4$};
\draw (4,0.5) circle(2pt) coordinate(F1) node[below]{$5$};
\draw (1,1.5) circle(2pt) coordinate(D1) node[above]{$1$};
\draw (3,1.5) circle(2pt) coordinate(E1) node[above]{$2$};
\qarrow{B1}{C1}
\qarrow{C1}{D1}
\qarrow{D1}{E1}
\qarrow{E1}{C1}
\qarrow{C1}{F1}
\qdarrow{D1}{B1}
\qdarrow{F1}{E1}
%}
\end{scope}
\begin{scope}[>=latex,xshift=175pt]
{\color{red}
\fill (3,1.5) circle(2pt) coordinate(A) node[above]{$1$};
\fill (2,0.5) circle(2pt) coordinate(B) node[below]{$2$};
\fill (1,1.5) circle(2pt) coordinate(C) node[above]{$3$};
\draw [-] (0,0) node[left=2pt] {\footnotesize $1$} to [out = 0, in = -135] (B);
\draw [-] (B) -- (A); 
\draw [->] (A) to [out = 45, in = 180] (4,2);
\draw [-] (0,1) node[left=2pt] {\footnotesize $2$} to [out = 0, in = -135] (C); 
\draw [-] (C) to [out = 45, in = 135] (A);
\draw [->] (A) to [out = -45, in = 180] (4,1);
\draw [-] (0,2) node[left=2pt] {\footnotesize $3$} to [out = 0, in = 135] (C); 
\draw [-] (C) -- (B);
\draw [->] (B) to [out = -45, in = 180] (4,0);
}
%
%
%{\color{blue}
\draw (0,1.5) circle(2pt) coordinate(B1) node[above]{$1$};%{$v^s_1$};
\draw (2,1.5) circle(2pt) coordinate(C1) node[above]{$4$};%{$v^s_2$};
\draw (4,1.5) circle(2pt) coordinate(D1) node[above]{$2$};%{$v^s_3$};
\draw (1,0.5) circle(2pt) coordinate(E1) node[below]{$3$};%{$v^t_1$};
\draw (3,0.5) circle(2pt) coordinate(F1) node[below]{$5$};%{$v^t_2$};
\qarrow{B1}{C1}
\qarrow{C1}{D1}
\qarrow{E1}{F1}
\qarrow{F1}{C1}
\qarrow{C1}{E1}
\qdarrow{E1}{B1}
\qdarrow{D1}{F1}
%}
\end{scope}
\end{tikzpicture}
\end{align}

\subsection{Cluster transformation for $R_{123}$}
\label{subsec:qWeyl-R}

Corresponding to the transformation \eqref{quiver:A2} of the quivers, the quantum $y$-variables are transformed by $\mu_4^{\ast} = \Ad (\Psi_q(Y_4)) \circ \tau_{4,+}$ as 
\begin{align}\label{eq:R-decomp}
\mu_4^{\ast}:
\begin{pmatrix}
Y_1' \\ Y_2' \\ Y_3' \\ Y_4' \\ Y_5'
\end{pmatrix}
\stackrel{\tau_{4,+}}{\mapsto} 
\begin{pmatrix}
Y_1 \\ q^{-1} Y_2 Y_4 \\ q^{-1} Y_3 Y_4 \\ Y_4^{-1} \\ Y_5
\end{pmatrix}
\stackrel{\Ad(\Psi_q(Y_4))}{\mapsto}
\begin{pmatrix}
Y_1(1+q Y_4) \\ Y_2(1+q Y_4^{-1})^{-1} \\ Y_3(1+q Y_4^{-1})^{-1} \\ Y_4^{-1} \\ Y_5(1+q Y_4)
\end{pmatrix}.   
\end{align}

Let $p_i, u_i ~(i=1,2,3)$ be noncommuting variables satisfying
\begin{align}\label{pu}
[p_i,u_j] = \hbar \delta_{ij}, \qquad  
[p_i,p_j] = [u_i,u_j] = 0.
\end{align} 
We assign a pair of $(p_i,u_i)$ to the crossing $i$ of the wiring diagram \eqref{quiver-d:A2}.
Let $\mathcal{W}(A_2)$ be the algebra over $\C$ generated by the 
$q$-Weyl pairs $e^{\pm p_i}, e^{\pm u_i} ~(i=1,2,3)$ which 
satisfy the $q$-commutativity $e^{p_i} e^{u_j} = q^{\delta_{ij}} e^{u_i} e^{p_j}$.
Here and in what follows, we use the notation and the non-zero complex parameters 
$\kappa_i$ given as
\begin{align}\label{para}
q=e^\hbar, \quad 
\kappa_j = e^{\lambda_j}, \quad \lambda_{jk}= \lambda_j-\lambda_k.
\end{align}
Let Frac$\mathcal{W}(A_2)$ be the noncommuting fractional field of $\mathcal{W}(A_2)$.   
We write $Y_i$ and $Y'_i$ for the generators of $\mathcal{Y}(J_{121})$ and $\mathcal{Y}(J_{212})$ respectively.
We define the embeddings $\phi: \mathcal{Y}(J_{121}) \hookrightarrow \mathrm{Frac}\mathcal{W}(A_2)$ and 
$\phi': \mathcal{Y}(J_{212}) \hookrightarrow \mathrm{Frac}\mathcal{W}(A_2)$ by
\begin{align}\label{eq:phi_kappa}
\phi:
\begin{cases}
Y_1 \mapsto \kappa_2^{-1}e^{p_2-u_2-p_1},
\\
Y_2 \mapsto \kappa_2 e^{p_2+u_2-p_3},
\\
Y_3 \mapsto \kappa_1^{-1}e^{p_1-u_1},
\\
Y_4 \mapsto \kappa_1\kappa_3^{-1} e^{p_1+u_1+p_3-u_3-p_2},
\\
Y_5 \mapsto \kappa_3 e^{p_3+u_3},
\end{cases}
\qquad 
\phi':
\begin{cases}
Y'_1 \mapsto \kappa_3^{-1}e^{p_3-u_3},
\\
Y'_2 \mapsto \kappa_1 e^{p_1+u_1},
\\
Y'_3 \mapsto \kappa_2^{-1} e^{p_2-u_2-p_3},
\\
Y'_4 \mapsto \kappa_1^{-1}\kappa_3 e^{p_3+u_3+p_1-u_1-p_2},
\\
Y'_5 \mapsto \kappa_2 e^{p_2+u_2-p_1}.
\end{cases}
\end{align}
We also define an isomorphism $\pi_{123}$ of $\mathcal{W}(A_2)$ by
\begin{align}\label{eq:R-pi}
\pi_{123}:
\begin{cases}
p_1 \mapsto p_1+  \lambda_{23},
\quad
p_2  \mapsto p_1+p_3,
\quad
p_3 \mapsto p_2-p_1- \lambda_{23}, 
\\
u_1  \mapsto u_1+u_2-u_3,
\quad 
u_2 \mapsto u_3,
\quad 
u_3  \mapsto u_2,
\end{cases}
\end{align} 
in the sense of exponentials.
It induces an isomorphism of Frac$\mathcal{W}(A_2)$. 
  
\begin{prop}\label{prop:R-ad}
(i) $\phi$ and $\phi'$ are morphisms of skew fields.
Moreover, the following diagram is commutative:
\begin{align*}
\xymatrix{
\mathcal{Y}(J_{212}) \ar[r]^{\phi'} \ar[d]_{\tau_{4,+}} & \mathrm{Frac}\mathcal{W}(A_2) \ar[d]_{\pi_{123}}
\\
\mathcal{Y}(J_{121}) \ar[r]^{\phi}& \mathrm{Frac}\mathcal{W}(A_2)
}
\end{align*}
(ii) The isomorphism $\pi_{123}$ is realized by an adjoint action $\Ad(P_{123})$ of 
\begin{align}\label{R-Pop}
P_{123} = \rho_{23}\, e^{\frac{1}{\hbar}p_1(u_3-u_2)} 
e^{\frac{\lambda_{23}}{\hbar}(u_3-u_1)},
\end{align}   
where $\rho_{23} \in \mathfrak{S}_3$ is a permutation of indices $2$ and $3$ of the generators of $\mathcal{W}(A_2)$. (It does not act on the indices of the parameters $\lambda_i, \kappa_i$.)  
\end{prop}

\begin{proof}
(i) Let $B(J_{121})=(b_{ij})_{i,j}$ and $B(J_{212})=(b'_{ij})_{i,j}$ be the exchange matrices.
The first claim is proved by checking 
$\phi(Y_i) \phi(Y_j) = q^{2 b_{ij}} \phi(Y_j) \phi(Y_i)$ and $\phi'(Y_i') \phi'(Y_j') = q^{2 b'_{ij}} \phi'(Y'_j) \phi'(Y'_i)$ by direct calculation. We omit the detail.   
The commutativity of the diagram is proved by checking
$\pi_{123} \circ \phi' (Y_i') = \phi \circ \tau_{4,+}(Y_i')$ for $i=1,\ldots,5$.
We illustrate the cases $i=1,2$. 
\begin{align*}
&\begin{cases}
Y_1' \stackrel{\phi'}{\longmapsto} \kappa_3^{-1} e^{p_3-u_3} \stackrel{\pi_{123}}{\longmapsto}  \kappa_3^{-1} e^{p_2-p_1-\lambda_{23}-u_2} = \kappa_2^{-1} e^{p_2-p_1-u_2},
\\
Y_1' \stackrel{\tau_{4,+}}{\longmapsto} Y_1  \stackrel{\phi}{\longmapsto} \kappa_2^{-1} e^{p_2-u_2-p_1},
\end{cases}
\\
&\begin{cases}
Y_2' \stackrel{\phi'}{\longmapsto} \kappa_1 e^{p_1+u_1} \stackrel{\pi_{123}}{\longmapsto} \kappa_1 e^{p_1
+ \lambda_{23} + u_1 + u_2-u_3}
= \kappa_1 \kappa_2 \kappa_3^{-1} e^{p_1+u_1+u_2-u_3},
\\
Y_2' \stackrel{\tau_{4,+}}{\longmapsto} q^{-1} Y_2 Y_4 \stackrel{\phi}{\longmapsto} q^{-1} \kappa_2 e^{p_2+u_2-p_3} \kappa_1\kappa_3^{-1} e^{p_1+u_1+p_3-u_3-p_2} = \kappa_1 \kappa_2 \kappa_3^{-1} e^{p_1+u_1+u_2-u_3}.
\end{cases}
\end{align*}
(ii) Recall a special case of the Baker-Campbell-Hausdorff formula (BCH formula, for short) as follows.
For noncommutative variables $a$ and $b$, which satisfy $[a,b] = c$ and $[a,c]=[b,c]=0$, 
the equalities $e^{a} e^{b} = e^\frac{c}{2} e^{a +b} = e^c e^b e^a$ hold.
By using this, one can verify the claim by computing the adjoint action of $P_{123}$ on $e^{p_i}$ and $e^{u_i}$, for example,
\begin{align*}
\Ad (P_{123})(e^{p_3}) 
&= \rho_{23}\, e^{\frac{1}{\hbar}p_1(u_3-u_2)} 
\underline{e^{\frac{\lambda_{23}}{\hbar}(u_3-u_1)} e^{p_3} 
e^{-\frac{\lambda_{23}}{\hbar}(u_3-u_1)}} e^{-\frac{1}{\hbar}p_1(u_3-u_2)} 
\rho_{23}
\\
&\stackrel{\mathrm{BCH}}{=} \rho_{23}\, \underline{e^{\frac{1}{\hbar}p_1(u_3-u_2)}} e^{-\lambda_{23}} \underline{e^{p_3} e^{-\frac{1}{\hbar}p_1(u_3-u_2)}} \rho_{23}
\\
&\stackrel{\mathrm{BCH}}{=} \rho_{23}\, e^{-p_1-\lambda_{23}} e^{p_3} \rho_{23} 
=  e^{p_2-p_1-\lambda_{23}}. 
\end{align*}
\end{proof}

A proper framework accommodating the object $P_{123}$ in (\ref{R-Pop})
will be provided in (\ref{N(A3)gen})--(\ref{eq:G-A3}).

\begin{remark}
When $\kappa_i = 1$ hence $\lambda_i = 0$, the transformations $\mu_4^\ast$ \eqref{eq:R-decomp} 
and $\pi_{123}$ \eqref{eq:R-pi} are related to those in the literature. 
The transformation $\pi_{123}$ is essentially 
the same as that of the operation $\hat{P}_{123}$ in \cite[Eq.(12)]{MS97}.
The transformation $\mu_4^\ast$ corresponds to $\mathcal{R}$ 
with $r=1$ and $\lambda=1$ in \cite[Prop.1]{BV15},
and $\pi_{123}$ coincides with $\mathcal{F}$ in \cite[Prop.2]{BV15}. 
\end{remark}

\subsection{Pictorial transformation $K_{1234}$  associated with $W(C_2)$}

In the case of $W(C_2)$, we have the two reduced expressions $1212$ and $2121$ of $w_0$. 
The corresponding FG quivers $J_{1212}$ and $J_{2121}$ are depicted  in the fist line of (\ref{quiver:C2}).
The vertices \maru{2} have weight two, and the others have weight one, hence $d = \mathrm{diag}(2,2,2,1,1,1)$.
For $J_{1212}$ we have 
\begin{align*}
&B = \begin{pmatrix}
0 & 1 & 0 & 1 & -2 & 0 \\
-1 & 0 & 1 & 0 & 2 & -2 \\
0 & -1 & 0 & 0 & 0 & 1\\
-\frac{1}{2} & 0 & 0 & 0 & 1 & 0\\
1 & -1 & 0 & -1 & 0 & 1\\
0 & 1 & -\frac{1}{2} & 0 & -1 & 0
\end{pmatrix},
\quad 
\widehat{B} = \begin{pmatrix}
0 & 2 & 0 & 1 & -2 & 0 \\
-2 & 0 & 2 & 0 & 2 & -2 \\
0 & -2 & 0 & 0 & 0 & 1\\
-1 & 0 & 0 & 0 & 1 & 0\\
2 & -2 & 0 & -1 & 0 & 1\\
0 & 2 & -1 & 0 & -1 & 0
\end{pmatrix},
\\
&\sigma = \begin{pmatrix}
0 & 1 & 0 & \frac{1}{2} & -1 & 0 \\
-1 & 0 & 1 & 0 & 1 & -1 \\
0 & -1 & 0 & 0 & 0 & \frac{1}{2}\\
-\frac{1}{2} & 0 & 0 & 0 & 1 & 0\\
1 & -1 & 0 & -1 & 0 & 1\\
0 & 1 & -\frac{1}{2} & 0 & -1 & 0
\end{pmatrix}.
\end{align*}
It is known that the two quivers are transformed 
to each other by a sequence of mutations $\mu_{2,5,2} := \mu_2 \mu_5 \mu_2$.

\begin{align}
\label{quiver:C2}
\begin{split}
\begin{tikzpicture}
\begin{scope}[>=latex]
\path (2,1) node[circle]{2} coordinate(E) node[above=0.2em]{$1$};
\draw (2,1) circle[radius=0.15];
\path (4,1) node[circle]{2} coordinate(F) node[above=0.2em]{$2$};
\draw (4,1) circle[radius=0.15];
\path (6,1) node[circle]{2} coordinate(G) node[above=0.2em]{$3$};
\draw (6,1) circle[radius=0.15];
\draw (1,0) circle(2pt) coordinate(B) node[below]{$4$};
\draw (3,0) circle(2pt) coordinate(C) node[below]{$5$};
\draw (5,0) circle(2pt) coordinate(D) node[below]{$6$};
\qarrow{B}{C}
\qarrow{C}{D}
\qsarrow{E}{F}
\qsarrow{F}{G}
\draw[->,shorten >=2pt,shorten <=4pt] (F) -- (C) [thick];
\draw[->,shorten >=4pt,shorten <=2pt] (C) -- (E) [thick];
\draw[->,shorten >=4pt,shorten <=2pt] (D) -- (F) [thick];
\draw[->,dashed,shorten >=2pt,shorten <=4pt] (E) -- (B) [thick];
\draw[->,dashed,shorten >=2pt,shorten <=4pt] (G) -- (D) [thick];
\coordinate (P1) at (6.6,0.5);
\coordinate (P2) at (7.6,0.5);
\draw[<->] (P1) -- (P2);
\draw (7.1,0) node{$\mu_{2,5,2}$};
\draw (3.5,-0.7) node{$J_{1212}$};  
\draw[<->] (3.5,-1.1) -- (3.5,-1.7);
\draw (3.5,-1.4) node[right]{$\mu_2$};
\end{scope}
\begin{scope}[>=latex,xshift=205pt]
\path (1,1) node[circle]{2} coordinate(E) node[above=0.2em]{$1$};
\draw (1,1) circle[radius=0.15];
\path (3,1) node[circle]{2} coordinate(F) node[above=0.2em]{$2$};
\draw (3,1) circle[radius=0.15];
\path (5,1) node[circle]{2} coordinate(G) node[above=0.2em]{$3$};
\draw (5,1) circle[radius=0.15];
\draw (2,0) circle(2pt) coordinate(B) node[below]{$4$};
\draw (4,0) circle(2pt) coordinate(C) node[below]{$5$};
\draw (6,0) circle(2pt) coordinate(D) node[below]{$6$};
\qarrow{B}{C}
\qarrow{C}{D}
\qsarrow{E}{F}
\qsarrow{F}{G}
%\draw[->,shorten >=4pt,shorten <=4pt] (E) -- (F) [thick];
%\draw[->,shorten >=4pt,shorten <=4pt] (F) -- (G) [thick];
\draw[->,shorten >=4pt,shorten <=2pt] (C) -- (F) [thick];
\draw[->,shorten >=2pt,shorten <=4pt] (F) -- (B) [thick];
\draw[->,shorten >=2pt,shorten <=4pt] (G) -- (C) [thick];
\draw[->,dashed,shorten >=4pt,shorten <=2pt] (B) -- (E) [thick];
\draw[->,dashed,shorten >=4pt,shorten <=2pt] (D) -- (G) [thick];
%{\color{red}
%\fill (3,0) circle(2pt) node[below]{$C$};
%\fill (5,0) circle(2pt) node[below]{$A$};
%\fill (2,1) circle(2pt) node[above]{$D$};
%\fill (4,1) circle(2pt) node[above]{$B$};
%}
\draw (3.5,-0.7) node{$J_{2121}$};  
\draw[<->] (3.5,-1.1) -- (3.5,-1.7);
\draw (3.5,-1.4) node[right]{$\mu_2$};
\end{scope}
\begin{scope}[>=latex,yshift=-100pt]
\path (2,1) node[circle]{2} coordinate(A) node[above=0.2em]{$1$};
\draw (2,1) circle[radius=0.15];
\path (4,1) node[circle]{2} coordinate(B) node[above=0.2em]{$2$};
\draw (4,1) circle[radius=0.15];
\path (6,1) node[circle]{2} coordinate(C) node[above=0.2em]{$3$};
\draw (6,1) circle[radius=0.15];
\draw (1,0) circle(2pt) coordinate(D) node[below]{$4$};
\draw (3,0) circle(2pt) coordinate(E) node[below]{$5$};
\draw (5,0) circle(2pt) coordinate(F) node[below]{$6$};
\qarrow{D}{E}
\qarrow{F}{E}
\qsarrow{B}{A}
\qsarrow{C}{B}
\draw[->,shorten >=4pt,shorten <=4pt] (A) to [out=30, in=150] (C) [thick];
\draw[->,shorten >=2pt,shorten <=4pt] (B) -- (F) [thick];
\draw[->,shorten >=4pt,shorten <=2pt] (E) -- (B) [thick];
%\draw[->,shorten >=4pt,shorten <=2pt] (D) -- (F) [thick];
\draw[->,dashed,shorten >=2pt,shorten <=4pt] (A) -- (D) [thick];
\draw[->,dashed,shorten >=4pt,shorten <=2pt] (F) -- (C) [thick];
%\qdarrow{A}{D}
%\qdarrow{F}{C} 
\draw[<->] (6.6,0.5) -- (7.6,0.5);
\draw (7.1,0) node{$\mu_5$};
\draw (3.5,-0.7) node{$B(2)$};
\end{scope}
\begin{scope}[>=latex,xshift=205pt,yshift=-100pt]
\path (2,1) node[circle]{2} coordinate(A) node[above=0.2em]{$1$};
\draw (2,1) circle[radius=0.15];
\path (4,1) node[circle]{2} coordinate(B) node[above=0.2em]{$2$};
\draw (4,1) circle[radius=0.15];
\path (6,1) node[circle]{2} coordinate(C) node[above=0.2em]{$3$};
\draw (6,1) circle[radius=0.15];
\draw (1,0) circle(2pt) coordinate(D) node[below]{$4$};
\draw (3,0) circle(2pt) coordinate(E) node[below]{$5$};
\draw (5,0) circle(2pt) coordinate(F) node[below]{$6$};
\qarrow{E}{D}
\qarrow{E}{F}
\qsarrow{B}{A}
\qsarrow{C}{B}
\draw[->,shorten >=4pt,shorten <=4pt] (A) to [out=30, in=150] (C) [thick];
\draw[->,shorten >=2pt,shorten <=4pt] (B) -- (E) [thick];
\draw[->,shorten >=4pt,shorten <=2pt] (D) -- (B) [thick];
\draw[->,dashed,shorten >=2pt,shorten <=4pt] (A) -- (D) [thick];
\draw[->,dashed,shorten >=4pt,shorten <=2pt] (F) -- (C) [thick];
%\draw[<->] (6.6,0.5) -- (7.6,0.5);
%\draw (7.1,0) node{$\mu_5$};
\draw (3.5,-0.7) node{$B(3)$};
\end{scope}
\end{tikzpicture}
\end{split}
\end{align}

Elements of $W(C_2)$ are depicted by wiring diagrams with two wires with a `wall' : 
$r_1$ has the same rule as $r_1$ for $W(A_2)$, while $r_2$ reflects a wire at the wall.
The above quivers are related to the wiring diagrams involving such reflections,
where the vertices of weight two are arranged along the wall.
We let $K_{1234}$ denote the transformation of the wiring diagrams as follows:

\begin{align}
\label{quiver-d:C2}
\begin{split}
\begin{tikzpicture}
\begin{scope}[>=latex,xshift=0pt]
{\color{red}
\fill (2,0) circle(2pt) coordinate(A) node[below]{$1$};
\fill (4,0) circle(2pt) coordinate(C) node[below]{$3$};
\fill (3,1) circle(2pt) coordinate(B) node[above]{$2$};
\fill (5,1) circle(2pt) coordinate(D) node[above]{$4$};
\draw [-] (1,0.5) node[left=2pt] {\footnotesize $2$} to [out = 0, in = 135] (A);
\draw [-] (A) to [out = -45, in = -135] (C); 
\draw [-] (C) -- (D);
\draw [->] (D) to [out = -45, in = 180] (6,0.5);
\draw [-] (1,-0.5) node[left=2pt] {\footnotesize $1$} to [out = 0, in = -135] (A); 
\draw [-] (A) -- (B);
\draw [-] (B) -- (C);
\draw [->] (C) to [out = -45, in = 180] (6,-0.5);
}
\draw[->] (6.5,0.5) -- (7.5,0.5);
\draw (7,0.5) circle(0pt) node[below]{$K_{1234}$};
%
%
%{\color{blue}
\path (2,1) node[circle]{2} coordinate(E) node[above=0.2em]{$1$};
\draw (2,1) circle[radius=0.15];
\path (4,1) node[circle]{2} coordinate(F) node[above=0.2em]{$2$};
\draw (4,1) circle[radius=0.15];
\path (6,1) node[circle]{2} coordinate(G) node[above=0.2em]{$3$};
\draw (6,1) circle[radius=0.15];
\draw (1,0) circle(2pt) coordinate(B) node[below]{$4$};%{$v^s_1$};
\draw (3,0) circle(2pt) coordinate(C) node[below]{$5$};%{$v^s_2$};
\draw (5,0) circle(2pt) coordinate(D) node[below]{$6$};%{$v^s_3$};
\qarrow{B}{C}
\qarrow{C}{D}
\draw[->,shorten >=4pt,shorten <=4pt] (E) -- (F) [thick];
\draw[->,shorten >=4pt,shorten <=4pt] (F) -- (G) [thick];
\draw[->,shorten >=2pt,shorten <=4pt] (F) -- (C) [thick];
\draw[->,shorten >=4pt,shorten <=2pt] (C) -- (E) [thick];
\draw[->,shorten >=4pt,shorten <=2pt] (D) -- (F) [thick];
\qdarrow{E}{B}
\qdarrow{G}{D} 
%\coordinate (P1) at (6.6,0.5);
%\coordinate (P2) at (7.6,0.5);
%\draw[<->] (P1) -- (P2);
%\draw (7.1,0) node{$\mu_2 \mu_5 \mu_2$};
%}
\end{scope}
\begin{scope}[>=latex,xshift=205pt]
{\color{red}
\fill (5,0) circle(2pt) coordinate(A) node[below]{$1$};
\fill (3,0) circle(2pt) coordinate(C) node[below]{$3$};
\fill (4,1) circle(2pt) coordinate(B) node[above]{$2$};
\fill (2,1) circle(2pt) coordinate(D) node[above]{$4$};
\draw [-] (1,0.5) node[left=2pt] {\footnotesize $2$} to [out = 0, in = -135] (D);
\draw [-] (D) -- (C); 
\draw [-] (C) to [out = -45, in = -135] (A);
\draw [->] (A) to [out = 45, in = 180] (6,0.5);
\draw [-] (1,-0.5) node[left=2pt] {\footnotesize $1$} to [out = 0, in = -135] (C); 
\draw [-] (C) -- (B);
\draw [-] (B) -- (A);
\draw [->] (A) to [out = -45, in = 180] (6,-0.5);
}
%
%
%{\color{blue}
\path (1,1) node[circle]{2} coordinate(E) node[above=0.2em]{$1$};%{$v_1^t$};
\draw (1,1) circle[radius=0.15];
\path (3,1) node[circle]{2} coordinate(F) node[above=0.2em]{$2$};%{$v_2^t$};
\draw (3,1) circle[radius=0.15];
\path (5,1) node[circle]{2} coordinate(G) node[above=0.2em]{$3$};%{$v_3^t$};
\draw (5,1) circle[radius=0.15];
\draw (2,0) circle(2pt) coordinate(B) node[below]{$4$};%{$v^s_1$};
\draw (4,0) circle(2pt) coordinate(C) node[below]{$5$};%{$v^s_2$};
\draw (6,0) circle(2pt) coordinate(D) node[below]{$6$};%{$v^s_3$};
\qarrow{B}{C}
\qarrow{C}{D}
\draw[->,shorten >=4pt,shorten <=4pt] (E) -- (F) [thick];
\draw[->,shorten >=4pt,shorten <=4pt] (F) -- (G) [thick];
\draw[->,shorten >=4pt,shorten <=2pt] (C) -- (F) [thick];
\draw[->,shorten >=2pt,shorten <=4pt] (F) -- (B) [thick];
\draw[->,shorten >=2pt,shorten <=4pt] (G) -- (C) [thick];
\draw[->,dashed,shorten >=4pt,shorten <=2pt] (B) -- (E) [thick];
\draw[->,dashed,shorten >=4pt,shorten <=2pt] (D) -- (G) [thick];
%}
\end{scope}
\end{tikzpicture}
\end{split}
\end{align}

\subsection{Cluster transformation for $K_{1234}$}\label{sb:cK}

We consider the sequence of quantum $y$-seeds associated to $\mu_{2,5,2}$:
$$
(B(1),Y(1)):=(J_{1212},Y)  \stackrel{\mu_{2}}{\longleftrightarrow}
(B(2),Y(2)) \stackrel{\mu_{5}}{\longleftrightarrow}
(B(3),Y(3)) \stackrel{\mu_{2}}{\longleftrightarrow} (B(4),Y(4)),
$$
where $B(4)=J_{2121}$ and $B(2), B(3)$ are found in \eqref{quiver:C2}. 
The quantum $y$-variables are given as follows:
\begin{align}\label{eq:K-full}
\begin{split}
&Y(1) = 
\begin{pmatrix}
Y_1 \\ Y_2 \\ Y_3 \\ Y_4 \\ Y_5 \\ Y_6
\end{pmatrix},
~~Y(2) = 
\begin{pmatrix}
Y_1(1+q^2Y_2^{-1})^{-1} \\ Y_2^{-1} \\ Y_3(1+q^2Y_2) \\ Y_4 \\
Y_5(1+q^2Y_2) \\ Y_6(1+q^2Y_2^{-1})^{-1}
\end{pmatrix},
~~Y(3) = 
\begin{pmatrix}
Y_1(2)\\ Y_2^{-1}(1+qY_5(2))(1+q^3Y_5(2)) \\ Y_3(2) \\ Y_4(1+qY_5(2)^{-1})^{-1} 
\\ Y_5(2)^{-1} \\ Y_6(2)(1+qY_5(2)^{-1})^{-1}
\end{pmatrix}
\\ 
&Y(4) =
\begin{pmatrix}
Y_1(2)(1+q^2Y_2(3))\\ Y_2(3)^{-1} \\ Y_3(2)(1+q^2Y_2(3)^{-1})^{-1} \\ Y_4(3)(1+q^2Y_2(3)^{-1})^{-1} \\ Y_5(3)(1+q^2Y_2(3)) \\ Y_6(3)
\end{pmatrix}
=
\begin{pmatrix}
Y_1 \,\Lambda
\\ 
(1+q^3Y_5(1+q^2Y_2))^{-1} (1+qY_5(1+q^2Y_2))^{-1} Y_2
\\
\Lambda^{-1} Y_3 (1+q^3Y_5(1+q^2Y_2)) (1+qY_5(1+q^2Y_2)) 
\\
\Lambda^{-1} q^{-1}Y_4 Y_5 (1+qY_5(1+q^2Y_2))
\\
q^2 Y_5^{-1} Y_2^{-1} \Lambda
\\
(1+qY_5(1+q^2Y_2))^{-1} q^{-3} Y_6 Y_2 Y_5   
\end{pmatrix}.
\end{split}
\end{align}
Here we have set $\Lambda := 1+(q+q^3)Y_5 + q^4Y_5^2(1+q^2Y_2)$.

\begin{lem}\label{lem:K-tropsign}
The tropical sign-sequence of $\mu_{2,5,2}$ for a tropical $y$-seed $(J_{1212},y)$ is $(+,+,-)$, if the tropical signs of $y_2$ and $y_5$ in $y = (y_1,y_2,y_3,y_4,y_5,y_6)$ are positive.
\end{lem}

\begin{proof}
When the tropical signs of $y_2$ and $y_5$ are positive, 
$\mu_{2,5,2}$ transforms $(J_{1212},y)$ as follows:
\begin{align}\label{eq:K-trop-mutation}
\begin{split}
&y=(y_1,\underline{y_2},y_3,y_4,y_5,y_6) \stackrel{\mu_2}{\mapsto}
(y_1 y_2, y_2^{-1},y_3,y_4,\underline{y_5}, y_2 y_6)
\\ &\qquad 
\stackrel{\mu_5}{\mapsto} (y_1 y_2, \underline{y_2^{-1}},y_3,y_4 y_5,y_5^{-1}, y_2 y_5 y_6)\stackrel{\mu_2}{\mapsto} (y_1, y_2,y_3,y_4 y_5,(y_2 y_5)^{-1}, y_2 y_5 y_6).
\end{split}
\end{align}
Thus the tropical sign sequence is determined by the underlined tropical $y$-variables as claimed.
\end{proof}

By applying \eqref{eq:ad-tau-decomp} to this mutation sequence with  
$(\delta_1,\delta_2,\delta_3)=(+,+,-)$, we decompose the transformation as
\begin{align}\label{eq:K-decomp}
\mu^\ast_{2,5,2} := 
\mu_2^\ast \,\mu_5^\ast \,\mu_2^\ast = \Ad(\Psi_{q^2}(Y_2) \Psi_q(Y_5) \Psi_{q^2}(Y_2)^{-1}) \tau_{2,+} \tau_{5,+} \tau_{2,-}.
\end{align}
One sees that the image of the generators $\rY_i(t) ~(t=1,2,3,4)$ of the quantum tori $\mathcal{T}(B(t))$ 
in the initial torus $\mathcal{T}(J_{1212})$ is  
\begin{equation}
\begin{split}
&(\rY_i(1)=\rY_i)_i = (Y_1,Y_2,Y_3,Y_4, Y_5,Y_6),
\\
&(\tau_{2,+}(\rY_i(2)))_i = (q^{-2}Y_1Y_2, Y_2^{-1},Y_3,Y_4,Y_5,q^{-2} Y_6 Y_2),
\\
&(\tau_{2,+} \tau_{5,+}(\rY_i(3)))_i = (q^{-2} Y_1 Y_2, Y_2^{-1}, Y_3, q^{-1} Y_4 Y_5, Y_5^{-1}, q^{-3} Y_6 Y_2 Y_5),
\\
&(\tau_{2,+} \tau_{5,+} \tau_{2,-}(\rY_i(4)))_i =
(Y_1, Y_2, Y_3, q^{-1} Y_4 Y_5, q^2 Y_5^{-1} Y_2^{-1},q^{-3}Y_6 Y_2 Y_5).  
\end{split}
\end{equation}

Assign a pair of noncommuting variables $(p_i,u_i)$ to the crossing $i$ of the wiring diagram \eqref{quiver-d:C2} for $i=1,2,3,4$, satisfying 
\begin{align}\label{pu2}
[p_i,u_j] 
= \begin{cases}
\hbar & i=j=1,3, \\
2 \hbar & i=j=2,4, \\
0 & \text{otherwise},
\end{cases} 
\qquad 
[p_i,p_j] = [u_i,u_j] = 0.
\end{align}
Let $\mathcal{W}(C_2)$ be the algebra over $\C$ generated by the $q$-commuting symbols
$e^{\pm p_i}, e^{\pm u_i} ~(i=1,2,3,4)$. 
We use non-zero complex parameters $\kappa_i ~(i=1,2,3,4)$ which are parameterized with 
$\lambda_i$ as in \eqref{para}.
Let $Y_i$ and $Y'_i$ denote the generators of $\mathcal{Y}(J_{1212})$ and $\mathcal{Y}(J_{2121})$ respectively.  
Let Frac$\mathcal{W}(C_2)$ be the noncommuting fractional field of $\mathcal{W}(C_2)$. 
We define the embeddings $\phi: \mathcal{Y}(J_{1212}) \hookrightarrow \mathrm{Frac}\mathcal{W}(C_2)$ and $\phi': \mathcal{Y}(J_{2121}) \hookrightarrow \mathrm{Frac}\mathcal{W}(C_2)$ by
\begin{align}\label{eq:phi-K}
\phi: 
\begin{cases}
Y_1 \mapsto \kappa_2^{-1}e^{p_2-u_2-2 p_1},
\\
Y_2 \mapsto \kappa_2 \kappa_4^{-1} e^{p_2+u_2+p_4-u_4-2 p_3},
\\
Y_3 \mapsto \kappa_4 \,e^{p_4+u_4},
\\
Y_4 \mapsto \kappa_1^{-1} e^{p_1-u_1},
\\
Y_5 \mapsto \kappa_1 \kappa_3^{-1} e^{p_1+u_1+p_3-u_3-p_2},
\\
Y_6 \mapsto \kappa_3 \,e^{p_3+u_3-p_4},
\end{cases}
\qquad
\phi':
\begin{cases}
Y'_1 \mapsto \kappa_4^{-1} e^{p_4-u_4},
\\
Y'_2 \mapsto \kappa_4 \kappa_2^{-1} e^{p_4+u_4+p_2-u_2-2 p_3},
\\
Y'_3 \mapsto \kappa_2\, e^{p_2+u_2-2 p_1},
\\
Y'_4 \mapsto \kappa_3^{-1} e^{p_3-u_3-p_4},
\\
Y'_5 \mapsto \kappa_3 \kappa_1^{-1} e^{p_3+u_3+p_1-u_1-p_2},
\\
Y'_6 \mapsto \kappa_1 \,e^{p_1+u_1}. 
\end{cases}
\end{align} 
We also define an isomorphism $\pi^K_{1234}$ of $\mathcal{W}(C_2)$ by
\begin{align}\label{eq:K-pi}
\pi^K_{1234}:
\begin{cases}
\;p_1 \mapsto p_1+ \lambda_{24}, 
\quad 
p_2 \mapsto p_4+ 2p_1+ \lambda_{24},
\\
\;p_3  \mapsto p_3-\lambda_{24},
\quad
p_4 \mapsto p_2-2 p_1- \lambda_{24},
\\
\;u_1 \mapsto u_1+u_2-u_4,
\quad 
u_2  \mapsto u_4,
\\
\; u_3  \mapsto u_3,
\quad 
u_4 \mapsto u_2,
\end{cases}
\end{align}
in the sense of exponentials, which induces an isomorphism of Frac$\mathcal{W}(C_2)$.
The following proposition is proved in the same manner as Proposition \ref{prop:R-ad}.

\begin{prop}\label{prop:K-ad}
(i) The maps $\phi$ and $\phi'$ are morphisms of skew fields.
Moreover, the following diagram is commutative:
\begin{align*}
\xymatrix{
\mathcal{Y}(J_{2121}) \ar[r]^{\phi'} \ar[d]_{\tau_{2,+}\tau_{5,+}\tau_{2,-}} & \mathrm{Frac}\mathcal{W}(C_2) \ar[d]_{\pi^K_{1234}}
\\
\mathcal{Y}(J_{1212}) \ar[r]^{\phi}& \mathrm{Frac}\mathcal{W}(C_2)
}
\end{align*}
(ii) The isomorphism $\pi^K_{1234}$ is realized by an adjoint action $\Ad(P^K_{1234})$ of 
\begin{align}\label{K-Pop}
P^K_{1234} = \rho_{24}\, e^{\frac{1}{\hbar}p_1(u_4-u_2)} 
e^{\frac{\lambda_{24}}{2\hbar}(2u_3-2u_1+u_4-u_2)}.
\end{align}   
Here $\rho_{24} \in \mathfrak{S}_4$ is a permutation of the indices $2$ and $4$ 
of the generators of $\mathcal{W}(C_2)$.
\end{prop}

%%%%%%%%%%%%%%%%%%%%%%%%%%%%%%
\section{Tetrahedron equation}\label{sec:tetra}
%%%%%%%%%%%%%%%%%%%%%%%%%%%%%%

We consider the transformation of reduced expressions of $w_0 \in W(A_3)$,
from $123121$ to $321323$, in two ways as follows:
$$
\begin{matrix}
\underline{121}321 & \to & 21\underline{232}1 & \to & 2\underline{1}3\underline{2}3\underline{1} & \to & \underline{232}123 & \to & 323123  
\\
\rotatebox{90}{=} & & & & & & & & \rotatebox{90}{=}
\\
123\underline{121} & \to & 1\underline{232}12 & \to & \underline{1}3\underline{2}3\underline{1}2 & \to & 321\underline{232} & \to & 321323
\end{matrix}
$$ 
This corresponds to the transformation of wiring diagrams shown in Figure \ref{fig:tetrahedron-eq}.
It follows that $R_{ijk}$ in \eqref{quiver-d:A2} 
satisfies the tetrahedron relation for the wiring diagrams:
\begin{align}\label{eq:tetra-R-eq}
R_{456} R_{236} R_{135} R_{124} = R_{124} R_{135} R_{236} R_{456}.
\end{align}

\begin{figure}[h]
\[
\scalebox{0.9}{
\begin{tikzpicture}
\begin{scope}[>=latex,xshift=0pt]
{\color{red}
\fill (1,0.5) circle(2pt) coordinate(A) node[below]{$1$};
\fill (2,1.5) circle(2pt) coordinate(B) node[above]{$2$};
\fill (3,0.5) circle(2pt) coordinate(C) node[below]{$4$};
\fill (3,2.5) circle(2pt) coordinate(D) node[above]{$3$};
\fill (4,1.5) circle(2pt) coordinate(E) node[above]{$5$};
\fill (5,0.5) circle(2pt) coordinate(F) node[below]{$6$};
\draw [-] (0,3) to [out = 0, in = 135] (D);
\draw [-] (D) -- (F); 
\draw [->] (F) to [out = -45, in = 180] (6,0);
\draw [-] (0,2) to [out = 0, in = 135] (B);
\draw [-] (B) -- (C); 
\draw [-] (C) to [out = -45, in = -135] (F);
\draw [->] (F) to [out = 45, in = 180] (6,1);
\draw [-] (0,1) to [out = 0, in = 135] (A); 
\draw [-] (A) to [out = -45, in = -135] (C);
\draw [-] (C) -- (E);
\draw [->] (E) to [out = 45, in = 180] (6,2);
\draw [-] (0,0) to [out = 0, in = -135] (A); 
\draw [-] (A) -- (D);
\draw [->] (D) to [out = 45, in = 180] (6,3);
}
\coordinate (P1) at (3,-0.2);
\coordinate (P2) at (3,-0.7);
\draw[->] (P1) -- (P2);
\draw (3,-0.5) node[right]{$\mu_7$};
\draw (3,-0.5) node[left]{$R_{124}$};
\draw[->] (6.5,1.5) -- (7.5,1.5);
\draw (7,1.5) node[below]{$\mu_8$};
\draw (7,1.5) node[above]{$R_{456}$};
%
%{\color{blue}
\draw (0,0.5) circle(2pt) coordinate(6) node[below]{$6$};
\draw (2,0.5) circle(2pt) coordinate(7) node[below]{$7$};
\draw (4,0.5) circle(2pt) coordinate(8) node[below]{$8$};
\draw (6,0.5) circle(2pt) coordinate(9) node[below]{$9$};
\draw (1,1.5) circle(2pt) coordinate(3) node[above]{$3$};
\draw (3,1.5) circle(2pt) coordinate(4) node[above]{$4$};
\draw (5,1.5) circle(2pt) coordinate(5) node[above]{$5$};
\draw (2,2.5) circle(2pt) coordinate(1) node[above]{$1$};
\draw (4,2.5) circle(2pt) coordinate(2) node[above]{$2$};
\qarrow{1}{2}
\qarrow{3}{4}
\qarrow{4}{5}
\qarrow{6}{7}
\qarrow{7}{8}
\qarrow{8}{9}
\qdarrow{9}{5}
\qdarrow{5}{2}
\qdarrow{1}{3}
\qdarrow{3}{6}
\qarrow{5}{8}
\qarrow{8}{4}
\qarrow{4}{7}
\qarrow{7}{3}
\qarrow{2}{4}
\qarrow{4}{1}
%}
\end{scope}
\begin{scope}[>=latex,yshift=-115pt]
{\color{red}
\fill (3,1.5) circle(2pt) coordinate(A) node[below]{$1$};
\fill (2,0.5) circle(2pt) coordinate(B) node[above]{$2$};
\fill (1,1.5) circle(2pt) coordinate(C) node[below]{$4$};
\fill (4,2.5) circle(2pt) coordinate(D) node[above]{$3$};
\fill (5,1.5) circle(2pt) coordinate(E) node[above]{$5$};
\fill (6,0.5) circle(2pt) coordinate(F) node[below]{$6$};
\draw [-] (0,3) to [out = 0, in = 135] (D);
\draw [-] (D) -- (F); 
\draw [->] (F) to [out = -45, in = 180] (7,0);
\draw [-] (0,2) to [out = 0, in = 135] (C);
\draw [-] (C) -- (B); 
\draw [-] (B) to [out = -45, in = -135] (F);
\draw [->] (F) to [out = 45, in = 180] (7,1);
\draw [-] (0,1) to [out = 0, in = -135] (C); 
\draw [-] (C) to [out = 45, in = 135] (A);
\draw [-] (A) to [out = -45, in = -135] (E);
\draw [->] (E) to [out = 45, in = 180] (7,2);
\draw [-] (0,0) to [out = 0, in = -135] (B); 
\draw [-] (B) -- (D);
\draw [->] (D) to [out = 45, in = 180] (7,3);
}
\coordinate (P1) at (3,-0.2);
\coordinate (P2) at (3,-0.7);
\draw[->] (P1) -- (P2);
\draw (3,-0.5) node[right]{$\mu_4$};
\draw (3,-0.5) node[left]{$R_{135}$};
%
%{\color{blue}
\draw (1,0.5) circle(2pt) coordinate(6) node[below]{$6$};
\draw (2,1.5) circle(2pt) coordinate(7) node[below]{$7$};
\draw (4,0.5) circle(2pt) coordinate(8) node[below]{$8$};
\draw (7,0.5) circle(2pt) coordinate(9) node[below]{$9$};
\draw (0,1.5) circle(2pt) coordinate(3) node[above]{$3$};
\draw (4,1.5) circle(2pt) coordinate(4) node[above]{$4$};
\draw (6,1.5) circle(2pt) coordinate(5) node[above]{$5$};
\draw (3,2.5) circle(2pt) coordinate(1) node[above]{$1$};
\draw (5,2.5) circle(2pt) coordinate(2) node[above]{$2$};
\qarrow{1}{2}
\qarrow{3}{7}
\qarrow{7}{4}
\qarrow{4}{5}
\qarrow{6}{8}
\qarrow{8}{9}
\qdarrow{5}{2}
\qarrow{2}{4}
\qarrow{4}{1}
\qdarrow{1}{3}
\qdarrow{9}{5}
\qarrow{5}{8}
\qarrow{8}{7}
\qarrow{7}{6}
\qdarrow{6}{3}
%}
\end{scope}
\begin{scope}[>=latex,yshift=-230pt,xshift=28pt]
{\color{red}
\fill (4,2.5) circle(2pt) coordinate(A) node[below]{$1$};
\fill (2,0.5) circle(2pt) coordinate(B) node[above]{$2$};
\fill (1,1.5) circle(2pt) coordinate(C) node[below]{$4$};
\fill (3,1.5) circle(2pt) coordinate(D) node[above]{$3$};
\fill (2,2.5) circle(2pt) coordinate(E) node[above]{$5$};
\fill (4,0.5) circle(2pt) coordinate(F) node[below]{$6$};
\draw [-] (0,3) to [out = 0, in = 135] (E);
\draw [-] (E) -- (F); 
\draw [->] (F) to [out = -45, in = 180] (5,0);
\draw [-] (0,2) to [out = 0, in = 135] (C);
\draw [-] (C) -- (B); 
\draw [-] (B) to [out = -45, in = -135] (F);
\draw [->] (F) to [out = 45, in = 180] (5,1);
\draw [-] (0,1) to [out = 0, in = -135] (C); 
\draw [-] (C) -- (E);
\draw [-] (E) to [out = 45, in = 135] (A);
\draw [->] (A) to [out = -45, in = 180] (5,2);
\draw [-] (0,0) to [out = 0, in = -135] (B); 
\draw [-] (B) -- (A);
\draw [->] (A) to [out = 45, in = 180] (5,3);
}
\draw[->] (2,-0.2) -- (2,-0.8);
\draw (2,-0.5) node[right]{$\mu_8$};
\draw (2,-0.5) node[left]{$R_{236}$};
%
%{\color{blue}
\draw (1,0.5) circle(2pt) coordinate(6) node[below]{$6$};
\draw (2,1.5) circle(2pt) coordinate(7) node[below]{$7$};
\draw (3,0.5) circle(2pt) coordinate(8) node[below]{$8$};
\draw (5,0.5) circle(2pt) coordinate(9) node[below]{$9$};
\draw (0,1.5) circle(2pt) coordinate(3) node[above]{$3$};
\draw (3,2.5) circle(2pt) coordinate(4) node[above]{$4$};
\draw (4,1.5) circle(2pt) coordinate(5) node[above]{$5$};
\draw (1,2.5) circle(2pt) coordinate(1) node[above]{$1$};
\draw (5,2.5) circle(2pt) coordinate(2) node[above]{$2$};
\qarrow{1}{4}
\qarrow{4}{2}
\qarrow{3}{7}
\qarrow{7}{5}
\qarrow{6}{8}
\qarrow{8}{9}
\qdarrow{2}{5}
\qarrow{5}{4}
\qarrow{4}{7}
\qarrow{7}{1}
\qdarrow{1}{3}
\qdarrow{9}{5}
\qarrow{5}{8}
\qarrow{8}{7}
\qarrow{7}{6}
\qdarrow{6}{3}
%}
\end{scope}
\begin{scope}[>=latex,yshift=-345pt]
{\color{red}
\fill (6,2.5) circle(2pt) coordinate(A) node[below]{$1$};
\fill (5,1.5) circle(2pt) coordinate(B) node[above]{$2$};
\fill (1,1.5) circle(2pt) coordinate(C) node[below]{$4$};
\fill (4,0.5) circle(2pt) coordinate(D) node[above]{$3$};
\fill (2,2.5) circle(2pt) coordinate(E) node[above]{$5$};
\fill (3,1.5) circle(2pt) coordinate(F) node[below]{$6$};
\draw [-] (0,3) to [out = 0, in = 135] (E);
\draw [-] (E) -- (D); 
\draw [->] (D) to [out = -45, in = 180] (7,0);
\draw [-] (0,2) to [out = 0, in = 135] (C);
\draw [-] (C) to [out = -45, in = -135] (F); 
\draw [-] (F) to [out = 45, in = 135] (B);
\draw [->] (B) to [out = -45, in = 180] (7,1);
\draw [-] (0,1) to [out = 0, in = -135] (C); 
\draw [-] (C) -- (E);
\draw [->] (E) to [out = 45, in = 135] (A);
\draw [->] (A) to [out = -45, in = 180] (7,2);
\draw [-] (0,0) to [out = 0, in = -135] (D); 
\draw [-] (D) -- (A);
\draw [->] (A) to [out = 45, in = 180] (7,3);
}
\coordinate (P1) at (3,-0.2);
\coordinate (P2) at (3,-0.7);
\draw[->] (P1) -- (P2);
\draw (3,-0.5) node[right]{$\mu_7$};
\draw (3,-0.5) node[left]{$R_{456}$};
%
%\coordinate (P3) at (6.5,1.5);
%\coordinate (P4) at (7.5,1.5);
%
%{\color{blue}
\draw (2,0.5) circle(2pt) coordinate(6) node[below]{$6$};
\draw (2,1.5) circle(2pt) coordinate(7) node[below]{$7$};
\draw (4,1.5) circle(2pt) coordinate(8) node[below]{$8$};
\draw (6,0.5) circle(2pt) coordinate(9) node[below]{$9$};
\draw (0,1.5) circle(2pt) coordinate(3) node[above]{$3$};
\draw (4,2.5) circle(2pt) coordinate(4) node[above]{$4$};
\draw (6,1.5) circle(2pt) coordinate(5) node[above]{$5$};
\draw (1,2.5) circle(2pt) coordinate(1) node[above]{$1$};
\draw (7,2.5) circle(2pt) coordinate(2) node[above]{$2$};
\qarrow{1}{4}
\qarrow{4}{2}
\qarrow{3}{7}
\qarrow{7}{8}
\qarrow{8}{5}
\qarrow{6}{9}
\qdarrow{2}{5}
\qarrow{5}{4}
\qarrow{4}{7}
\qarrow{7}{1}
\qdarrow{1}{3}
\qdarrow{5}{9}
\qarrow{9}{8}
\qarrow{8}{6}
\qdarrow{6}{3}
%}
\end{scope}
\begin{scope}[>=latex,yshift=-460pt]
{\color{red}
\fill (5,2.5) circle(2pt) coordinate(A) node[below]{$1$};
\fill (4,1.5) circle(2pt) coordinate(B) node[above]{$2$};
\fill (3,2.5) circle(2pt) coordinate(C) node[below]{$4$};
\fill (3,0.5) circle(2pt) coordinate(D) node[above]{$3$};
\fill (2,1.5) circle(2pt) coordinate(E) node[above]{$5$};
\fill (1,2.5) circle(2pt) coordinate(F) node[below]{$6$};
\draw [-] (0,3) to [out = 0, in = 135] (F);
\draw [-] (F) -- (D); 
\draw [->] (D) to [out = -45, in = 180] (6,0);
\draw [-] (0,2) to [out = 0, in = -135] (F);
\draw [-] (F) to [out = 45, in = 135] (C); 
\draw [-] (C) -- (B);
\draw [->] (B) to [out = -45, in = 180] (6,1);
\draw [-] (0,1) to [out = 0, in = -135] (E); 
\draw [-] (E) -- (C);
\draw [->] (C) to [out = 45, in = 135] (A);
\draw [->] (A) to [out = -45, in = 180] (6,2);
\draw [-] (0,0) to [out = 0, in = -135] (D); 
\draw [-] (D) -- (A);
\draw [->] (A) to [out = 45, in = 180] (6,3);
}
\draw[->] (6.5,1.5) -- (7.5,1.5);
\draw (7,1.5) node[below]{$\sigma_{4,8}$};
%
%{\color{blue}
\draw (2,0.5) circle(2pt) coordinate(6) node[below]{$6$};
\draw (2,2.5) circle(2pt) coordinate(7) node[above]{$7$};
\draw (3,1.5) circle(2pt) coordinate(8) node[below]{$8$};
\draw (4,0.5) circle(2pt) coordinate(9) node[below]{$9$};
\draw (1,1.5) circle(2pt) coordinate(3) node[above]{$3$};
\draw (4,2.5) circle(2pt) coordinate(4) node[above]{$4$};
\draw (5,1.5) circle(2pt) coordinate(5) node[above]{$5$};
\draw (0,2.5) circle(2pt) coordinate(1) node[above]{$1$};
\draw (6,2.5) circle(2pt) coordinate(2) node[above]{$2$};
\qarrow{1}{7}
\qarrow{7}{4}
\qarrow{4}{2}
\qarrow{3}{8}
\qarrow{8}{5}
\qarrow{6}{9}
\qdarrow{2}{5}
\qarrow{5}{4}
\qarrow{4}{8}
\qarrow{8}{7}
\qarrow{7}{3}
\qdarrow{3}{1}
\qdarrow{5}{9}
\qarrow{9}{8}
\qarrow{8}{6}
\qdarrow{6}{3}
%}
\end{scope}
\begin{scope}[>=latex,xshift=230pt]
{\color{red}
\fill (1,0.5) circle(2pt) coordinate(A) node[below]{$1$};
\fill (2,1.5) circle(2pt) coordinate(B) node[above]{$2$};
\fill (6,1.5) circle(2pt) coordinate(C) node[below]{$4$};
\fill (3,2.5) circle(2pt) coordinate(D) node[above]{$3$};
\fill (5,0.5) circle(2pt) coordinate(E) node[above]{$5$};
\fill (4,1.5) circle(2pt) coordinate(F) node[below]{$6$};
\draw [-] (0,3) to [out = 0, in = 135] (D);
\draw [-] (D) -- (E); 
\draw [->] (E) to [out = -45, in = 180] (7,0);
\draw [-] (0,2) to [out = 0, in = 135] (B);
\draw [-] (B) to [out = -45, in = -135] (F); 
\draw [-] (F) to [out = 45, in = 135] (C);
\draw [->] (C) to [out = -45, in = 180] (7,1);
\draw [-] (0,1) to [out = 0, in = 135] (A); 
\draw [-] (A) to [out = -45, in = -135] (E);
\draw [-] (E) -- (C);
\draw [->] (C) to [out = 45, in = 180] (7,2);
\draw [-] (0,0) to [out = 0, in = -135] (A); 
\draw [-] (A) -- (D);
\draw [->] (D) to [out = 45, in = 180] (7,3);
}
\coordinate (P1) at (3,-0.2);
\coordinate (P2) at (3,-0.7);
\draw[->] (P1) -- (P2);
\draw (3,-0.5) node[right]{$\mu_4$};
\draw (3,-0.5) node[left]{$R_{236}$};
%
%
%{\color{blue}
\draw (0,0.5) circle(2pt) coordinate(6) node[below]{$6$};%{$v_1^1$};
\draw (3,0.5) circle(2pt) coordinate(7) node[below]{$7$};%{$v_2^1$}; %B1
\draw (5,1.5) circle(2pt) coordinate(8) node[below]{$8$};%{$v_3^1$}; %C1
\draw (6,0.5) circle(2pt) coordinate(9) node[below]{$9$};%{$v_4^1$};%C1'
\draw (1,1.5) circle(2pt) coordinate(3) node[above]{$3$};%{$v_1^2$};%D1
\draw (3,1.5) circle(2pt) coordinate(4) node[above]{$4$};%{$v_2^2$};
\draw (7,1.5) circle(2pt) coordinate(5) node[above]{$5$};%{$v_3^2$};
\draw (2,2.5) circle(2pt) coordinate(1) node[above]{$1$};%{$v_1^3$};
\draw (4,2.5) circle(2pt) coordinate(2) node[above]{$2$};%{$v_2^3$};%H1
\qarrow{1}{2}
\qarrow{3}{4}
\qarrow{4}{8}
\qarrow{8}{5}
\qarrow{6}{7}
\qarrow{7}{9}
\qdarrow{5}{2}
\qarrow{2}{4}
\qarrow{4}{1}
\qdarrow{1}{3}
\qdarrow{5}{9}
\qarrow{9}{8}
\qarrow{8}{7}
\qarrow{7}{3}
\qdarrow{3}{6}
%}
\end{scope}
\begin{scope}[>=latex,xshift=258pt,yshift=-115pt]
{\color{red}
\fill (1,0.5) circle(2pt) coordinate(A) node[below]{$1$};
\fill (3,2.5) circle(2pt) coordinate(B) node[above]{$2$};
\fill (4,1.5) circle(2pt) coordinate(C) node[below]{$4$};
\fill (2,1.5) circle(2pt) coordinate(D) node[above]{$3$};
\fill (3,0.5) circle(2pt) coordinate(E) node[above]{$5$};
\fill (1,2.5) circle(2pt) coordinate(F) node[below]{$6$};
\draw [-] (0,3) to [out = 0, in = 135] (F);
\draw [-] (F) -- (E); 
\draw [->] (E) to [out = -45, in = 180] (5,0);
\draw [-] (0,2) to [out = 0, in = 45] (F);
\draw [-] (F) to [out = 45, in = 135] (B); 
\draw [-] (B) -- (C);
\draw [->] (C) to [out = -45, in = 180] (5,1);
\draw [-] (0,1) to [out = 0, in = -45] (A); 
\draw [-] (A) to [out = -45, in = -135] (E);
\draw [-] (E) -- (C);
\draw [->] (C) to [out = 45, in = 180] (5,2);
\draw [-] (0,0) to [out = 0, in = -135] (A); 
\draw [-] (A) -- (B);
\draw [->] (B) to [out = 45, in = 180] (5,3);
}
\draw[->] (2,-0.2) -- (2,-0.7);
\draw (2,-0.5) node[right]{$\mu_7$};
\draw (2,-0.5) node[left]{$R_{135}$};
%
%{\color{blue}
\draw (0,0.5) circle(2pt) coordinate(6) node[below]{$6$};%{$v_1^1$};
\draw (2,0.5) circle(2pt) coordinate(7) node[below]{$7$};%{$v_2^1$}; %B1
\draw (3,1.5) circle(2pt) coordinate(8) node[below]{$8$};%{$v_3^1$}; %C1
\draw (4,0.5) circle(2pt) coordinate(9) node[below]{$9$};%{$v_4^1$};%C1'
\draw (1,1.5) circle(2pt) coordinate(3) node[above]{$3$};%{$v_1^2$};%D1
\draw (2,2.5) circle(2pt) coordinate(4) node[above]{$4$};%{$v_2^2$};
\draw (5,1.5) circle(2pt) coordinate(5) node[above]{$5$};%{$v_3^2$};
\draw (0,2.5) circle(2pt) coordinate(1) node[above]{$1$};%{$v_1^3$};
\draw (4,2.5) circle(2pt) coordinate(2) node[above]{$2$};%{$v_2^3$};%H1
\qarrow{1}{4}
\qarrow{4}{2}
\qarrow{3}{8}
\qarrow{8}{5}
\qarrow{6}{7}
\qarrow{7}{9}
\qdarrow{5}{2}
\qarrow{2}{8}
\qarrow{8}{4}
\qarrow{4}{3}
\qdarrow{3}{1}
\qdarrow{5}{9}
\qarrow{9}{8}
\qarrow{8}{7}
\qarrow{7}{3}
\qdarrow{3}{6}
%}
\end{scope}
\begin{scope}[>=latex,xshift=230pt,yshift=-230pt]
{\color{red}
\fill (4,1.5) circle(2pt) coordinate(A) node[below]{$1$};
\fill (5,2.5) circle(2pt) coordinate(B) node[above]{$2$};
\fill (6,1.5) circle(2pt) coordinate(C) node[below]{$4$};
\fill (3,0.5) circle(2pt) coordinate(D) node[above]{$3$};
\fill (2,1.5) circle(2pt) coordinate(E) node[above]{$5$};
\fill (1,2.5) circle(2pt) coordinate(F) node[below]{$6$};
\draw [-] (0,3) to [out = 0, in = 135] (F);
\draw [-] (F) -- (D); 
\draw [->] (D) to [out = -45, in = 180] (7,0);
\draw [-] (0,2) to [out = 0, in = 45] (F);
\draw [-] (F) to [out = 45, in = 135] (B); 
\draw [-] (B) -- (C);
\draw [->] (C) to [out = -45, in = 180] (7,1);
\draw [-] (0,1) to [out = 0, in = -135] (E); 
\draw [-] (E) to [out = 45, in = 135] (A);
\draw [-] (A) to [out = -45, in = -135] (C);
\draw [->] (C) to [out = 45, in = 180] (7,2);
\draw [-] (0,0) to [out = 0, in = -135] (D); 
\draw [-] (D) -- (B);
\draw [->] (B) to [out = 45, in = 180] (7,3);
}
\coordinate (P1) at (3,-0.2);
\coordinate (P2) at (3,-0.7);
\draw[->] (P1) -- (P2);
\draw (3,-0.5) node[right]{$\mu_8$};
\draw (3,-0.5) node[left]{$R_{124}$};
%
%{\color{blue}
\draw (2,0.5) circle(2pt) coordinate(6) node[below]{$6$};%{$v_1^1$};
\draw (3,1.5) circle(2pt) coordinate(7) node[below]{$7$};%{$v_2^1$}; %B1
\draw (5,1.5) circle(2pt) coordinate(8) node[below]{$8$};%{$v_3^1$}; %C1
\draw (4,0.5) circle(2pt) coordinate(9) node[below]{$9$};%{$v_4^1$};%C1'
\draw (1,1.5) circle(2pt) coordinate(3) node[above]{$3$};%{$v_1^2$};%D1
\draw (3,2.5) circle(2pt) coordinate(4) node[above]{$4$};%{$v_2^2$};
\draw (7,1.5) circle(2pt) coordinate(5) node[above]{$5$};%{$v_3^2$};
\draw (0,2.5) circle(2pt) coordinate(1) node[above]{$1$};%{$v_1^3$};
\draw (6,2.5) circle(2pt) coordinate(2) node[above]{$2$};%{$v_2^3$};%H1
\qarrow{1}{4}
\qarrow{4}{2}
\qarrow{3}{7}
\qarrow{7}{8}
\qarrow{8}{5}
\qarrow{6}{9}
\qdarrow{5}{2}
\qarrow{2}{8}
\qarrow{8}{4}
\qarrow{4}{3}
\qdarrow{3}{1}
\qdarrow{5}{9}
\qarrow{9}{7}
\qarrow{7}{6}
\qdarrow{6}{3}
%}
\end{scope}
\begin{scope}[>=latex,xshift=230pt,yshift=-345pt]
{\color{red}
\fill (5,2.5) circle(2pt) coordinate(A) node[below]{$1$};
\fill (4,1.5) circle(2pt) coordinate(B) node[above]{$2$};
\fill (3,2.5) circle(2pt) coordinate(C) node[below]{$4$};
\fill (3,0.5) circle(2pt) coordinate(D) node[above]{$3$};
\fill (2,1.5) circle(2pt) coordinate(E) node[above]{$5$};
\fill (1,2.5) circle(2pt) coordinate(F) node[below]{$6$};
\draw [-] (0,3) to [out = 0, in = 135] (F);
\draw [-] (F) -- (D); 
\draw [->] (D) to [out = -45, in = 180] (6,0);
\draw [-] (0,2) to [out = 0, in = -135] (F);
\draw [-] (F) to [out = 45, in = 135] (C); 
\draw [-] (C) -- (B);
\draw [->] (B) to [out = -45, in = 180] (6,1);
\draw [-] (0,1) to [out = 0, in = -135] (E); 
\draw [-] (E) -- (C);
\draw [-] (C) to [out = 45, in = 135] (A);
\draw [->] (A) to [out = -45, in = 180] (6,2);
\draw [-] (0,0) to [out = 0, in = -135] (D); 
\draw [-] (D) -- (A);
\draw [->] (A) to [out = 45, in = 180] (6,3);
}
\coordinate (P1) at (3,-0.2);
\coordinate (P2) at (3,-0.7);
\draw[->] (P1) -- (P2);
\draw (3,-0.5) node[right]{$\sigma_{4,7}$};
%
%{\color{blue}
\draw (2,0.5) circle(2pt) coordinate(6) node[below]{$6$};%{$v_1^1$};
\draw (3,1.5) circle(2pt) coordinate(7) node[below]{$7$};%{$v_2^1$}; %B1
\draw (4,2.5) circle(2pt) coordinate(8) node[above]{$8$};%{$v_3^1$}; %C1
\draw (4,0.5) circle(2pt) coordinate(9) node[below]{$9$};%{$v_4^1$};%C1'
\draw (1,1.5) circle(2pt) coordinate(3) node[above]{$3$};%{$v_1^2$};%D1
\draw (2,2.5) circle(2pt) coordinate(4) node[above]{$4$};%{$v_2^2$};
\draw (5,1.5) circle(2pt) coordinate(5) node[above]{$5$};%{$v_3^2$};
\draw (0,2.5) circle(2pt) coordinate(1) node[above]{$1$};%{$v_1^3$};
\draw (6,2.5) circle(2pt) coordinate(2) node[above]{$2$};%{$v_2^3$};%H1
\qarrow{1}{4}
\qarrow{4}{8}
\qarrow{8}{2}
\qarrow{3}{7}
\qarrow{7}{5}
\qarrow{6}{9}
\qdarrow{2}{5}
\qarrow{5}{8}
\qarrow{8}{7}
\qarrow{7}{4}
\qarrow{4}{3}
\qdarrow{3}{1}
\qdarrow{5}{9}
\qarrow{9}{7}
\qarrow{7}{6}
\qdarrow{6}{3}
%}
\end{scope}
\begin{scope}[>=latex,xshift=230pt,yshift=-460pt]
{\color{red}
\fill (5,2.5) circle(2pt) coordinate(A) node[below]{$1$};
\fill (4,1.5) circle(2pt) coordinate(B) node[above]{$2$};
\fill (3,2.5) circle(2pt) coordinate(C) node[below]{$4$};
\fill (3,0.5) circle(2pt) coordinate(D) node[above]{$3$};
\fill (2,1.5) circle(2pt) coordinate(E) node[above]{$5$};
\fill (1,2.5) circle(2pt) coordinate(F) node[below]{$6$};
\draw [-] (0,3) to [out = 0, in = 135] (F);
\draw [-] (F) -- (D); 
\draw [->] (D) to [out = -45, in = 180] (6,0);
\draw [-] (0,2) to [out = 0, in = -135] (F);
\draw [-] (F) to [out = 45, in = 135] (C); 
\draw [-] (C) -- (B);
\draw [->] (B) to [out = -45, in = 180] (6,1);
\draw [-] (0,1) to [out = 0, in = -135] (E); 
\draw [-] (E) -- (C);
\draw [-] (C) to [out = 45, in = 135] (A);
\draw [->] (A) to [out = -45, in = 180] (6,2);
\draw [-] (0,0) to [out = 0, in = -135] (D); 
\draw [-] (D) -- (A);
\draw [->] (A) to [out = 45, in = 180] (6,3);
}
%
%{\color{blue}
\draw (2,0.5) circle(2pt) coordinate(6) node[below]{$6$};%{$v_1^1$};
\draw (3,1.5) circle(2pt) coordinate(7) node[below]{$4$};%{$v_2^1$}; 
\draw (4,2.5) circle(2pt) coordinate(8) node[above]{$8$};%{$v_3^1$}; 
\draw (4,0.5) circle(2pt) coordinate(9) node[below]{$9$};%{$v_4^1$};
\draw (1,1.5) circle(2pt) coordinate(3) node[above]{$3$};%{$v_1^2$};
\draw (2,2.5) circle(2pt) coordinate(4) node[above]{$7$};%{$v_2^2$};
\draw (5,1.5) circle(2pt) coordinate(5) node[above]{$5$};%{$v_3^2$};
\draw (0,2.5) circle(2pt) coordinate(1) node[above]{$1$};%{$v_1^3$};
\draw (6,2.5) circle(2pt) coordinate(2) node[above]{$2$};%{$v_2^3$};
\qarrow{1}{4}
\qarrow{4}{8}
\qarrow{8}{2}
\qarrow{3}{7}
\qarrow{7}{5}
\qarrow{6}{9}
\qdarrow{2}{5}
\qdarrow{5}{9}
\qdarrow{6}{3}
\qdarrow{3}{1}
\qarrow{9}{7}
\qarrow{7}{6}
\qarrow{5}{8}
\qarrow{8}{7}
\qarrow{7}{4}
\qarrow{4}{3}
%}
\end{scope}
\end{tikzpicture}
}
\]
\caption{Tetrahedron transformation. The $R_{ijk}$ transforms the wiring diagrams, 
and the $\mu_i$ and the $\sigma_{ij}$ transform the 
quivers. (Note that the $\sigma_{ij}$ act on wiring diagrams trivially.)}
\label{fig:tetrahedron-eq}
\end{figure}
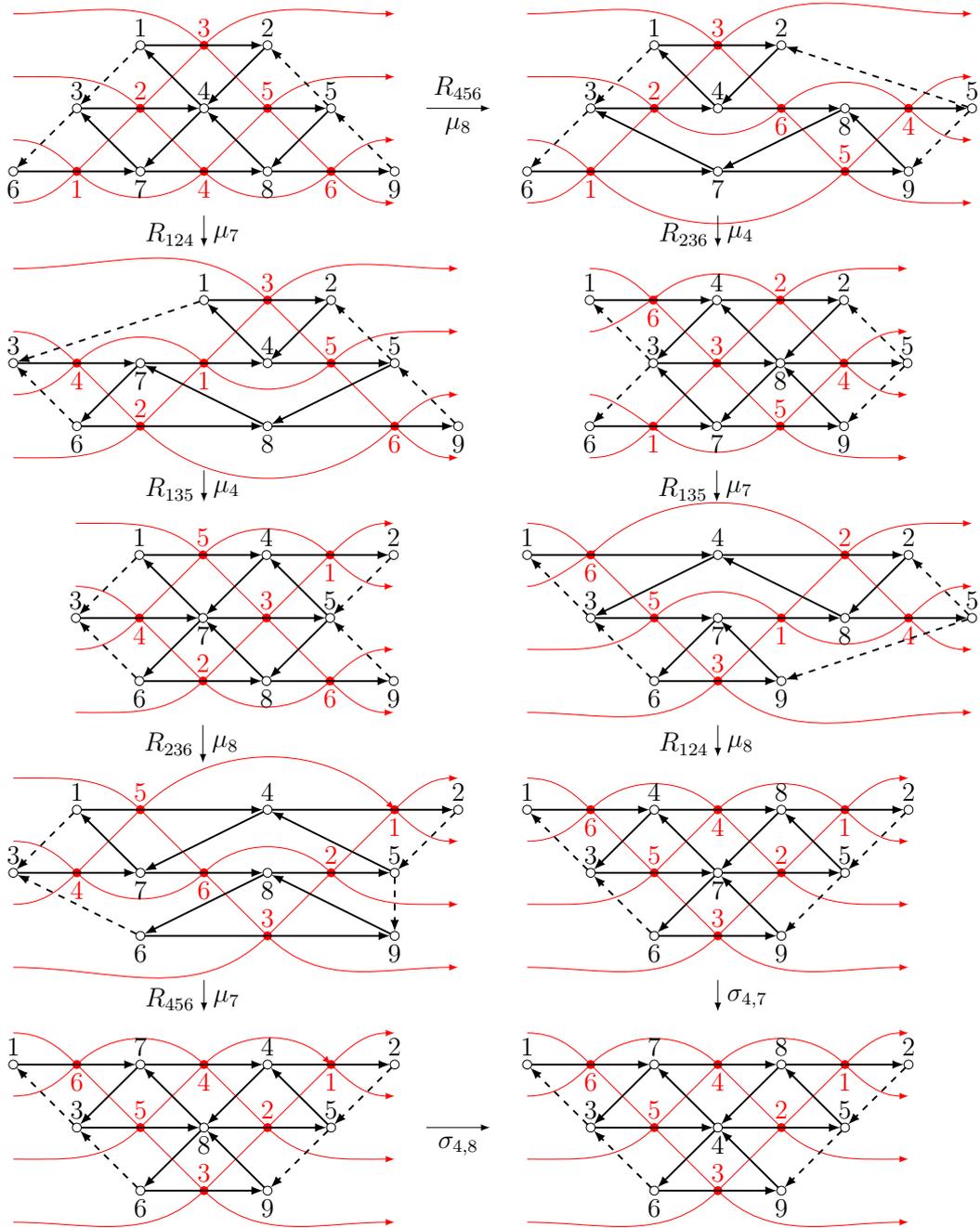

As studied in \cite{SY22}, this transformation is accompanied by cluster mutation sequences. 
We associate the FG quivers $J_{123121}$ and $J_{321323}$ to the initial and the final reduced words as follows:
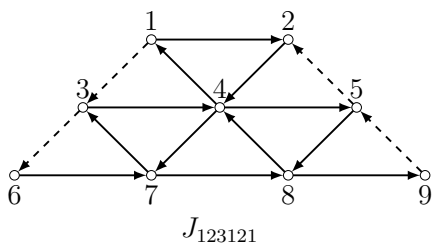
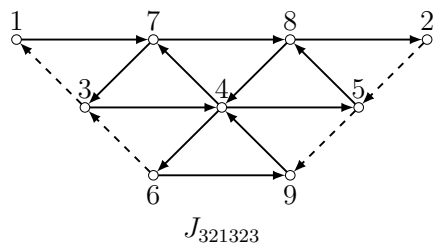
\[
\scalebox{0.9}{
\begin{tikzpicture}
\begin{scope}[>=latex]
\draw (0,0.5) circle(2pt) coordinate(6) node[below]{$6$};
\draw (2,0.5) circle(2pt) coordinate(7) node[below]{$7$};
\draw (4,0.5) circle(2pt) coordinate(8) node[below]{$8$};
\draw (6,0.5) circle(2pt) coordinate(9) node[below]{$9$};
\draw (1,1.5) circle(2pt) coordinate(3) node[above]{$3$};
\draw (3,1.5) circle(2pt) coordinate(4) node[above]{$4$};
\draw (5,1.5) circle(2pt) coordinate(5) node[above]{$5$};
\draw (2,2.5) circle(2pt) coordinate(1) node[above]{$1$};
\draw (4,2.5) circle(2pt) coordinate(2) node[above]{$2$};
\qarrow{1}{2}
\qarrow{3}{4}
\qarrow{4}{5}
\qarrow{6}{7}
\qarrow{7}{8}
\qarrow{8}{9}
\qdarrow{9}{5}
\qdarrow{5}{2}
\qdarrow{1}{3}
\qdarrow{3}{6}
\qarrow{5}{8}
\qarrow{8}{4}
\qarrow{4}{7}
\qarrow{7}{3}
\qarrow{2}{4}
\qarrow{4}{1}
\draw (3,-0.3) node{$J_{123121}$};
\end{scope}
\begin{scope}[>=latex,xshift=200pt]
\draw (2,0.5) circle(2pt) coordinate(6) node[below]{$6$};%{$v_1^1$};
\draw (3,1.5) circle(2pt) coordinate(7) node[above]{$4$};%{$v_2^1$}; 
\draw (4,2.5) circle(2pt) coordinate(8) node[above]{$8$};%{$v_3^1$}; 
\draw (4,0.5) circle(2pt) coordinate(9) node[below]{$9$};%{$v_4^1$};
\draw (1,1.5) circle(2pt) coordinate(3) node[above]{$3$};%{$v_1^2$};
\draw (2,2.5) circle(2pt) coordinate(4) node[above]{$7$};%{$v_2^2$};
\draw (5,1.5) circle(2pt) coordinate(5) node[above]{$5$};%{$v_3^2$};
\draw (0,2.5) circle(2pt) coordinate(1) node[above]{$1$};%{$v_1^3$};
\draw (6,2.5) circle(2pt) coordinate(2) node[above]{$2$};%{$v_2^3$};
\qarrow{1}{4}
\qarrow{4}{8}
\qarrow{8}{2}
\qarrow{3}{7}
\qarrow{7}{5}
\qarrow{6}{9}
\qdarrow{2}{5}
\qdarrow{5}{9}
\qdarrow{6}{3}
\qdarrow{3}{1}
\qarrow{9}{7}
\qarrow{7}{6}
\qarrow{5}{8}
\qarrow{8}{7}
\qarrow{7}{4}
\qarrow{4}{3}
\draw (3,-0.3) node{$J_{321323}$};
\end{scope}
\end{tikzpicture}
}
\] 
Figure \ref{fig:tetrahedron-eq} shows the correspondence between the quivers and the wiring diagrams. One sees that the transformation \eqref{quiver-d:A2} is embedded therein, where a pair of dashed arrows $\dashrightarrow$ with 
the same (resp. opposite) direction between two vertices 
is regarded as an ordinary arrow $\longrightarrow$ (resp. no arrow).
For the quiver $J_{123121}$ we set $I = \{1,2,\ldots,9 \}$.
% and $I_0 = \{1,2,3,5,6,9\}$. 

\begin{lem}(cf. \cite[Proposition 3.7]{SY22})\label{lem:R-trop}
Let $(J_{123121},y)$ be a tropical $y$-seed for $J_{123121}$.
It holds that 
\begin{align}
%  \label{eq:R-J-id}
%  &\sigma_{4,7}\mu_8 \mu_7 \mu_4 \mu_8 (J_{123121}) =  \sigma_{4,8} \mu_7 \mu_8% \mu_4 \mu_7(J_{123121}) = J_{321323},
%\\ 
  \label{eq:R-trop-id}
  \sigma_{4,7}\mu_8 \mu_7 \mu_4 \mu_8 (J_{123121},y) 
  = \sigma_{4,8} \mu_7 \mu_8 \mu_4 \mu_7 (J_{123121},y),
\end{align}
where $\sigma_{4,7}, \sigma_{4,8} \in \mathfrak{S}_9$ act on the vertex set of the quiver. 
%and $y'$ is a tuple of $y$-variables.
The tropical sign-sequences of the two mutation sequences $\mu_8 \mu_7 \mu_4 \mu_8$ and $\mu_7 \mu_8 \mu_4 \mu_7$ 
are all positive, i.e., $(+,+,+,+)$.
We also have $\sigma_{4,7}\mu_8 \mu_7 \mu_4 \mu_8 (J_{123121}) =  \sigma_{4,8} \mu_7 \mu_8 \mu_4 \mu_7(J_{123121}) = J_{321323}$.
\end{lem}

\begin{proof}
The tropical $y$-variables are transformed as follows.
\begin{align*}
\mathrm{LHS: ~}&(y_1,y_2,y_3,y_4,y_5,y_6,y_7,\underline{y_8},y_9) 
%\displaybreak[0] 
\\ 
&\stackrel{\mu_8}{\mapsto}
(y_1,y_2,y_3,\underline{y_4},y_5 y_8,y_6,y_7 y_8,y_8^{-1},y_9)
%\displaybreak[0] 
\\
&\stackrel{\mu_4}{\mapsto}
(y_1,y_2 y_4,y_3 y_4,y_4^{-1},y_5 y_8,y_6,\underline{y_7 y_8},y_8^{-1},y_9)
%\displaybreak[0] 
\\
&\stackrel{\mu_7}{\mapsto}
(y_1,y_2 y_4,y_3 y_4,y_4^{-1},y_5 y_8,y_6 y_7 y_8,(y_7 y_8)^{-1},\underline{y_7},y_9)
\displaybreak[0] 
\\
&\stackrel{\mu_8}{\mapsto}
(y_1,y_2 y_4 y_7,y_3 y_4,y_4^{-1},y_5 y_8,y_6 y_7 y_8,y_8^{-1},y_7^{-1}, y_9)
\displaybreak[0] 
\\
&\stackrel{\sigma_{4,7}}{\mapsto}
(y_1,y_2 y_4 y_7,y_3 y_4,y_8^{-1},y_5 y_8,y_6 y_7 y_8,y_4^{-1},y_7^{-1},y_9)
%\displaybreak[0]
\\[1mm]
\mathrm{RHS: ~}&(y_1,y_2,y_3,y_4,y_5,y_6,\underline{y_7},y_8,y_9) 
\displaybreak[0] 
\\
&\stackrel{\mu_7}{\mapsto}
(y_1,y_2,y_3,\underline{y_4 y_7},y_5,y_6 y_7,y_7^{-1},y_8,y_9)
\displaybreak[0] \\
&\stackrel{\mu_4}{\mapsto}(y_1 ,y_2 y_4 y_7,y_3,(y_4 y_7)^{-1},y_5,y_6 y_7,y_4,\underline{y_8},y_9)
\displaybreak[0] \\
&\stackrel{\mu_8}{\mapsto}(y_1,y_2 y_4 y_7,y_3,(y_4 y_7)^{-1},y_5 y_8,y_6 y_7 y_8,\underline{y_4},y_8^{-1},y_9)
\displaybreak[0] \\
&\stackrel{\mu_7}{\mapsto}(y_1,y_2 y_4 y_7,y_3 y_4,y_7^{-1},y_5 y_8,y_6 y_7 y_8,y_4^{-1},y_8^{-1},y_9)
\displaybreak[0] \\
&\stackrel{\sigma_{4,8}}{\mapsto}(y_1,y_2 y_4 y_7,y_3 y_4,y_8^{-1},y_5 y_8,y_6 y_7 y_8,y_4^{-1},y_7^{-1},y_9),
\end{align*}
which proves \eqref{eq:R-trop-id}. The tropical sign-sequences of these mutation sequences are obtained from the underlined $y$-variables.
One sees that the final quiver is $J_{321323}$ from Figure \ref{fig:tetrahedron-eq}.
\end{proof}

\begin{remark}
The second claim of the above lemma follows from the fact that 
the two mutation sequences are known as the {\em maximal green sequences} \cite{Ke11} for a directed closed path $4 \to 7 \to 8 \to 4$. 
\end{remark}

Due to Theorem \ref{thm:period} and Theorem \ref{thm:id-mono-qdilog} we obtain the following proposition as a corollary of Lemma \ref{lem:R-trop}.
\begin{prop}
Let $(J_{123121},Y)$ be a quantum $y$-seed for $J_{123121}$. It holds that
\begin{align}\label{eq:R-q-id}
  \sigma_{4,7}\mu_8 \mu_7 \mu_4 \mu_8 (J_{123121},Y) 
  = \sigma_{4,8} \mu_7 \mu_8 \mu_4 \mu_7 (J_{123121},Y).
\end{align}
In particular, we have identities
\begin{align}
\label{R-mono-id}
\tau_{8,+} \tau_{4,+} \tau_{7,+} \tau_{8,+} \sigma_{4,7}
= \tau_{7,+} \tau_{4,+} \tau_{8,+} \tau_{7,+} \sigma_{4,8}
\end{align}
as a morphism from $\mathcal{T}(J_{321323})$ to $\mathcal{T}(J_{123121})$,
and  
\begin{align}
\label{R-qdilog-id}
&\Psi_q(\rY^{e_8}) \Psi_q(\rY^{e_4}) \Psi_q(\rY^{e_7+e_8}) \Psi_q(\rY^{e_7})
= \Psi_q(\rY^{e_7}) \Psi_q(\rY^{e_4+e_7}) \Psi_q(\rY^{e_8}) \Psi_q(\rY^{e_4})
\end{align}
in $\hat{\mathbb{A}}^{\! \times}(J_{123121})$.
\end{prop}

Let $S := \{1,2,\ldots,6 \}$ be the set of crossings in the wiring diagrams in Figure \ref{fig:tetrahedron-eq}.
Extending the setting in \S \ref{subsec:qWeyl-R}, 
we assign a pair $(p_i,u_i)$ satisfying (\ref{pu}) to the crossing $i \in S$ of the wiring diagrams.
Let $\mathcal{W}(A_3)$ be the algebra over $\C$ generated by the $q$-commuting $q$-Weyl pairs 
$e^{\pm p_i}, e^{\pm u_i}~(i \in S)$.
Let $N(A_3)$ be a group generated by 
\begin{align}\label{N(A3)gen}
e^{\pm \frac{1}{\hbar}p_i u_j}, ~ 
e^{\pm \frac{a}{\hbar}u_i},~ 
(a \in \C, ~i,j \in S; i \neq j), ~b \in \C^\times,
\end{align}
where the multiplication is defined in terms of (generalized) BCH formula. 
It is well defined thanks to the grading by $\hbar^{-1}$ in \eqref{N(A3)gen}.
The symmetric group $\mathfrak{S}_6$ generated by permutations $\rho_{ij} ~(i,j \in S)$ acts on the indices of $q$-Weyl pairs. This acts on $N(A_3)$ by adjoint action, and a semidirect product $N(A_3) \rtimes \mathfrak{S}_6$ acts on $\mathcal{W}(A_3)$ by adjoint action. 
We use the complex parameters $\kappa_i~(i \in S)$ and the notation \eqref{para}.

Corresponding to the transformation $R_{ijk}$ of the wiring diagrams, the morphism $\pi_{123}$ \eqref{eq:R-pi} is naturally extended to the isomorphism $\pi_{ijk}$ of $\mathcal{W}(A_3)$ given by
\begin{align}\label{eq:R-pi-general}
\pi_{ijk}:
\begin{cases}
p_i \mapsto p_i+ \lambda_{jk},
\quad
p_j  \mapsto p_k+p_i,
\quad
p_k \mapsto p_j-p_i-\lambda_{jk}, 
\\
u_i  \mapsto u_i+u_j-u_k,
\quad 
u_j \mapsto u_k,
\quad 
u_k  \mapsto u_j,
\end{cases}
\end{align} 
in the sense of exponentials. We also define 
\begin{align}\label{eq:R-Pop-general}
P_{ijk}: = \rho_{jk}\, e^{\frac{1}{\hbar}p_i(u_k-u_j)} 
e^{\frac{\lambda_{jk}}{\hbar}(u_k-u_i)}
\in N(A_3) \rtimes \mathfrak{S}_6
\end{align} 
from \eqref{R-Pop}. 
The embedding $\phi$ \eqref{eq:phi_kappa} is also naturally extended to that from 
$\mathcal{Y}(J_{123121})$ to Frac$\mathcal{W}(A_3)$ by the following rule, where red vertices are crossings in the wiring diagram. 
\begin{align}\label{fig:A_3-phi}
\begin{tikzpicture}
\begin{scope}[>=latex,xshift=0pt]
\draw (2,1) circle(2pt) coordinate(A) node[above]{$a$};
\draw (4,1) circle(2pt) coordinate(B) node[above]{$b$};
\draw (1,0) circle(2pt) coordinate(C) node[below]{$c$};
\draw (3,0) circle(2pt) coordinate(D) node[below]{$d$};
\draw (5,0) circle(2pt) coordinate(E) node[below]{$e$};
\draw (2,-1) circle(2pt) coordinate(F) node[below]{$f$};
\draw (4,-1) circle(2pt) coordinate(G) node[below]{$g$};
\qarrow{A}{B}
\qarrow{C}{D}
\qarrow{D}{E}
\qarrow{F}{G}
\qarrow{B}{D}
\qarrow{D}{A}
\qarrow{G}{D}
\qarrow{D}{F}
\draw (0.3,-2) node[right]{$\phi: Y_d \mapsto \kappa_j \kappa_k^{-1} e^{p_j+u_j+p_k-u_k-p_i-p_l}$};
{\color{red}
\fill (3,1) circle(2pt) node[above]{$i$};
\fill (2,0) circle(2pt) node[above]{$j$};
\fill (4,0) circle(2pt) node[above]{$k$};
\fill (3,-1) circle(2pt) node[below]{$l$};
}
\end{scope}
\end{tikzpicture}
\end{align}
Explicitly, from the top left diagram in Figure \ref{fig:tetrahedron-eq} we have 
\begin{align}\label{eq:phi-A3}
\phi: 
\begin{cases}
Y_1 \mapsto \kappa_3^{-1} e^{p_3-u_3-p_2}, ~~ 
Y_2 \mapsto \kappa_3 e^{p_3+u_3-p_5}, ~~
Y_3 \mapsto \kappa_2^{-1} e^{p_2-u_2-p_1}, 
\\
Y_4 \mapsto \kappa_2 \kappa_5^{-1} e^{p_2+u_2+p_5-u_5-p_3-p_4}, ~~ 
Y_5 \mapsto \kappa_5 e^{p_5+u_5-p_6}, ~~
Y_6 \mapsto \kappa_1^{-1} e^{p_1-u_1},
\\ 
Y_7 \mapsto \kappa_1 \kappa_4^{-1} e^{p_1+u_1+p_4-u_4-p_2}, ~~
Y_8 \mapsto \kappa_4 \kappa_6^{-1}e^{p_4+u_4+p_6-u_6-p_5}, ~~
Y_9 \mapsto \kappa_6 e^{p_6+u_6}.
\end{cases}
\end{align}
These formulas can be seen as arising informally from the application of the rule in \eqref{fig:A_3-phi} 
when some surrounding vertices are absent or when certain arrows are dashed.

Let $\mathcal{W}_{\mathbb{A}}(J_{123121}) \subset \mathcal{W}(A_3)$ be the image of $\mathbb{A}(J_{123121})$ by $\phi$, and $\hat{\mathcal{W}}_{\mathbb{A}}(J_{123121})$ be the completion of $\mathcal{W}_{\mathbb{A}}(J_{123121})$ with respect to the ideal generated by $\phi(Y_i) ~(i \in I)$. 
The map $\phi$ induces the map from $\hat{\mathbb{A}}(J_{123121})$ to $\hat{\mathcal{W}}_{\mathbb{A}}(J_{123121})$. We also denote this induced map as $\phi$.
Let $\hat{\mathcal{W}}_{\mathbb{A}}^\times(J_{123121})$ be the image of $\hat{\mathbb{A}}^{\! \times}(J_{123121}) \subset \hat{\mathbb{A}}(J_{123121})$ by $\phi$. Recall the definitions in the last paragraph of \S \ref{subsec:mutation}. 
In the remainder of this section, for simplicity we abbreviate the dependence on $J_{123121}$.
The relations among the noncommuting algebras (including fields and groups) are summarized as follows:
\begin{align}\label{A-W-relation}
\xymatrix{
\mathcal{Y} \ar[d]^{\phi} & \supset & \mathbb{A} \ar[d]^{\phi|_{\mathbb{A}}} & \subset & \hat{\mathbb{A}}  \ar[d]^{\phi} & \supset & \hat{\mathbb{A}}^{\! \times} \ar[d]^{\phi|_{\hat{\mathbb{A}}^{\! \times}}}
\\
\mathrm{Frac}\mathcal{W}(A_3) & \supset & \mathcal{W}_{\mathbb{A}} & \subset & \hat{\mathcal{W}}_{\mathbb{A}} & \supset & \hat{\mathcal{W}}_{\mathbb{A}}^\times
}
\end{align}

Let $G = G(J_{123121})$ be a subgroup of $N(A_3) \rtimes \mathfrak{S}_6$ defined by
\begin{align}\label{eq:G-A3}
G(J_{123121}) = \{ w \in N(A_3) \rtimes \mathfrak{S}_6; ~ \mathrm{Ad}(w)(\phi(\mathcal{Y})) \in \phi(\mathcal{Y})\}.
\end{align}
This group acts on $\hat{\mathcal{W}}_{\mathbb{A}}$ by adjoint action,
and this action is restricted to that on $\mathcal{W}_{\mathbb{A}}$
and on $\hat{\mathcal{W}}_{\mathbb{A}}^\times$.

\begin{example}
For $Y_1 \in \mathbb{A}$, one has $\Psi_q(Y_1) \in \hat{\mathbb{A}}^{\! \times}$ and $\phi(\Psi_q(Y_1)) = \Psi_q(\phi(Y_1)) = \Psi_q(\kappa_3^{-1} \,e^{p_3-u_3-p_2}) \in \hat{\mathcal{W}}_\mathbb{A}^{\times}$.
By the definition, $P_{ijk}$ \eqref{eq:R-Pop-general} belongs to $G$ when $(i,j,k)$ 
is $(4,5,6)$, $(2,3,6)$, $(1,3,5)$ or $(1,2,4)$ appearing in the tetrahedron relation \eqref{eq:tetra-R-eq}.
On the other hand, $e^{\frac{1}{\hbar}p_1 u_2} \in N(A_3)$ does not belong to $G$, since Ad$(e^{\frac{1}{\hbar}p_1 u_2})(\phi(Y_1)) = e^{p_1} \phi(Y_1)$ does not belong to $\phi(\mathcal{Y})$ because 
there is no $\alpha \in \Z^I$ such that $\phi(\rY^\alpha) = a \,e^{p_3-u_3-p_2+p_1}$ for any $a \in \C^\times$.    
\end{example}

The following lemma is easily proved by direct calculation.
    
\begin{lem}\label{lem:pi-tetra}
The morphisms $\pi_{ijk}$ \eqref{eq:R-pi-general} on $\mathcal{W}(A_3)$ satisfy the tetrahedron equation:
\begin{equation}\label{pite}
\pi_{456} \pi_{236} \pi_{135} \pi_{124} = \pi_{124} \pi_{135} \pi_{236} \pi_{456}.
\end{equation}
\end{lem}

In view of \eqref{eq:R-decomp}, \eqref{eq:phi_kappa} and (\ref{R-Pop}), we
introduce $\mathcal{R}_{ijk} = \mathcal{R}(\lambda_i,\lambda_j,\lambda_k)_{ijk}$ as
\begin{align}\label{eq:R-ad}
\mathcal{R}_{ijk} = \Psi_q( \e^{p_i+u_i+p_k-u_k-p_j+\lambda_{ik}}) 
\rho_{jk}\, e^{\frac{1}{\hbar}p_i(u_k-u_j)} 
e^{\frac{\lambda_{jk}}{\hbar}(u_k-u_i)},
\end{align}
where the latter part is $P_{ijk}$ in \eqref{eq:R-Pop-general}.
Note that $\mathcal{R}_{ijk}$ depends on the parameters $\lambda_i$ only through 
their differences $\lambda_{ik}$ and $\lambda_{jk}$.  
Now  we state the main result of this section.
\begin{thm}\label{thm:R-tetra}
The operator $P_{ijk}$  \eqref{eq:R-Pop-general} 
satisfies the tetrahedron equation in $N(A_3) \rtimes \mathfrak{S}_6$:
\begin{align}\label{eq:R-monoad-id}
P_{456} P_{236} P_{135} P_{124} = P_{124} P_{135} P_{236} P_{456}.
\end{align}
The operator $\mathcal{R}_{ijk}$ 
\eqref{eq:R-ad} satisfies the tetrahedron equation in $\hat{\mathcal{W}}_{\mathbb{A}}^\times \rtimes G$:
\begin{align}\label{eq:R-full-id}
\begin{split}
&\mathcal{R}(\lambda_4,\lambda_5,\lambda_6)_{456} \mathcal{R}(\lambda_2,\lambda_3,\lambda_6)_{236} \mathcal{R}(\lambda_1,\lambda_3,\lambda_5)_{135} \mathcal{R}(\lambda_1,\lambda_2,\lambda_4)_{124} 
\\
& \quad = \mathcal{R}(\lambda_1,\lambda_2,\lambda_4)_{124} \mathcal{R}(\lambda_1,\lambda_3,\lambda_5)_{135} \mathcal{R}(\lambda_2,\lambda_3,\lambda_6)_{236} \mathcal{R}(\lambda_4,\lambda_5,\lambda_6)_{456}.
\end{split}
\end{align}
\end{thm} 

\begin{proof}
We prove \eqref{eq:R-monoad-id} by direct calculations.
For simplicity we demonstrate the calculation when $\lambda_j = 0$ for all $j$.
By using the action of $\mathfrak{S}_3$ on $N(A_3)$, each side of \eqref{eq:R-monoad-id} becomes 
\begin{align}
\begin{split}\label{eq:P-right}
P_{124} P_{135} P_{236} P_{456}
&= \rho_{24}\, e^{\frac{1}{\hbar}p_1(u_4-u_2)}
\rho_{35}\, e^{\frac{1}{\hbar}p_1(u_5-u_3)}
\rho_{36}\, e^{\frac{1}{\hbar}p_2(u_6-u_3)}
\rho_{56}\, e^{\frac{1}{\hbar}p_4(u_6-u_5)}
\\ 
&= e^{\frac{1}{\hbar}p_1(u_2-u_4)} e^{\frac{1}{\hbar}p_1(u_3-u_5)}
 e^{\frac{1}{\hbar}p_4(u_5-u_6)} e^{\frac{1}{\hbar}p_2(u_3-u_5)}
\rho_{24} \rho_{35} \rho_{36} \rho_{56},
\end{split}
\\
\begin{split}\label{eq:P-left}
P_{456} P_{236} P_{135} P_{124} 
&= \rho_{56}\, e^{\frac{1}{\hbar}p_4(u_6-u_5)}
\rho_{36}\, e^{\frac{1}{\hbar}p_2(u_6-u_3)}
\rho_{35}\, e^{\frac{1}{\hbar}p_1(u_5-u_3)}
\rho_{24}\, e^{\frac{1}{\hbar}p_1(u_4-u_2)}
\\ 
&= e^{\frac{1}{\hbar}p_4(u_5-u_6)} e^{\frac{1}{\hbar}p_2(u_3-u_5)}
e^{\frac{1}{\hbar}p_1(u_5-u_6)} e^{\frac{1}{\hbar}p_1(u_2-u_4)}
\rho_{56} \rho_{36} \rho_{35} \rho_{24}.
\end{split}
\end{align}
The exponential part of \eqref{eq:P-right} is calculated by applying the BCH formula as
\begin{align*}
&e^{\frac{1}{\hbar}p_1(u_2-u_4)} e^{\frac{1}{\hbar}p_1(u_3-u_5)}
 e^{\frac{1}{\hbar}p_4(u_5-u_6)} e^{\frac{1}{\hbar}p_2(u_3-u_5)}
\\ \quad
&= e^{\frac{1}{\hbar}p_1(u_3-u_5)} \underline{e^{\frac{1}{\hbar}p_1(u_2-u_4)} 
 e^{\frac{1}{\hbar}p_4(u_5-u_6)}} e^{\frac{1}{\hbar}p_2(u_3-u_5)}
\\ \quad
&\stackrel{\mathrm{BCH}}{=} e^{\frac{1}{\hbar}p_1(u_3-u_5)} e^{\frac{1}{\hbar}p_1(u_5-u_6)}
e^{\frac{1}{\hbar}p_4(u_5-u_6)} \underline{e^{\frac{1}{\hbar}p_1(u_2-u_4)}
e^{\frac{1}{\hbar}p_2(u_3-u_5)}}
\\ \quad
&\stackrel{\mathrm{BCH}}{=}
\underline{e^{\frac{1}{\hbar}p_1(u_3-u_5)} e^{\frac{1}{\hbar}p_1(u_5-u_6)}}
e^{\frac{1}{\hbar}p_4(u_5-u_6)}
\underline{e^{\frac{1}{\hbar}p_1(u_5-u_3)}}e^{\frac{1}{\hbar}p_2(u_3-u_5)}
e^{\frac{1}{\hbar}p_1(u_2-u_4)}
\\ \quad
&= e^{\frac{1}{\hbar}p_4(u_5-u_6)} e^{\frac{1}{\hbar}p_2(u_3-u_5)}
e^{\frac{1}{\hbar}p_1(u_5-u_6)} e^{\frac{1}{\hbar}p_1(u_2-u_4)},
\end{align*}
which coincides with the exponential part of \eqref{eq:P-left}.
It is easily checked that 
$$
  \rho_{24} \rho_{35} \rho_{36} \rho_{56}
  = \rho_{56} \rho_{36} \rho_{35} \rho_{24},
$$
by using the braid relation $\rho_{ij} \rho_{ik} \rho_{jk} = \rho_{jk} \rho_{ik} \rho_{ij}$,
hence \eqref{eq:R-monoad-id} is proved. 

Multiplying \eqref{R-qdilog-id} and \eqref{eq:R-monoad-id} side by side, we obtain an identity in 
$\hat{\mathcal{W}}_{\mathbb{A}}^\times \rtimes G$,
\begin{align}\label{eq:Psi-P-before}
\begin{split}
&\Psi_q(Y_8) \Psi_q(Y_4) \Psi_q(q^{-1} \,Y_7 Y_8) \Psi_q(Y_7) 
P_{456} P_{236} P_{135} P_{124}
\\
&\quad =
\Psi_q(Y_7) \Psi_q(q^{-1} \,Y_4 Y_7) \Psi_q(Y_8) \Psi_q(Y_4)
P_{124} P_{135} P_{236} P_{456},
\end{split}
\end{align}
where all $Y_i$ are understood to be $\phi(Y_i)$.
The equality \eqref{eq:R-full-id} is derived from this by moving $P_{ijk}$ to the left appropriately, using the action of $G$ on $\hat{\mathcal{W}}_{\mathbb{A}}^{\times}$.
Let us demonstrate this computation for the LHS of \eqref{eq:Psi-P-before}. 
We rewrite the LHS as follows.
\begin{align*}
\Psi_q(Y_8) P_{456} &\cdot \underline{P_{456}^{-1} \Psi_q(Y_4)P_{456} P_{236} }_{(a)}
\\
&\cdot \underline{P_{236}^{-1} P_{456}^{-1} \Psi_q(q^{-1} \,Y_7 Y_8)P_{456} P_{236} P_{135}}_{(b)} \cdot \underline{P_{135}^{-1} P_{236}^{-1} P_{456}^{-1} \Psi_q(Y_7) 
P_{456} P_{236} P_{135} P_{124}}_{(c)},
\end{align*}
where the first component $\Psi_q(Y_8) P_{456}$ is $\mathcal{R}_{456}$ \eqref{eq:R-ad}.
By using the map $\pi_{ijk}^{-1} = \text{Ad}(P_{ijk}^{-1})$ \eqref{eq:R-pi}\footnote{We obtain $\pi_{ijk}^{-1}$ by replacing $\lambda_{jk}$ with $\lambda_{kj}$ in \eqref{eq:R-pi-general}.}, components (a)--(c) are computed as
\begin{align*}
&(a): P_{456}^{-1} \Psi_q(Y_4)P_{456} P_{236} = P_{456}^{-1} \Psi_q(\kappa_2 \kappa_5^{-1} e^{p_2+u_2+p_5-u_5-p_3-p_4})P_{456}P_{236} 
\displaybreak[0]
\\
& \qquad = \Psi_q(\kappa_2 \kappa_6^{-1}e^{p_2+u_2+p_6-u_6-p_3}) P_{236} = \mathcal{R}_{236}, 
\displaybreak[0]
\\
&(b): P_{236}^{-1} P_{456}^{-1} \Psi_q(q^{-1} \,Y_7 Y_8)P_{456} P_{236} P_{135}
\displaybreak[0]
\\
&\qquad = P_{236}^{-1} P_{456}^{-1} \Psi_q(q^{-1}\kappa_1 \kappa_6^{-1}e^{p_1+u_1+p_4-u_4-p_2} e^{p_4+u_4+p_6-u_6-p_5})P_{456} P_{236} P_{135}
\displaybreak[0]
\\
&\qquad = P_{236}^{-1} \Psi_q(\kappa_1 \kappa_5^{-1}e^{p_1+u_1+p_5-u_5-p_2-p_6}) P_{236}P_{135}
= \Psi_q(\kappa_1 \kappa_5^{-1}e^{p_1+u_1+p_5-u_5-p_3})P_{135} = \mathcal{R}_{135},
\displaybreak[0]
\\
&(c): P_{135}^{-1} P_{236}^{-1} P_{456}^{-1} \Psi_q(Y_7) 
P_{456} P_{236} P_{135} P_{124}
\displaybreak[0]
\\
&\qquad = P_{135}^{-1} P_{236}^{-1} P_{456}^{-1} \Psi_q(\kappa_1 \kappa_4^{-1}e^{p_1+u_1+p_4-u_4-p_2}) P_{456} P_{236} P_{135} P_{124}
\displaybreak[0]
\\
&\qquad = P_{135}^{-1} P_{236}^{-1} \Psi_q(\kappa_1 \kappa_4^{-1} \kappa_5 \kappa_6^{-1}e^{p_1+u_1+p_4-u_4-u_5+u_6-p_2})
P_{236} P_{135} P_{124} 
\displaybreak[0]
\\
&\qquad = \cdots = \Psi_q(\kappa_1 \kappa_4^{-1} e^{p_1+u_1+p_4-u_4-p_2}) P_{124}= \mathcal{R}_{124},
\end{align*}
hence the LHS of \eqref{eq:R-full-id} is obtained.
\end{proof}

Our solution \eqref{eq:R-ad} is different from the other single
dilogarithm solutions in \cite[(ii)]{MS97} and \cite{BV15}.

\begin{remark}\label{re:Rcd}
Consider the generalizations of $\pi_{ijk}$ (\ref{eq:R-pi-general}),
$P_{ijk}$ (\ref{eq:R-Pop-general}) and $\mathcal{R}_{ijk}$ (\ref{eq:R-ad}) 
including a parameter $\alpha \in \C$ as follows:
\begin{align}
\pi_{ijk}:&
\begin{cases}
p_i \mapsto p_i+ (1-\alpha) \lambda_{jk},
\quad
p_j  \mapsto p_k+p_i,
\quad
p_k \mapsto p_j-p_i-(1-\alpha)  \lambda_{jk}, 
\\
u_i  \mapsto u_i+u_j-u_k+\alpha \lambda_{jk},
\quad 
u_j \mapsto u_k -\alpha \lambda_{jk},
\quad 
u_k  \mapsto u_j +\alpha  \lambda_{jk},
\end{cases}
\label{pal}
\\
P_{ijk} &= \rho_{jk}\, e^{\frac{1}{\hbar}p_i(u_k-u_j)} 
e^{\frac{(1-\alpha)\lambda_{jk}}{\hbar}(u_k-u_i)}
e^{\frac{\alpha \lambda_{jk}}{\hbar}(p_i-p_j+p_k)},
\label{P_general}
\\
\mathcal{R}_{ijk} &= \Psi_q( \e^{p_i+u_i+p_k-u_k-p_j+\lambda_{ik}})P_{ijk}.
\end{align}
Then Proposition \ref{prop:R-ad}, Lemma \ref{lem:pi-tetra} 
and Theorem \ref{thm:R-tetra} remain valid\footnote{The list \eqref{N(A3)gen} 
should be supplemented with $e^{\pm \frac{a}{\hbar}p_i}$ in this generalization.
A similar caution also applies to Remark \ref{re:ezk}.}.
\end{remark}

\begin{remark}\label{re:sp}
If one employs $\alpha=1$ in the map $\pi_{ijk}$ in (\ref{pal}), 
then the parameters $\kappa_i= e^{\lambda_i}$'s originally introduced 
in (\ref{eq:phi_kappa}) can be removed from everywhere by shifting the canonical variables $u_i$ by $\lambda_i$.
Precisely, let $\mathcal{W}' := \mathcal{W}'(A_3)$ be the $q$-Weyl algebra generated by $e^{\pm p_i'}, e^{\pm u_i'} ~(i \in S)$, and define an isomorphism $\omega: \mathcal{W} := \mathcal{W}(A_3) \to \mathcal{W}'$ of $q$-Weyl algebras by $u_i \mapsto u_i' - \lambda_i$ and $p_i \mapsto p_i'$ in the sense of exponentials. 
Let $P_{ijk} = \rho_{jk}\, e^{\frac{1}{\hbar}p_i(u_k-u_j)} 
e^{\frac{\lambda_{jk}}{\hbar}(p_i-p_j+p_k)}$ be the one  in \eqref{P_general} with $\alpha=1$.
Then, the adjoint action of $P_{ijk}$ on $\mathcal{W}$ induces the adjoint action of 
the parameter-independent operator $P_{ijk}' := \rho_{jk}\, e^{\frac{1}{\hbar}p'_i(u'_k-u'_j)}$ 
on $\mathcal{W}'$ that is compatible with $\omega$, namely the commutativity $\mathrm{Ad}P_{ijk}' \circ \omega = \omega \circ \mathrm{Ad}P_{ijk}$ holds.  
We also remark that the adjoint action of $\rho_{jk} \in \mathfrak{S}_6$ on $\mathcal{W}$ induces the adjoint action of $\rho_{jk}' = \rho_{jk} e^{\frac{\lambda_{jk}}{\hbar}(p'_j-p'_k)}$ on $\mathcal{W}'$ compatible with $\omega$.

In this sense, our spectral parameters $\lambda_i$'s can be regarded as `survivors' thanks to 
the degree of freedom of $\alpha$ in $\pi_{ijk}$ (\ref{pal}) and the choice $\alpha=0 \neq 1$.
A similar remark applies to the spectral parameters entering 
$\mathcal{K}_{ijkl}(\lambda_i,\lambda_j,\lambda_k, \lambda_l)$  in Remark \ref{re:ezk}.
\end{remark}

%%%%%%%%%%%%%%%%%%%%%%%%%%%%%
\section{3D reflection equation}
%%%%%%%%%%%%%%%%%%%%%%%%%%%%%

Now we consider the transformation of the reduced expressions of $w_0 \in W(C_3)$,
from $123123123$ to $321321321$, in two ways:
$$
\begin{matrix}
\underline{12}3\underline{1}23123 & \to & 21\underline{2323}123 & \to & 21323\underline{212}3 & \to & 2\underline{1}3\underline{2}3\underline{1}213
& \to & 2321\underline{232}1\underline{3}
\\
\rotatebox{90}{=} & & & & & & & & \downarrow
\\
123\underline{121}323 & & & & & & & & \underline{232}1\underline{3}2321
\\
\downarrow & & & & & & & & \downarrow
\\
12321\underline{2323} & & & & & & & & 323\underline{212}321
\\
\downarrow & & & & & & & & \downarrow
\\
1\underline{232}1\underline{3}232 & & & & & & & & 323121321
\\
\downarrow & & & & & & & & \rotatebox{90}{=}
\\
1323\underline{212}32 & \to & \underline{1}3\underline{2}3\underline{1}2132 & \to & 321\underline{232}1\underline{3}2 & \to & 321323\underline{212} & \to & 321323121
\end{matrix}
$$ 
This corresponds to the transformation of the wiring diagrams shown 
in Figure \ref{fig:reflection-eq} and \ref{fig:reflection-eq2}.
It implies the 3D reflection relation for the wiring diagrams:
\begin{align}\label{eq:reflection-RK-eq}
R_{457} K_{4689} K_{2379} R_{258} R_{178} K_{1356}R_{124} 
= R_{124} K_{1356} R_{178} R_{258} K_{2379} K_{4689} R_{457},
\end{align}
where $R_{ijk}$ and $K_{ijkl}$ are respectively given by \eqref{quiver-d:A2} and \eqref{quiver-d:C2}.

This transformation is accompanied by cluster mutation sequences. 
We associate the FG quivers 
$J_{123123123}$ and $J_{321321321}$ to the initial and the final reduced words as follows:
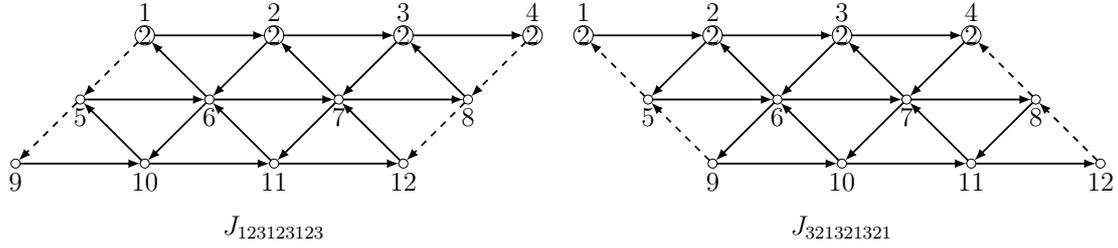
\begin{align*}
%\label{quiver:m=4}
\scalebox{0.85}{
\begin{tikzpicture}
\begin{scope}[>=latex]
\path (2,2) node[circle]{2} coordinate(1) node[above=0.2em]{$1$};
\draw (2,2) circle[radius=0.15];
\path (4,2) node[circle]{2} coordinate(2) node[above=0.2em]{$2$};
\draw (4,2) circle[radius=0.15];
\path (6,2) node[circle]{2} coordinate(3) node[above=0.2em]{$3$};
\draw (6,2) circle[radius=0.15];
\path (8,2) node[circle]{2} coordinate(4) node[above=0.2em]{$4$};
\draw (8,2) circle[radius=0.15];
\draw (1,1) circle(2pt) coordinate(5) node[below]{$5$};
\draw (3,1) circle(2pt) coordinate(6) node[below]{$6$};
\draw (5,1) circle(2pt) coordinate(7) node[below]{$7$};
\draw (7,1) circle(2pt) coordinate(8) node[below]{$8$};
\draw (0,0) circle(2pt) coordinate(9) node[below]{$9$};
\draw (2,0) circle(2pt) coordinate(10) node[below]{$10$};
\draw (4,0) circle(2pt) coordinate(11) node[below]{$11$};
\draw (6,0) circle(2pt) coordinate(12) node[below]{$12$};
\qsarrow{1}{2}
\qsarrow{2}{3}
\qsarrow{3}{4}
\qarrow{5}{6}
\qarrow{6}{7}
\qarrow{7}{8}
\qarrow{9}{10}
\qarrow{10}{11}
\qarrow{11}{12}
\draw[->,shorten >=4pt,shorten <=2pt] (8) -- (3) [thick];
\draw[->,shorten >=2pt,shorten <=4pt] (3) -- (7) [thick];
\draw[->,shorten >=4pt,shorten <=2pt] (7) -- (2) [thick];
\draw[->,shorten >=2pt,shorten <=4pt] (2) -- (6) [thick];
\draw[->,shorten >=4pt,shorten <=2pt] (6) -- (1) [thick];
\qarrow{12}{7}
\qarrow{7}{11}
\qarrow{11}{6}
\qarrow{6}{10}
\qarrow{10}{5}
\draw[->,dashed,shorten >=2pt,shorten <=4pt] (4) -- (8) [thick];
\draw[->,dashed,shorten >=2pt,shorten <=4pt] (1) -- (5) [thick];
\qdarrow{5}{9}
\qdarrow{8}{12} 
%\coordinate (P1) at (6.6,0.5);
%\coordinate (P2) at (7.6,0.5);
%\draw[<->] (P1) -- (P2);
%\draw (7.1,0) node{$\mu_2 \mu_5 \mu_2$};
%{\color{red}
%\fill (3,1) circle(2pt) node[above]{$B$};
%\fill (5,1) circle(2pt) node[above]{$D$};
%\fill (2,0) circle(2pt) node[below]{$A$};
%\fill (4,0) circle(2pt) node[below]{$C$};
%}
\draw (4,-1) node{$J_{123123123}$};  
\end{scope}
\begin{scope}[>=latex,xshift=250pt]
\path (0,2) node[circle]{2} coordinate(1) node[above=0.2em]{$1$};
\draw (0,2) circle[radius=0.15];
\path (2,2) node[circle]{2} coordinate(2) node[above=0.2em]{$2$};
\draw (2,2) circle[radius=0.15];
\path (4,2) node[circle]{2} coordinate(3) node[above=0.2em]{$3$};
\draw (4,2) circle[radius=0.15];
\path (6,2) node[circle]{2} coordinate(4) node[above=0.2em]{$4$};
\draw (6,2) circle[radius=0.15];
\draw (1,1) circle(2pt) coordinate(5) node[below]{$5$};
\draw (3,1) circle(2pt) coordinate(6) node[below]{$6$};
\draw (5,1) circle(2pt) coordinate(7) node[below]{$7$};
\draw (7,1) circle(2pt) coordinate(8) node[below]{$8$};
\draw (2,0) circle(2pt) coordinate(9) node[below]{$9$};
\draw (4,0) circle(2pt) coordinate(10) node[below]{$10$};
\draw (6,0) circle(2pt) coordinate(11) node[below]{$11$};
\draw (8,0) circle(2pt) coordinate(12) node[below]{$12$};
\qsarrow{1}{2}
\qsarrow{2}{3}
\qsarrow{3}{4}
\qarrow{5}{6}
\qarrow{6}{7}
\qarrow{7}{8}
\qarrow{9}{10}
\qarrow{10}{11}
\qarrow{11}{12}
\draw[->,shorten >=2pt,shorten <=4pt] (4) -- (7) [thick];
\draw[->,shorten >=4pt,shorten <=2pt] (7) -- (3) [thick];
\draw[->,shorten >=2pt,shorten <=4pt] (3) -- (6) [thick];
\draw[->,shorten >=4pt,shorten <=2pt] (6) -- (2) [thick];
\draw[->,shorten >=2pt,shorten <=4pt] (2) -- (5) [thick];
\qarrow{8}{11}
\qarrow{11}{7}
\qarrow{7}{10}
\qarrow{10}{6}
\qarrow{6}{9}
\draw[->,dashed,shorten >=4pt,shorten <=2pt] (8) -- (4) [thick];
\draw[->,dashed,shorten >=4pt,shorten <=2pt] (5) -- (1) [thick];
\qdarrow{9}{5}
\qdarrow{12}{8} 
%{\color{red}
%\fill (3,1) circle(2pt) node[above]{$B$};
%\fill (5,1) circle(2pt) node[above]{$D$};
%\fill (2,0) circle(2pt) node[below]{$A$};
%\fill (4,0) circle(2pt) node[below]{$C$};
%}
\draw (4,-1) node{$J_{321321321}$};  
\end{scope}
\end{tikzpicture}
}
\end{align*}
For the quiver $J_{123123123}$ we set $I = \{1,2,\ldots,12 \}$.
% and $I_0 = \{1,4,5,8,9,12\}$. 
The correspondence between these quivers and the wiring diagrams 
is shown in Figures \ref{fig:reflection-eq} and \ref{fig:reflection-eq2}.
One sees that the transformations \eqref{quiver-d:A2} and \eqref{quiver-d:C2} are embedded therein, in the same manner as the tetrahedron transformation.

\begin{figure}[h]
\[
\scalebox{0.85}{
\begin{tikzpicture}
\begin{scope}[>=latex]
\path (2,2) node[circle]{2} coordinate(1) node[above=0.2em]{$1$};
\draw (2,2) circle[radius=0.15];
\path (4,2) node[circle]{2} coordinate(2) node[above=0.2em]{$2$};
\draw (4,2) circle[radius=0.15];
\path (6,2) node[circle]{2} coordinate(3) node[above=0.2em]{$3$};
\draw (6,2) circle[radius=0.15];
\path (8,2) node[circle]{2} coordinate(4) node[above=0.2em]{$4$};
\draw (8,2) circle[radius=0.15];
\draw (1,1) circle(2pt) coordinate(5) node[below]{$5$};
\draw (3,1) circle(2pt) coordinate(6) node[below]{$6$};
\draw (5,1) circle(2pt) coordinate(7) node[below]{$7$};
\draw (7,1) circle(2pt) coordinate(8) node[below]{$8$};
\draw (0,0) circle(2pt) coordinate(9) node[below]{$9$};
\draw (2,0) circle(2pt) coordinate(10) node[below]{$10$};
\draw (4,0) circle(2pt) coordinate(11) node[below]{$11$};
\draw (6,0) circle(2pt) coordinate(12) node[below]{$12$};
\qsarrow{1}{2}
\qsarrow{2}{3}
\qsarrow{3}{4}
\qarrow{5}{6}
\qarrow{6}{7}
\qarrow{7}{8}
\qarrow{9}{10}
\qarrow{10}{11}
\qarrow{11}{12}
\draw[->,shorten >=4pt,shorten <=2pt] (8) -- (3) [thick];
\draw[->,shorten >=2pt,shorten <=4pt] (3) -- (7) [thick];
\draw[->,shorten >=4pt,shorten <=2pt] (7) -- (2) [thick];
\draw[->,shorten >=2pt,shorten <=4pt] (2) -- (6) [thick];
\draw[->,shorten >=4pt,shorten <=2pt] (6) -- (1) [thick];
\qarrow{12}{7}
\qarrow{7}{11}
\qarrow{11}{6}
\qarrow{6}{10}
\qarrow{10}{5}
\draw[->,dashed,shorten >=2pt,shorten <=4pt] (4) -- (8) [thick];
\draw[->,dashed,shorten >=2pt,shorten <=4pt] (1) -- (5) [thick];
\qdarrow{5}{9}
\qdarrow{8}{12} 
{\color{red}
\fill (1,0) circle(2pt) coordinate(A) node[below]{$1$};
\fill (2,1) circle(2pt) coordinate(B) node[above]{$2$};
\fill (3,2) circle(2pt) coordinate(C) node[below]{$3$};
\fill (3,0) circle(2pt) coordinate(D) node[above]{$4$};
\fill (4,1) circle(2pt) coordinate(E) node[above]{$5$};
\fill (5,2) circle(2pt) coordinate(F) node[above]{$6$};
\fill (5,0) circle(2pt) coordinate(G) node[above]{$7$};
\fill (6,1) circle(2pt) coordinate(H) node[above]{$8$};
\fill (7,2) circle(2pt) coordinate(I) node[above]{$9$};
\draw [-] (0,1.5) to [out = 0, in = 135] (B);
\draw [-] (B) -- (D);
\draw [-] (D) to [out = -45, in = -135] (G); 
\draw [-] (G) -- (I);
\draw [->] (I) to [out = -45, in = 180] (8,1.5);
\draw [-] (0,0.5) to [out = 0, in = 135] (A); 
\draw [-] (A) to [out = -45, in = -135] (D);
\draw [-] (D) -- (F);
\draw [-] (F) -- (H);
\draw [->] (H) to [out = -45, in = 180] (8,0.5);
\draw [-] (0,-0.5) to [out = 0, in = -135] (A); 
\draw [-] (A) -- (C);
\draw [-] (C) -- (G);
\draw [->] (G) to [out = -45, in = 180] (8,-0.5);
}
\draw[->] (4,-0.8) -- (4,-1.3);
\draw (4,-1.1) circle(0pt) node[left]{$R_{124}$} node[right]{$\mu_{10}$};
%
%\draw (4,-1) node{$J_{123123123}$};  
\end{scope}
\begin{scope}[>=latex,yshift=-115pt]
\path (2.7,2) node[circle]{2} coordinate(1) node[above=0.2em]{$1$};
\draw (2.7,2) circle[radius=0.15];
\path (4.5,2) node[circle]{2} coordinate(2) node[above=0.2em]{$2$};
\draw (4.5,2) circle[radius=0.15];
\path (6.3,2) node[circle]{2} coordinate(3) node[above=0.2em]{$3$};
\draw (6.3,2) circle[radius=0.15];
\path (8.1,2) node[circle]{2} coordinate(4) node[above=0.2em]{$4$};
\draw (8.1,2) circle[radius=0.15];
\draw (0,1) circle(2pt) coordinate(5) node[below]{$5$};
\draw (3.6,1) circle(2pt) coordinate(6) node[below]{$6$};
\draw (5.4,1) circle(2pt) coordinate(7) node[below]{$7$};
\draw (7.2,1) circle(2pt) coordinate(8) node[below]{$8$};
\draw (0.9,0) circle(2pt) coordinate(9) node[below]{$9$};
\draw (1.8,1) circle(2pt) coordinate(10) node[below]{$10$};
\draw (3.6,0) circle(2pt) coordinate(11) node[below]{$11$};
\draw (6.3,0) circle(2pt) coordinate(12) node[below]{$12$};
\qsarrow{1}{2}
\qsarrow{2}{3}
\qsarrow{3}{4}
\qarrow{5}{10}
\qarrow{10}{6}
\qarrow{6}{7}
\qarrow{7}{8}
\qarrow{9}{11}
\qarrow{11}{12}
\draw[->,dashed,shorten >=2pt,shorten <=4pt] (4) -- (8) [thick];
\qarrowsb{8}{3}
\qarrowsa{3}{7}
\qarrowsb{7}{2}
\qarrowsa{2}{6}
\qarrowsb{6}{1}
\draw[->,dashed,shorten >=2pt,shorten <=4pt] (1) -- (5) [thick];
\qdarrow{8}{12}
\qarrow{12}{7}
\qarrow{7}{11}
\qarrow{11}{10}
\qarrow{10}{9}
\qdarrow{9}{5}
{\color{red}
\fill (2.7,1) circle(2pt) coordinate(A) node[below]{$1$};
\fill (1.8,0) circle(2pt) coordinate(B) node[above]{$2$};
\fill (3.6,2) circle(2pt) coordinate(C) node[below]{$3$};
\fill (0.9,1) circle(2pt) coordinate(D) node[above]{$4$};
\fill (4.5,1) circle(2pt) coordinate(E) node[above]{$5$};
\fill (5.4,2) circle(2pt) coordinate(F) node[above]{$6$};
\fill (5.4,0) circle(2pt) coordinate(G) node[above]{$7$};
\fill (6.3,1) circle(2pt) coordinate(H) node[above]{$8$};
\fill (7.2,2) circle(2pt) coordinate(I) node[above]{$9$};
\draw [-] (0,1.5) to [out = 0, in = 135] (D);
\draw [-] (D) -- (B);
\draw [-] (B) to [out = -45, in = -135] (G); 
\draw [-] (G) -- (I);
\draw [->] (I) to [out = -45, in = 180] (8.1,1.5);
\draw [-] (0,0.5) to [out = 0, in = -135] (D); 
\draw [-] (D) to [out = 45, in = 135] (A);
\draw [-] (A) to [out = -45, in = -135] (E);
\draw [-] (E) -- (F);
\draw [-] (F) -- (H);
\draw [->] (H) to [out = -45, in = 180] (8.1,0.5);
\draw [-] (0,-0.5) to [out = 0, in = -135] (B); 
\draw [-] (B) -- (C);
\draw [-] (C) -- (G);
\draw [->] (G) to [out = -45, in = 180] (8.1,-0.5);
}
\draw[->] (4,-0.9) -- (4,-1.4);
\draw (4,-1.2) circle(0pt) node[left]{$K_{1356}$} node[right]{$\mu_{2,6,2}$};
\end{scope}
\begin{scope}[>=latex,yshift=-230pt]
\path (0.9,2) node[circle]{2} coordinate(1) node[above=0.2em]{$1$};
\draw (0.9,2) circle[radius=0.15];
\path (2.7,2) node[circle]{2} coordinate(2) node[above=0.2em]{$2$};
\draw (2.7,2) circle[radius=0.15];
\path (5.4,2) node[circle]{2} coordinate(3) node[above=0.2em]{$3$};
\draw (5.4,2) circle[radius=0.15];
\path (8.1,2) node[circle]{2} coordinate(4) node[above=0.2em]{$4$};
\draw (8.1,2) circle[radius=0.15];
\draw (0,1) circle(2pt) coordinate(5) node[below]{$5$};
\draw (3.6,1) circle(2pt) coordinate(6) node[below]{$6$};
\draw (5.4,1) circle(2pt) coordinate(7) node[below]{$7$};
\draw (7.2,1) circle(2pt) coordinate(8) node[below]{$8$};
\draw (0.9,0) circle(2pt) coordinate(9) node[below]{$9$};
\draw (1.8,1) circle(2pt) coordinate(10) node[below]{$10$};
\draw (3.6,0) circle(2pt) coordinate(11) node[below]{$11$};
\draw (6.3,0) circle(2pt) coordinate(12) node[below]{$12$};
\qsarrow{1}{2}
\qsarrow{2}{3}
\qsarrow{3}{4}
\qarrow{5}{10}
\qarrow{10}{6}
\qarrow{6}{7}
\qarrow{7}{8}
\qarrow{9}{11}
\qarrow{11}{12}
\draw[->,dashed,shorten >=2pt,shorten <=4pt] (4) -- (8) [thick];
\qarrowsb{8}{3}
\qarrowsa{3}{6}
\qarrowsb{6}{2}
\qarrowsa{2}{10}
\qarrowsb{10}{1}
\draw[->,dashed,shorten >=2pt,shorten <=4pt] (1) -- (5) [thick];
\qdarrow{8}{12}
\qarrow{12}{7}
\qarrow{7}{11}
\qarrow{11}{10}
\qarrow{10}{9}
\qdarrow{9}{5}
{\color{red}
\fill (4.5,1) circle(2pt) coordinate(A) node[below]{$1$};
\fill (1.8,0) circle(2pt) coordinate(B) node[above]{$2$};
\fill (3.6,2) circle(2pt) coordinate(C) node[below]{$3$};
\fill (0.9,1) circle(2pt) coordinate(D) node[above]{$4$};
\fill (2.7,1) circle(2pt) coordinate(E) node[above]{$5$};
\fill (1.8,2) circle(2pt) coordinate(F) node[above]{$6$};
\fill (5.4,0) circle(2pt) coordinate(G) node[above]{$7$};
\fill (6.3,1) circle(2pt) coordinate(H) node[above]{$8$};
\fill (7.2,2) circle(2pt) coordinate(I) node[above]{$9$};
\draw [-] (0,1.5) to [out = 0, in = 135] (D);
\draw [-] (D) -- (B);
\draw [-] (B) to [out = -45, in = -135] (G); 
\draw [-] (G) -- (I);
\draw [->] (I) to [out = -45, in = 180] (8.1,1.5);
\draw [-] (0,0.5) to [out = 0, in = -135] (D); 
\draw [-] (D) to [out = 45, in = -135] (F);
\draw [-] (F) to [out = -45, in = 135] (E);
\draw [-] (E) to [out = -45, in = -135] (A);
\draw [-] (A) to [out = 45, in = 135] (H);
\draw [->] (H) to [out = -45, in = 180] (8.1,0.5);
\draw [-] (0,-0.5) to [out = 0, in = -135] (B); 
\draw [-] (B) -- (C);
\draw [-] (C) -- (G);
\draw [->] (G) to [out = -45, in = 180] (8.1,-0.5);
}
\draw[->] (4,-0.9) -- (4,-1.4);
\draw (4,-1.2) circle(0pt) node[left]{$R_{178}$} node[right]{$\mu_7$};
\end{scope}
\begin{scope}[>=latex,yshift=-345pt]
\path (1,2) node[circle]{2} coordinate(1) node[above=0.2em]{$1$};
\draw (1,2) circle[radius=0.15];
\path (3,2) node[circle]{2} coordinate(2) node[above=0.2em]{$2$};
\draw (3,2) circle[radius=0.15];
\path (5,2) node[circle]{2} coordinate(3) node[above=0.2em]{$3$};
\draw (5,2) circle[radius=0.15];
\path (7,2) node[circle]{2} coordinate(4) node[above=0.2em]{$4$};
\draw (7,2) circle[radius=0.15];
\draw (0,1) circle(2pt) coordinate(5) node[below]{$5$};
\draw (4,1) circle(2pt) coordinate(6) node[below]{$6$};
\draw (5,0) circle(2pt) coordinate(7) node[below]{$7$};
\draw (6,1) circle(2pt) coordinate(8) node[below]{$8$};
\draw (1,0) circle(2pt) coordinate(9) node[below]{$9$};
\draw (2,1) circle(2pt) coordinate(10) node[below]{$10$};
\draw (3,0) circle(2pt) coordinate(11) node[below]{$11$};
\draw (7,0) circle(2pt) coordinate(12) node[below]{$12$};
\qsarrow{1}{2}
\qsarrow{2}{3}
\qsarrow{3}{4}
\qarrow{5}{10}
\qarrow{10}{6}
\qarrow{6}{8}
\qarrow{9}{11}
\qarrow{11}{7}
\qarrow{7}{12}
\draw[->,dashed,shorten >=2pt,shorten <=4pt] (4) -- (8) [thick];
\qarrowsb{8}{3}
\qarrowsa{3}{6}
\qarrowsb{6}{2}
\qarrowsa{2}{10}
\qarrowsb{10}{1}
\draw[->,dashed,shorten >=2pt,shorten <=4pt] (1) -- (5) [thick];
\qdarrow{12}{8}
\qarrow{8}{7}
\qarrow{7}{6}
\qarrow{6}{11}
\qarrow{11}{10}
\qarrow{10}{9}
\qdarrow{9}{5}
{\color{red}
\fill (6,0) circle(2pt) coordinate(A) node[below]{$1$};
\fill (2,0) circle(2pt) coordinate(B) node[above]{$2$};
\fill (4,2) circle(2pt) coordinate(C) node[below]{$3$};
\fill (1,1) circle(2pt) coordinate(D) node[above]{$4$};
\fill (3,1) circle(2pt) coordinate(E) node[above]{$5$};
\fill (2,2) circle(2pt) coordinate(F) node[above]{$6$};
\fill (5,1) circle(2pt) coordinate(G) node[above]{$7$};
\fill (4,0) circle(2pt) coordinate(H) node[above]{$8$};
\fill (6,2) circle(2pt) coordinate(I) node[above]{$9$};
\draw [-] (0,1.5) to [out = 0, in = 135] (D);
\draw [-] (D) -- (B);
\draw [-] (B) to [out = -45, in = -135] (H); 
\draw [-] (H) -- (I);
\draw [->] (I) to [out = -45, in = 180] (7,1.5);
\draw [-] (0,0.5) to [out = 0, in = -135] (D); 
\draw [-] (D) to [out = 45, in = -135] (F);
\draw [-] (F) -- (H);
\draw [-] (H) to [out = -45, in = -135] (A);
\draw [->] (A) to [out = 45, in = 180] (7,0.5);
\draw [-] (0,-0.5) to [out = 0, in = -135] (B); 
\draw [-] (B) -- (C);
\draw [-] (C) -- (A);
\draw [->] (A) to [out = -45, in = 180] (7,-0.5);
}
\draw[->] (7.7,1) -- (8.7,1);
\draw (8.2,1) circle(0pt) node[below]{$R_{258}$} node[above]{$\mu_{11}$};
\end{scope}
\begin{scope}[>=latex,xshift=270,yshift=-345pt]
\path (0.9,2) node[circle]{2} coordinate(1) node[above=0.2em]{$1$};
\draw (0.9,2) circle[radius=0.15];
\path (3.6,2) node[circle]{2} coordinate(2) node[above=0.2em]{$2$};
\draw (3.6,2) circle[radius=0.15];
\path (6.3,2) node[circle]{2} coordinate(3) node[above=0.2em]{$3$};
\draw (6.3,2) circle[radius=0.15];
\path (8.1,2) node[circle]{2} coordinate(4) node[above=0.2em]{$4$};
\draw (8.1,2) circle[radius=0.15];
\draw (0,1) circle(2pt) coordinate(5) node[below]{$5$};
\draw (5.4,1) circle(2pt) coordinate(6) node[below]{$6$};
\draw (5.4,0) circle(2pt) coordinate(7) node[below]{$7$};
\draw (7.2,1) circle(2pt) coordinate(8) node[below]{$8$};
\draw (1.8,0) circle(2pt) coordinate(9) node[below]{$9$};
\draw (1.8,1) circle(2pt) coordinate(10) node[below]{$10$};
\draw (3.6,1) circle(2pt) coordinate(11) node[below]{$11$};
\draw (8.1,0) circle(2pt) coordinate(12) node[below]{$12$};
\qsarrow{1}{2}
\qsarrow{2}{3}
\qsarrow{3}{4}
\qarrow{5}{10}
\qarrow{10}{11}
\qarrow{11}{6}
\qarrow{6}{8}
\qarrow{9}{7}
\qarrow{7}{12}
\draw[->,dashed,shorten >=2pt,shorten <=4pt] (4) -- (8) [thick];
\qarrowsb{8}{3}
\qarrowsa{3}{6}
\qarrowsb{6}{2}
\qarrowsa{2}{10}
\qarrowsb{10}{1}
\draw[->,dashed,shorten >=2pt,shorten <=4pt] (1) -- (5) [thick];
\qdarrow{12}{8}
\qarrow{8}{7}
\qarrow{7}{11}
\qarrow{11}{9}
\qdarrow{9}{5}
{\color{red}
\fill (7.2,0) circle(2pt) coordinate(A) node[below]{$1$};
\fill (4.5,1) circle(2pt) coordinate(B) node[above]{$2$};
\fill (5.4,2) circle(2pt) coordinate(C) node[below]{$3$};
\fill (0.8,1) circle(2pt) coordinate(D) node[above]{$4$};
\fill (3.6,0) circle(2pt) coordinate(E) node[above]{$5$};
\fill (1.8,2) circle(2pt) coordinate(F) node[above]{$6$};
\fill (6.3,1) circle(2pt) coordinate(G) node[above]{$7$};
\fill (2.7,1) circle(2pt) coordinate(H) node[above]{$8$};
\fill (7.2,2) circle(2pt) coordinate(I) node[above]{$9$};
\draw [-] (0,1.5) to [out = 0, in = 135] (D);
\draw [-] (D) to [out = -45, in = -135] (H);
\draw [-] (H) to [out = 45, in = 135] (B); 
\draw [-] (B) to [out = -45, in = -135] (G);
\draw [-] (G) -- (I);
\draw [->] (I) to [out = -45, in = 180] (8.1,1.5);
\draw [-] (0,0.5) to [out = 0, in = -135] (D); 
\draw [-] (D) to [out = 45, in = -135] (F);
\draw [-] (F) -- (E);
\draw [-] (E) to [out = -45, in = -135] (A);
\draw [->] (A) to [out = 45, in = 180] (8.1,0.5);
\draw [-] (0,-0.5) to [out = 0, in = -135] (E); 
\draw [-] (E) -- (C);
\draw [-] (C) -- (A);
\draw [->] (A) to [out = -45, in = 180] (8.1,-0.5);
}
\draw[->] (4,2.7) -- (4,3.2);
\draw (4,2.9) circle(0pt) node[left]{$K_{2379}$} node[right]{$\mu_{3,6,3}$};
\end{scope}
\begin{scope}[>=latex,xshift=270,yshift=-230pt]
\path (0.9,2) node[circle]{2} coordinate(1) node[above=0.2em]{$1$};
\draw (0.9,2) circle[radius=0.15];
\path (2.7,2) node[circle]{2} coordinate(2) node[above=0.2em]{$2$};
\draw (2.7,2) circle[radius=0.15];
\path (4.5,2) node[circle]{2} coordinate(3) node[above=0.2em]{$3$};
\draw (4.5,2) circle[radius=0.15];
\path (6.3,2) node[circle]{2} coordinate(4) node[above=0.2em]{$4$};
\draw (6.3,2) circle[radius=0.15];
\draw (0,1) circle(2pt) coordinate(5) node[below]{$5$};
\draw (5.4,1) circle(2pt) coordinate(6) node[below]{$6$};
\draw (5.4,0) circle(2pt) coordinate(7) node[below]{$7$};
\draw (7.2,1) circle(2pt) coordinate(8) node[below]{$8$};
\draw (1.6,0) circle(2pt) coordinate(9) node[below]{$9$};
\draw (1.6,1) circle(2pt) coordinate(10) node[below]{$10$};
\draw (3.6,1) circle(2pt) coordinate(11) node[below]{$11$};
\draw (8.1,0) circle(2pt) coordinate(12) node[below]{$12$};
\qsarrow{1}{2}
\qsarrow{2}{3}
\qsarrow{3}{4}
\qarrow{5}{10}
\qarrow{10}{11}
\qarrow{11}{6}
\qarrow{6}{8}
\qarrow{9}{7}
\qarrow{7}{12}
\draw[->,dashed,shorten >=4pt,shorten <=2pt] (8) -- (4) [thick];
\qarrowsa{4}{6}
\qarrowsb{6}{3}
\qarrowsa{3}{11}
\qarrowsb{11}{2}
\qarrowsa{2}{10}
\qarrowsb{10}{1}
\draw[->,dashed,shorten >=2pt,shorten <=4pt] (1) -- (5) [thick];
\qdarrow{12}{8}
\qarrow{8}{7}
\qarrow{7}{11}
\qarrow{11}{9}
\qdarrow{9}{5}
{\color{red}
\fill (7.2,0) circle(2pt) coordinate(A) node[below]{$1$};
\fill (6.3,1) circle(2pt) coordinate(B) node[above]{$2$};
\fill (5.4,2) circle(2pt) coordinate(C) node[below]{$3$};
\fill (0.9,1) circle(2pt) coordinate(D) node[above]{$4$};
\fill (3.6,0) circle(2pt) coordinate(E) node[above]{$5$};
\fill (1.8,2) circle(2pt) coordinate(F) node[above]{$6$};
\fill (4.5,1) circle(2pt) coordinate(G) node[above]{$7$};
\fill (2.7,1) circle(2pt) coordinate(H) node[above]{$8$};
\fill (3.6,2) circle(2pt) coordinate(I) node[above]{$9$};
\draw [-] (0,1.5) to [out = 0, in = 135] (D);
\draw [-] (D) to [out = -45, in = -135] (H);
\draw [-] (H) -- (I);
\draw [-] (I) -- (G);
\draw [-] (G) to [out = -45, in = -135] (B);
\draw [->] (B) to [out = 45, in = 180] (8.1,1.5);
\draw [-] (0,0.5) to [out = 0, in = -135] (D); 
\draw [-] (D) to [out = 45, in = -135] (F);
\draw [-] (F) -- (E);
\draw [-] (E) to [out = -45, in = -135] (A);
\draw [->] (A) to [out = 45, in = 180] (8.1,0.5);
\draw [-] (0,-0.5) to [out = 0, in = -135] (E); 
\draw [-] (E) -- (C);
\draw [-] (C) -- (A);
\draw [->] (A) to [out = -45, in = 180] (8.1,-0.5);
}
\draw[->] (4,2.7) -- (4,3.2);
\draw (4,2.9) circle(0pt) node[left]{$K_{4689}$} node[right]{$\mu_{2,10,2}$};
\end{scope}
\begin{scope}[>=latex,xshift=270,yshift=-115pt]
\path (0,2) node[circle]{2} coordinate(1) node[above=0.2em]{$1$};
\draw (0,2) circle[radius=0.15];
\path (1.6,2) node[circle]{2} coordinate(2) node[above=0.2em]{$2$};
\draw (1.6,2) circle[radius=0.15];
\path (4,2) node[circle]{2} coordinate(3) node[above=0.2em]{$3$};
\draw (4,2) circle[radius=0.15];
\path (7.2,2) node[circle]{2} coordinate(4) node[above=0.2em]{$4$};
\draw (7.2,2) circle[radius=0.15];
\draw (0.8,1) circle(2pt) coordinate(5) node[below]{$5$};
\draw (5.6,1) circle(2pt) coordinate(6) node[below]{$6$};
\draw (5.6,0) circle(2pt) coordinate(7) node[below]{$7$};
\draw (7.2,1) circle(2pt) coordinate(8) node[below]{$8$};
\draw (1.6,0) circle(2pt) coordinate(9) node[below]{$9$};
\draw (2.4,1) circle(2pt) coordinate(10) node[below]{$10$};
\draw (4,1) circle(2pt) coordinate(11) node[below]{$11$};
\draw (8,0) circle(2pt) coordinate(12) node[below]{$12$};
\qsarrow{1}{2}
\qsarrow{2}{3}
\qsarrow{3}{4}
\qarrow{5}{10}
\qarrow{10}{11}
\qarrow{11}{6}
\qarrow{6}{8}
\qarrow{9}{7}
\qarrow{7}{12}
\draw[->,dashed,shorten >=4pt,shorten <=2pt] (8) -- (4) [thick];
\qarrowsa{4}{6}
\qarrowsb{6}{3}
\qarrowsa{3}{10}
\qarrowsb{10}{2}
\qarrowsa{2}{5}
\draw[->,dashed,shorten >=4pt,shorten <=2pt] (5) -- (1) [thick];
\qdarrow{12}{8}
\qarrow{8}{7}
\qarrow{7}{11}
\qarrow{11}{9}
\qdarrow{9}{5}
{\color{red}
\fill (7.2,0) circle(2pt) coordinate(A) node[below]{$1$};
\fill (6.4,1) circle(2pt) coordinate(B) node[above]{$2$};
\fill (5.6,2) circle(2pt) coordinate(C) node[below]{$3$};
\fill (3.2,1) circle(2pt) coordinate(D) node[above]{$4$};
\fill (4,0) circle(2pt) coordinate(E) node[above]{$5$};
\fill (2.4,2) circle(2pt) coordinate(F) node[above]{$6$};
\fill (4.8,1) circle(2pt) coordinate(G) node[above]{$7$};
\fill (1.6,1) circle(2pt) coordinate(H) node[above]{$8$};
\fill (0.8,2) circle(2pt) coordinate(I) node[above]{$9$};
\draw [-] (0,1.5) to [out = 0, in = -135] (I);
\draw [-] (I) -- (H);
\draw [-] (H) to [out = -45, in = -135] (D);
\draw [-] (D) to [out = 45, in = 135] (G);
\draw [-] (G) to [out = -45, in = -135] (B);
\draw [->] (B) to [out = 45, in = 180] (8,1.5);
\draw [-] (0,0.5) to [out = 0, in = -135] (H);
\draw [-] (H) -- (F); 
\draw [-] (F) to [out = 45, in = -135] (F);
\draw [-] (F) -- (E);
\draw [-] (E) to [out = -45, in = -135] (A);
\draw [->] (A) to [out = 45, in = 180] (8,0.5);
\draw [-] (0,-0.5) to [out = 0, in = -135] (E); 
\draw [-] (E) -- (C);
\draw [-] (C) -- (A);
\draw [->] (A) to [out = -45, in = 180] (8,-0.5);
}
\draw[->] (4,2.7) -- (4,3.2);
\draw (4,2.9) circle(0pt) node[left]{$R_{457}$} node[right]{$\mu_{11}$};
\end{scope}
\begin{scope}[>=latex,xshift=270pt]
\path (0,2) node[circle]{2} coordinate(1) node[above=0.2em]{$1$};
\draw (0,2) circle[radius=0.15];
\path (2,2) node[circle]{2} coordinate(2) node[above=0.2em]{$2$};
\draw (2,2) circle[radius=0.15];
\path (4,2) node[circle]{2} coordinate(3) node[above=0.2em]{$3$};
\draw (4,2) circle[radius=0.15];
\path (6,2) node[circle]{2} coordinate(4) node[above=0.2em]{$4$};
\draw (6,2) circle[radius=0.15];
\draw (1,1) circle(2pt) coordinate(5) node[below]{$5$};
\draw (3,1) circle(2pt) coordinate(6) node[below]{$6$};
\draw (5,1) circle(2pt) coordinate(7) node[below]{$7$};
\draw (7,1) circle(2pt) coordinate(8) node[below]{$8$};
\draw (2,0) circle(2pt) coordinate(9) node[below]{$9$};
\draw (4,0) circle(2pt) coordinate(10) node[below]{$10$};
\draw (6,0) circle(2pt) coordinate(11) node[below]{$11$};
\draw (8,0) circle(2pt) coordinate(12) node[below]{$12$};
\qsarrow{1}{2}
\qsarrow{2}{3}
\qsarrow{3}{4}
\qarrow{5}{6}
\qarrow{6}{7}
\qarrow{7}{8}
\qarrow{9}{10}
\qarrow{10}{11}
\qarrow{11}{12}
\draw[->,shorten >=2pt,shorten <=4pt] (4) -- (7) [thick];
\draw[->,shorten >=4pt,shorten <=2pt] (7) -- (3) [thick];
\draw[->,shorten >=2pt,shorten <=4pt] (3) -- (6) [thick];
\draw[->,shorten >=4pt,shorten <=2pt] (6) -- (2) [thick];
\draw[->,shorten >=2pt,shorten <=4pt] (2) -- (5) [thick];
\qarrow{8}{11}
\qarrow{11}{7}
\qarrow{7}{10}
\qarrow{10}{6}
\qarrow{6}{9}
\draw[->,dashed,shorten >=4pt,shorten <=2pt] (8) -- (4) [thick];
\draw[->,dashed,shorten >=4pt,shorten <=2pt] (5) -- (1) [thick];
\qdarrow{9}{5}
\qdarrow{12}{8} 
{\color{red}
\fill (7,0) circle(2pt) coordinate(A) node[below]{$1$};
\fill (6,1) circle(2pt) coordinate(B) node[above]{$2$};
\fill (5,2) circle(2pt) coordinate(C) node[below]{$3$};
\fill (5,0) circle(2pt) coordinate(D) node[above]{$4$};
\fill (4,1) circle(2pt) coordinate(E) node[above]{$5$};
\fill (3,2) circle(2pt) coordinate(F) node[above]{$6$};
\fill (3,0) circle(2pt) coordinate(G) node[above]{$7$};
\fill (2,1) circle(2pt) coordinate(H) node[above]{$8$};
\fill (1,2) circle(2pt) coordinate(I) node[above]{$9$};
\draw [-] (0,1.5) to [out = 0, in = -135] (I);
\draw [-] (I) -- (G);
\draw [-] (G) to [out = -45, in = -135] (D);
\draw [-] (D) -- (B);
\draw [->] (B) to [out = 45, in = 180] (8,1.5);
\draw [-] (0,0.5) to [out = 0, in = -135] (H);
\draw [-] (H) -- (F); 
\draw [-] (F) to [out = 45, in = -135] (F);
\draw [-] (F) -- (D);
\draw [-] (D) to [out = -45, in = -135] (A);
\draw [->] (A) to [out = 45, in = 180] (8,0.5);
\draw [-] (0,-0.5) to [out = 0, in = -135] (G); 
\draw [-] (G) -- (C);
\draw [-] (C) -- (A);
\draw [->] (A) to [out = -45, in = 180] (8,-0.5);
}
\end{scope}
\end{tikzpicture}
}
\]
\caption{3D reflection transformation (LHS of \eqref{eq:reflection-RK-eq}).}
\label{fig:reflection-eq}
\end{figure}
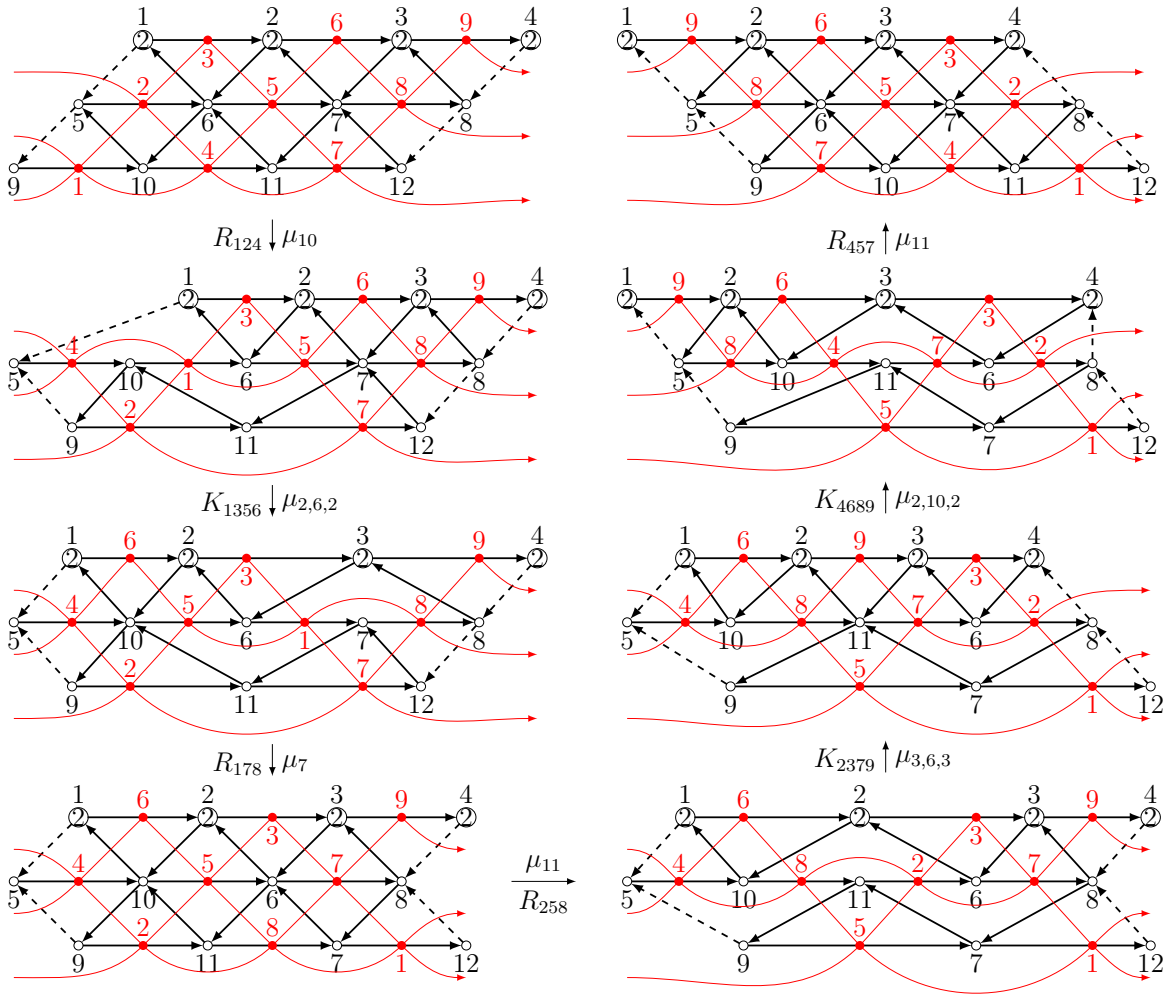

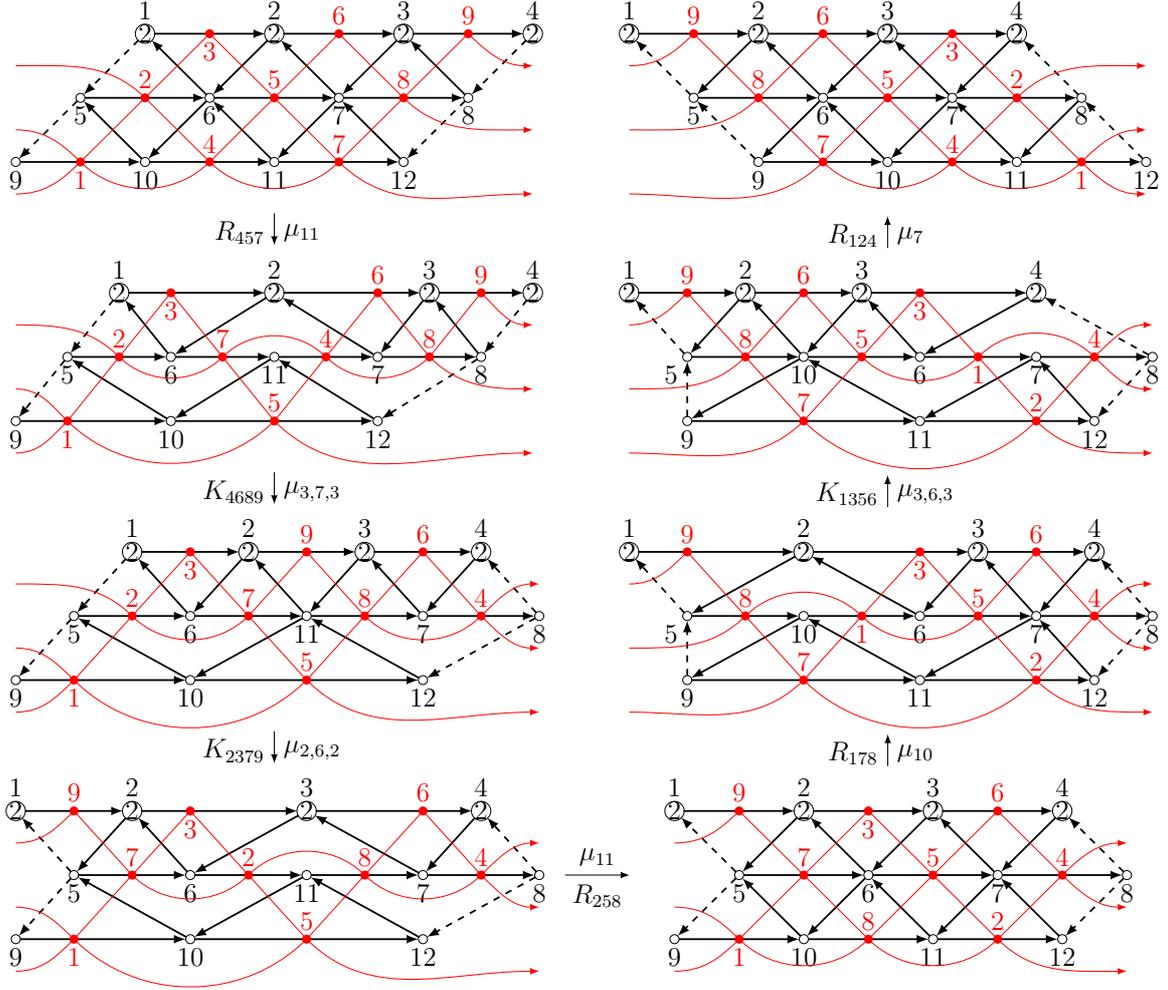
\begin{figure}[h]
\[
\scalebox{0.85}{
\begin{tikzpicture}
\begin{scope}[>=latex]
\path (2,2) node[circle]{2} coordinate(1) node[above=0.2em]{$1$};
\draw (2,2) circle[radius=0.15];
\path (4,2) node[circle]{2} coordinate(2) node[above=0.2em]{$2$};
\draw (4,2) circle[radius=0.15];
\path (6,2) node[circle]{2} coordinate(3) node[above=0.2em]{$3$};
\draw (6,2) circle[radius=0.15];
\path (8,2) node[circle]{2} coordinate(4) node[above=0.2em]{$4$};
\draw (8,2) circle[radius=0.15];
\draw (1,1) circle(2pt) coordinate(5) node[below]{$5$};
\draw (3,1) circle(2pt) coordinate(6) node[below]{$6$};
\draw (5,1) circle(2pt) coordinate(7) node[below]{$7$};
\draw (7,1) circle(2pt) coordinate(8) node[below]{$8$};
\draw (0,0) circle(2pt) coordinate(9) node[below]{$9$};
\draw (2,0) circle(2pt) coordinate(10) node[below]{$10$};
\draw (4,0) circle(2pt) coordinate(11) node[below]{$11$};
\draw (6,0) circle(2pt) coordinate(12) node[below]{$12$};
\qsarrow{1}{2}
\qsarrow{2}{3}
\qsarrow{3}{4}
\qarrow{5}{6}
\qarrow{6}{7}
\qarrow{7}{8}
\qarrow{9}{10}
\qarrow{10}{11}
\qarrow{11}{12}
\draw[->,shorten >=4pt,shorten <=2pt] (8) -- (3) [thick];
\draw[->,shorten >=2pt,shorten <=4pt] (3) -- (7) [thick];
\draw[->,shorten >=4pt,shorten <=2pt] (7) -- (2) [thick];
\draw[->,shorten >=2pt,shorten <=4pt] (2) -- (6) [thick];
\draw[->,shorten >=4pt,shorten <=2pt] (6) -- (1) [thick];
\qarrow{12}{7}
\qarrow{7}{11}
\qarrow{11}{6}
\qarrow{6}{10}
\qarrow{10}{5}
\draw[->,dashed,shorten >=2pt,shorten <=4pt] (4) -- (8) [thick];
\draw[->,dashed,shorten >=2pt,shorten <=4pt] (1) -- (5) [thick];
\qdarrow{5}{9}
\qdarrow{8}{12} 
{\color{red}
\fill (1,0) circle(2pt) coordinate(A) node[below]{$1$};
\fill (2,1) circle(2pt) coordinate(B) node[above]{$2$};
\fill (3,2) circle(2pt) coordinate(C) node[below]{$3$};
\fill (3,0) circle(2pt) coordinate(D) node[above]{$4$};
\fill (4,1) circle(2pt) coordinate(E) node[above]{$5$};
\fill (5,2) circle(2pt) coordinate(F) node[above]{$6$};
\fill (5,0) circle(2pt) coordinate(G) node[above]{$7$};
\fill (6,1) circle(2pt) coordinate(H) node[above]{$8$};
\fill (7,2) circle(2pt) coordinate(I) node[above]{$9$};
\draw [-] (0,1.5) to [out = 0, in = 135] (B);
\draw [-] (B) -- (D);
\draw [-] (D) to [out = -45, in = -135] (G); 
\draw [-] (G) -- (I);
\draw [->] (I) to [out = -45, in = 180] (8,1.5);
\draw [-] (0,0.5) to [out = 0, in = 135] (A); 
\draw [-] (A) to [out = -45, in = -135] (D);
\draw [-] (D) -- (F);
\draw [-] (F) -- (H);
\draw [->] (H) to [out = -45, in = 180] (8,0.5);
\draw [-] (0,-0.5) to [out = 0, in = -135] (A); 
\draw [-] (A) -- (C);
\draw [-] (C) -- (G);
\draw [->] (G) to [out = -45, in = 180] (8,-0.5);
}
\draw[->] (4,-0.8) -- (4,-1.3);
\draw (4,-1.1) circle(0pt) node[left]{$R_{457}$} node[right]{$\mu_{11}$};
\end{scope}
\begin{scope}[>=latex,yshift=-115pt]
\path (1.6,2) node[circle]{2} coordinate(1) node[above=0.2em]{$1$};
\draw (1.6,2) circle[radius=0.15];
\path (4,2) node[circle]{2} coordinate(2) node[above=0.2em]{$2$};
\draw (4,2) circle[radius=0.15];
\path (6.4,2) node[circle]{2} coordinate(3) node[above=0.2em]{$3$};
\draw (6.4,2) circle[radius=0.15];
\path (8,2) node[circle]{2} coordinate(4) node[above=0.2em]{$4$};
\draw (8,2) circle[radius=0.15];
\draw (0.8,1) circle(2pt) coordinate(5) node[below]{$5$};
\draw (2.4,1) circle(2pt) coordinate(6) node[below]{$6$};
\draw (5.6,1) circle(2pt) coordinate(7) node[below]{$7$};
\draw (7.2,1) circle(2pt) coordinate(8) node[below]{$8$};
\draw (0,0) circle(2pt) coordinate(9) node[below]{$9$};
\draw (2.4,0) circle(2pt) coordinate(10) node[below]{$10$};
\draw (4,1) circle(2pt) coordinate(11) node[below]{$11$};
\draw (5.6,0) circle(2pt) coordinate(12) node[below]{$12$};
\qsarrow{1}{2}
\qsarrow{2}{3}
\qsarrow{3}{4}
\qarrow{5}{6}
\qarrow{6}{11}
\qarrow{11}{7}
\qarrow{7}{8}
\qarrow{9}{10}
\qarrow{10}{12}
\draw[->,dashed,shorten >=2pt,shorten <=4pt] (4) -- (8) [thick];
\qarrowsb{8}{3}
\qarrowsa{3}{7}
\qarrowsb{7}{2}
\qarrowsa{2}{6}
\qarrowsb{6}{1}
\draw[->,dashed,shorten >=2pt,shorten <=4pt] (1) -- (5) [thick];
\qdarrow{8}{12} 
\qarrow{12}{11}
\qarrow{11}{10}
\qarrow{10}{5}
\qdarrow{5}{9}
{\color{red}
\fill (0.8,0) circle(2pt) coordinate(A) node[below]{$1$};
\fill (1.6,1) circle(2pt) coordinate(B) node[above]{$2$};
\fill (2.4,2) circle(2pt) coordinate(C) node[below]{$3$};
\fill (4.8,1) circle(2pt) coordinate(D) node[above]{$4$};
\fill (4,0) circle(2pt) coordinate(E) node[above]{$5$};
\fill (5.6,2) circle(2pt) coordinate(F) node[above]{$6$};
\fill (3.2,1) circle(2pt) coordinate(G) node[above]{$7$};
\fill (6.4,1) circle(2pt) coordinate(H) node[above]{$8$};
\fill (7.2,2) circle(2pt) coordinate(I) node[above]{$9$};
\draw [-] (0,1.5) to [out = 0, in = 135] (B);
\draw [-] (B) to [out = -45, in = -135] (G);
\draw [-] (G) to [out = 45, in = 135] (D); 
\draw [-] (D) to [out = -45, in = -135] (H);
\draw [-] (H) -- (I);
\draw [->] (I) to [out = -45, in = 180] (8,1.5);
\draw [-] (0,0.5) to [out = 0, in = 135] (A); 
\draw [-] (A) to [out = -45, in = -135] (E);
\draw [-] (E) -- (F);
\draw [-] (F) -- (H);
\draw [->] (H) to [out = -45, in = 180] (8,0.5);
\draw [-] (0,-0.5) to [out = 0, in = -135] (A); 
\draw [-] (A) -- (C);
\draw [-] (C) -- (E);
\draw [->] (E) to [out = -45, in = 180] (8,-0.5);
}
\draw[->] (4,-0.8) -- (4,-1.3);
\draw (4,-1.1) circle(0pt) node[left]{$K_{4689}$} node[right]{$\mu_{3,7,3}$};
\end{scope}
\begin{scope}[>=latex,yshift=-230pt]
\path (1.8,2) node[circle]{2} coordinate(1) node[above=0.2em]{$1$};
\draw (1.8,2) circle[radius=0.15];
\path (3.6,2) node[circle]{2} coordinate(2) node[above=0.2em]{$2$};
\draw (3.6,2) circle[radius=0.15];
\path (5.4,2) node[circle]{2} coordinate(3) node[above=0.2em]{$3$};
\draw (5.4,2) circle[radius=0.15];
\path (7.2,2) node[circle]{2} coordinate(4) node[above=0.2em]{$4$};
\draw (7.2,2) circle[radius=0.15];
\draw (0.9,1) circle(2pt) coordinate(5) node[below]{$5$};
\draw (2.7,1) circle(2pt) coordinate(6) node[below]{$6$};
\draw (6.3,1) circle(2pt) coordinate(7) node[below]{$7$};
\draw (8.1,1) circle(2pt) coordinate(8) node[below]{$8$};
\draw (0,0) circle(2pt) coordinate(9) node[below]{$9$};
\draw (2.7,0) circle(2pt) coordinate(10) node[below]{$10$};
\draw (4.5,1) circle(2pt) coordinate(11) node[below]{$11$};
\draw (6.3,0) circle(2pt) coordinate(12) node[below]{$12$};
\qsarrow{1}{2}
\qsarrow{2}{3}
\qsarrow{3}{4}
\qarrow{5}{6}
\qarrow{6}{11}
\qarrow{11}{7}
\qarrow{7}{8}
\qarrow{9}{10}
\qarrow{10}{12}
\draw[->,dashed,shorten >=4pt,shorten <=2pt] (8) -- (4) [thick];
\qarrowsa{4}{7}
\qarrowsb{7}{3}
\qarrowsa{3}{11}
\qarrowsb{11}{2}
\qarrowsa{2}{6}
\qarrowsb{6}{1}
\draw[->,dashed,shorten >=2pt,shorten <=4pt] (1) -- (5) [thick];
\qdarrow{8}{12} 
\qarrow{12}{11}
\qarrow{11}{10}
\qarrow{10}{5}
\qdarrow{5}{9}
{\color{red}
\fill (0.9,0) circle(2pt) coordinate(A) node[below]{$1$};
\fill (1.8,1) circle(2pt) coordinate(B) node[above]{$2$};
\fill (2.7,2) circle(2pt) coordinate(C) node[below]{$3$};
\fill (7.2,1) circle(2pt) coordinate(D) node[above]{$4$};
\fill (4.5,0) circle(2pt) coordinate(E) node[above]{$5$};
\fill (6.3,2) circle(2pt) coordinate(F) node[above]{$6$};
\fill (3.6,1) circle(2pt) coordinate(G) node[above]{$7$};
\fill (5.4,1) circle(2pt) coordinate(H) node[above]{$8$};
\fill (4.5,2) circle(2pt) coordinate(I) node[above]{$9$};
\draw [-] (0,1.5) to [out = 0, in = 135] (B);
\draw [-] (B) to [out = -45, in = -135] (G);
\draw [-] (G) -- (I);
\draw [-] (I) -- (H);
\draw [-] (H) to [out = -45, in = -135] (D);
\draw [->] (D) to [out = 45, in = 180] (8.1,1.5);
\draw [-] (0,0.5) to [out = 0, in = 135] (A); 
\draw [-] (A) to [out = -45, in = -135] (E);
\draw [-] (E) -- (F);
\draw [-] (F) -- (D);
\draw [->] (D) to [out = -45, in = 180] (8.1,0.5);
\draw [-] (0,-0.5) to [out = 0, in = -135] (A); 
\draw [-] (A) -- (C);
\draw [-] (C) -- (E);
\draw [->] (E) to [out = -45, in = 180] (8.1,-0.5);
}
\draw[->] (4,-0.8) -- (4,-1.3);
\draw (4,-1.1) circle(0pt) node[left]{$K_{2379}$} node[right]{$\mu_{2,6,2}$};
\end{scope}
\begin{scope}[>=latex,yshift=-345pt]
\path (0,2) node[circle]{2} coordinate(1) node[above=0.2em]{$1$};
\draw (0,2) circle[radius=0.15];
\path (1.8,2) node[circle]{2} coordinate(2) node[above=0.2em]{$2$};
\draw (1.8,2) circle[radius=0.15];
\path (4.5,2) node[circle]{2} coordinate(3) node[above=0.2em]{$3$};
\draw (4.5,2) circle[radius=0.15];
\path (7.2,2) node[circle]{2} coordinate(4) node[above=0.2em]{$4$};
\draw (7.2,2) circle[radius=0.15];
\draw (0.9,1) circle(2pt) coordinate(5) node[below]{$5$};
\draw (2.7,1) circle(2pt) coordinate(6) node[below]{$6$};
\draw (6.3,1) circle(2pt) coordinate(7) node[below]{$7$};
\draw (8.1,1) circle(2pt) coordinate(8) node[below]{$8$};
\draw (0,0) circle(2pt) coordinate(9) node[below]{$9$};
\draw (2.7,0) circle(2pt) coordinate(10) node[below]{$10$};
\draw (4.5,1) circle(2pt) coordinate(11) node[below]{$11$};
\draw (6.3,0) circle(2pt) coordinate(12) node[below]{$12$};
\qsarrow{1}{2}
\qsarrow{2}{3}
\qsarrow{3}{4}
\qarrow{5}{6}
\qarrow{6}{11}
\qarrow{11}{7}
\qarrow{7}{8}
\qarrow{9}{10}
\qarrow{10}{12}
\draw[->,dashed,shorten >=4pt,shorten <=2pt] (8) -- (4) [thick];
\qarrowsa{4}{7}
\qarrowsb{7}{3}
\qarrowsa{3}{6}
\qarrowsb{6}{2}
\qarrowsa{2}{5}
\draw[->,dashed,shorten >=4pt,shorten <=2pt] (5) -- (1) [thick];
\qdarrow{8}{12} 
\qarrow{12}{11}
\qarrow{11}{10}
\qarrow{10}{5}
\qdarrow{5}{9}
{\color{red}
\fill (0.9,0) circle(2pt) coordinate(A) node[below]{$1$};
\fill (3.6,1) circle(2pt) coordinate(B) node[above]{$2$};
\fill (2.7,2) circle(2pt) coordinate(C) node[below]{$3$};
\fill (7.2,1) circle(2pt) coordinate(D) node[above]{$4$};
\fill (4.5,0) circle(2pt) coordinate(E) node[above]{$5$};
\fill (6.3,2) circle(2pt) coordinate(F) node[above]{$6$};
\fill (1.8,1) circle(2pt) coordinate(G) node[above]{$7$};
\fill (5.4,1) circle(2pt) coordinate(H) node[above]{$8$};
\fill (0.9,2) circle(2pt) coordinate(I) node[above]{$9$};
\draw [-] (0,1.5) to [out = 0, in = -135] (I);
\draw [-] (I) -- (G);
\draw [-] (G) to [out = -45, in = -135] (B);
\draw [-] (B) to [out = 45, in = 135] (H);
\draw [-] (H) to [out = -45, in = -135] (D);
\draw [->] (D) to [out = 45, in = 180] (8.1,1.5);
\draw [-] (0,0.5) to [out = 0, in = 135] (A); 
\draw [-] (A) to [out = -45, in = -135] (E);
\draw [-] (E) -- (F);
\draw [-] (F) -- (D);
\draw [->] (D) to [out = -45, in = 180] (8.1,0.5);
\draw [-] (0,-0.5) to [out = 0, in = -135] (A); 
\draw [-] (A) -- (C);
\draw [-] (C) -- (E);
\draw [->] (E) to [out = -45, in = 180] (8.1,-0.5);
}
\draw[->] (8.5,1) -- (9.5,1);
\draw (9,1) circle(0pt) node[below]{$R_{258}$} node[above]{$\mu_{11}$};
%\draw[->] (4,-0.8) -- (4,-1.3);
%\draw (4,-0.9) circle(0pt) node[right]{$R_{258}$};
\end{scope}
\begin{scope}[>=latex,xshift=290pt,yshift=-345pt]
\path (0,2) node[circle]{2} coordinate(1) node[above=0.2em]{$1$};
\draw (0,2) circle[radius=0.15];
\path (2,2) node[circle]{2} coordinate(2) node[above=0.2em]{$2$};
\draw (2,2) circle[radius=0.15];
\path (4,2) node[circle]{2} coordinate(3) node[above=0.2em]{$3$};
\draw (4,2) circle[radius=0.15];
\path (6,2) node[circle]{2} coordinate(4) node[above=0.2em]{$4$};
\draw (6,2) circle[radius=0.15];
\draw (1,1) circle(2pt) coordinate(5) node[below]{$5$};
\draw (3,1) circle(2pt) coordinate(6) node[below]{$6$};
\draw (5,1) circle(2pt) coordinate(7) node[below]{$7$};
\draw (7,1) circle(2pt) coordinate(8) node[below]{$8$};
\draw (0,0) circle(2pt) coordinate(9) node[below]{$9$};
\draw (2,0) circle(2pt) coordinate(10) node[below]{$10$};
\draw (4,0) circle(2pt) coordinate(11) node[below]{$11$};
\draw (6,0) circle(2pt) coordinate(12) node[below]{$12$};
\qsarrow{1}{2}
\qsarrow{2}{3}
\qsarrow{3}{4}
\qarrow{5}{6}
\qarrow{6}{7}
\qarrow{7}{8}
\qarrow{9}{10}
\qarrow{10}{11}
\qarrow{11}{12}
\draw[->,dashed,shorten >=4pt,shorten <=2pt] (8) -- (4) [thick];
\qarrowsa{4}{7}
\qarrowsb{7}{3}
\qarrowsa{3}{6}
\qarrowsb{6}{2}
\qarrowsa{2}{5}
\draw[->,dashed,shorten >=4pt,shorten <=2pt] (5) -- (1) [thick];
\qdarrow{8}{12} 
\qarrow{12}{7}
\qarrow{7}{11}
\qarrow{11}{6}
\qarrow{6}{10}
\qarrow{10}{5}
\qdarrow{5}{9}
{\color{red}
\fill (1,0) circle(2pt) coordinate(A) node[below]{$1$};
\fill (5,0) circle(2pt) coordinate(B) node[above]{$2$};
\fill (3,2) circle(2pt) coordinate(C) node[below]{$3$};
\fill (6,1) circle(2pt) coordinate(D) node[above]{$4$};
\fill (4,1) circle(2pt) coordinate(E) node[above]{$5$};
\fill (5,2) circle(2pt) coordinate(F) node[above]{$6$};
\fill (2,1) circle(2pt) coordinate(G) node[above]{$7$};
\fill (3,0) circle(2pt) coordinate(H) node[above]{$8$};
\fill (1,2) circle(2pt) coordinate(I) node[above]{$9$};
\draw [-] (0,1.5) to [out = 0, in = -135] (I);
\draw [-] (I) -- (H);
\draw [-] (H) to [out = -45, in = -135] (B);
\draw [-] (B) -- (D);
\draw [->] (D) to [out = 45, in = 180] (7,1.5);
\draw [-] (0,0.5) to [out = 0, in = 135] (A); 
\draw [-] (A) to [out = -45, in = -135] (H);
\draw [-] (H) -- (F);
\draw [-] (F) -- (D);
\draw [->] (D) to [out = -45, in = 180] (7,0.5);
\draw [-] (0,-0.5) to [out = 0, in = -135] (A); 
\draw [-] (A) -- (C);
\draw [-] (C) -- (B);
\draw [->] (B) to [out = -45, in = 180] (7,-0.5);
}
\draw[->] (3.3,2.7) -- (3.3,3.2);
\draw (3.3,2.9) circle(0pt) node[left]{$R_{178}$} node[right]{$\mu_{10}$};
\end{scope}
\begin{scope}[>=latex,xshift=270pt,yshift=-230pt]
\path (0,2) node[circle]{2} coordinate(1) node[above=0.2em]{$1$};
\draw (0,2) circle[radius=0.15];
\path (2.7,2) node[circle]{2} coordinate(2) node[above=0.2em]{$2$};
\draw (2.7,2) circle[radius=0.15];
\path (5.4,2) node[circle]{2} coordinate(3) node[above=0.2em]{$3$};
\draw (5.4,2) circle[radius=0.15];
\path (7.2,2) node[circle]{2} coordinate(4) node[above=0.2em]{$4$};
\draw (7.2,2) circle[radius=0.15];
\draw (0.9,1) circle(2pt) coordinate(5) node[below left]{$5$};
\draw (4.5,1) circle(2pt) coordinate(6) node[below]{$6$};
\draw (6.3,1) circle(2pt) coordinate(7) node[below]{$7$};
\draw (8.1,1) circle(2pt) coordinate(8) node[below]{$8$};
\draw (0.9,0) circle(2pt) coordinate(9) node[below]{$9$};
\draw (2.7,1) circle(2pt) coordinate(10) node[below]{$10$};
\draw (4.5,0) circle(2pt) coordinate(11) node[below]{$11$};
\draw (7.2,0) circle(2pt) coordinate(12) node[below]{$12$};
\qsarrow{1}{2}
\qsarrow{2}{3}
\qsarrow{3}{4}
\qarrow{5}{10}
\qarrow{10}{6}
\qarrow{6}{7}
\qarrow{7}{8}
\qarrow{9}{11}
\qarrow{11}{12}
\draw[->,dashed,shorten >=4pt,shorten <=2pt] (8) -- (4) [thick];
\qarrowsa{4}{7}
\qarrowsb{7}{3}
\qarrowsa{3}{6}
\qarrowsb{6}{2}
\qarrowsa{2}{5}
\draw[->,dashed,shorten >=4pt,shorten <=2pt] (5) -- (1) [thick];
\qdarrow{8}{12} 
\qarrow{12}{7}
\qarrow{7}{11}
\qarrow{11}{10}
\qarrow{10}{9}
\qdarrow{9}{5}
{\color{red}
\fill (3.6,1) circle(2pt) coordinate(A) node[below]{$1$};
\fill (6.3,0) circle(2pt) coordinate(B) node[above]{$2$};
\fill (4.5,2) circle(2pt) coordinate(C) node[below]{$3$};
\fill (7.2,1) circle(2pt) coordinate(D) node[above]{$4$};
\fill (5.4,1) circle(2pt) coordinate(E) node[above]{$5$};
\fill (6.3,2) circle(2pt) coordinate(F) node[above]{$6$};
\fill (2.7,0) circle(2pt) coordinate(G) node[above]{$7$};
\fill (1.8,1) circle(2pt) coordinate(H) node[above]{$8$};
\fill (0.9,2) circle(2pt) coordinate(I) node[above]{$9$};
\draw [-] (0,1.5) to [out = 0, in = -135] (I);
\draw [-] (I) -- (G);
\draw [-] (G) to [out = -45, in = -135] (B);
\draw [-] (B) -- (D);
\draw [->] (D) to [out = 45, in = 180] (8.1,1.5);
\draw [-] (0,0.5) to [out = 0, in = -135] (H); 
\draw [-] (H) to [out = 45, in = 135] (A);
\draw [-] (A) to [out = -45, in = -135] (E);
\draw [-] (E) -- (F);
\draw [-] (F) -- (D);
\draw [->] (D) to [out = -45, in = 180] (8.1,0.5);
\draw [-] (0,-0.5) to [out = 0, in = -135] (G); 
\draw [-] (G) -- (C);
\draw [-] (C) -- (B);
\draw [->] (B) to [out = -45, in = 180] (8.1,-0.5);
}
\draw[->] (4,2.7) -- (4,3.2);
\draw (4,2.9) circle(0pt) node[left]{$K_{1356}$} node[right]{$\mu_{3,6,3}$};
\end{scope}
\begin{scope}[>=latex,xshift=270pt,yshift=-115pt]
\path (0,2) node[circle]{2} coordinate(1) node[above=0.2em]{$1$};
\draw (0,2) circle[radius=0.15];
\path (1.8,2) node[circle]{2} coordinate(2) node[above=0.2em]{$2$};
\draw (1.8,2) circle[radius=0.15];
\path (3.6,2) node[circle]{2} coordinate(3) node[above=0.2em]{$3$};
\draw (3.6,2) circle[radius=0.15];
\path (6.3,2) node[circle]{2} coordinate(4) node[above=0.2em]{$4$};
\draw (6.3,2) circle[radius=0.15];
\draw (0.9,1) circle(2pt) coordinate(5) node[below left]{$5$};
\draw (4.5,1) circle(2pt) coordinate(6) node[below]{$6$};
\draw (6.3,1) circle(2pt) coordinate(7) node[below]{$7$};
\draw (8.1,1) circle(2pt) coordinate(8) node[below]{$8$};
\draw (0.9,0) circle(2pt) coordinate(9) node[below]{$9$};
\draw (2.7,1) circle(2pt) coordinate(10) node[below]{$10$};
\draw (4.5,0) circle(2pt) coordinate(11) node[below]{$11$};
\draw (7.2,0) circle(2pt) coordinate(12) node[below]{$12$};
\qsarrow{1}{2}
\qsarrow{2}{3}
\qsarrow{3}{4}
\qarrow{5}{10}
\qarrow{10}{6}
\qarrow{6}{7}
\qarrow{7}{8}
\qarrow{9}{11}
\qarrow{11}{12}
\draw[->,dashed,shorten >=4pt,shorten <=2pt] (8) -- (4) [thick];
\qarrowsa{4}{6}
\qarrowsb{6}{3}
\qarrowsa{3}{10}
\qarrowsb{10}{2}
\qarrowsa{2}{5}
\draw[->,dashed,shorten >=4pt,shorten <=2pt] (5) -- (1) [thick];
\qdarrow{8}{12} 
\qarrow{12}{7}
\qarrow{7}{11}
\qarrow{11}{10}
\qarrow{10}{9}
\qdarrow{9}{5}
{\color{red}
\fill (5.4,1) circle(2pt) coordinate(A) node[below]{$1$};
\fill (6.3,0) circle(2pt) coordinate(B) node[above]{$2$};
\fill (4.5,2) circle(2pt) coordinate(C) node[below]{$3$};
\fill (7.2,1) circle(2pt) coordinate(D) node[above]{$4$};
\fill (3.6,1) circle(2pt) coordinate(E) node[above]{$5$};
\fill (2.7,2) circle(2pt) coordinate(F) node[above]{$6$};
\fill (2.7,0) circle(2pt) coordinate(G) node[above]{$7$};
\fill (1.8,1) circle(2pt) coordinate(H) node[above]{$8$};
\fill (0.9,2) circle(2pt) coordinate(I) node[above]{$9$};
\draw [-] (0,1.5) to [out = 0, in = -135] (I);
\draw [-] (I) -- (G);
\draw [-] (G) to [out = -45, in = -135] (B);
\draw [-] (B) -- (D);
\draw [->] (D) to [out = 45, in = 180] (8.1,1.5);
\draw [-] (0,0.5) to [out = 0, in = -135] (H); 
\draw [-] (H) -- (F);
\draw [-] (F) -- (E);
\draw [-] (E) to [out = -45, in = -135] (A);
\draw [-] (A) to [out = 45, in = 135] (D);
\draw [->] (D) to [out = -45, in = 180] (8.1,0.5);
\draw [-] (0,-0.5) to [out = 0, in = -135] (G); 
\draw [-] (G) -- (C);
\draw [-] (C) -- (B);
\draw [->] (B) to [out = -45, in = 180] (8.1,-0.5);
}
\draw[->] (4,2.7) -- (4,3.2);
\draw (4,2.9) circle(0pt) node[left]{$R_{124}$} node[right]{$\mu_7$};
\end{scope}
\begin{scope}[>=latex,xshift=270pt]
\path (0,2) node[circle]{2} coordinate(1) node[above=0.2em]{$1$};
\draw (0,2) circle[radius=0.15];
\path (2,2) node[circle]{2} coordinate(2) node[above=0.2em]{$2$};
\draw (2,2) circle[radius=0.15];
\path (4,2) node[circle]{2} coordinate(3) node[above=0.2em]{$3$};
\draw (4,2) circle[radius=0.15];
\path (6,2) node[circle]{2} coordinate(4) node[above=0.2em]{$4$};
\draw (6,2) circle[radius=0.15];
\draw (1,1) circle(2pt) coordinate(5) node[below]{$5$};
\draw (3,1) circle(2pt) coordinate(6) node[below]{$6$};
\draw (5,1) circle(2pt) coordinate(7) node[below]{$7$};
\draw (7,1) circle(2pt) coordinate(8) node[below]{$8$};
\draw (2,0) circle(2pt) coordinate(9) node[below]{$9$};
\draw (4,0) circle(2pt) coordinate(10) node[below]{$10$};
\draw (6,0) circle(2pt) coordinate(11) node[below]{$11$};
\draw (8,0) circle(2pt) coordinate(12) node[below]{$12$};
\qsarrow{1}{2}
\qsarrow{2}{3}
\qsarrow{3}{4}
\qarrow{5}{6}
\qarrow{6}{7}
\qarrow{7}{8}
\qarrow{9}{10}
\qarrow{10}{11}
\qarrow{11}{12}
\draw[->,shorten >=2pt,shorten <=4pt] (4) -- (7) [thick];
\draw[->,shorten >=4pt,shorten <=2pt] (7) -- (3) [thick];
\draw[->,shorten >=2pt,shorten <=4pt] (3) -- (6) [thick];
\draw[->,shorten >=4pt,shorten <=2pt] (6) -- (2) [thick];
\draw[->,shorten >=2pt,shorten <=4pt] (2) -- (5) [thick];
\qarrow{8}{11}
\qarrow{11}{7}
\qarrow{7}{10}
\qarrow{10}{6}
\qarrow{6}{9}
\draw[->,dashed,shorten >=4pt,shorten <=2pt] (8) -- (4) [thick];
\draw[->,dashed,shorten >=4pt,shorten <=2pt] (5) -- (1) [thick];
\qdarrow{9}{5}
\qdarrow{12}{8} 
{\color{red}
\fill (7,0) circle(2pt) coordinate(A) node[below]{$1$};
\fill (6,1) circle(2pt) coordinate(B) node[above]{$2$};
\fill (5,2) circle(2pt) coordinate(C) node[below]{$3$};
\fill (5,0) circle(2pt) coordinate(D) node[above]{$4$};
\fill (4,1) circle(2pt) coordinate(E) node[above]{$5$};
\fill (3,2) circle(2pt) coordinate(F) node[above]{$6$};
\fill (3,0) circle(2pt) coordinate(G) node[above]{$7$};
\fill (2,1) circle(2pt) coordinate(H) node[above]{$8$};
\fill (1,2) circle(2pt) coordinate(I) node[above]{$9$};
\draw [-] (0,1.5) to [out = 0, in = -135] (I);
\draw [-] (I) -- (G);
\draw [-] (G) to [out = -45, in = -135] (D);
\draw [-] (D) -- (B);
\draw [->] (B) to [out = 45, in = 180] (8,1.5);
\draw [-] (0,0.5) to [out = 0, in = -135] (H);
\draw [-] (H) -- (F); 
\draw [-] (F) to [out = 45, in = -135] (F);
\draw [-] (F) -- (D);
\draw [-] (D) to [out = -45, in = -135] (A);
\draw [->] (A) to [out = 45, in = 180] (8,0.5);
\draw [-] (0,-0.5) to [out = 0, in = -135] (G); 
\draw [-] (G) -- (C);
\draw [-] (C) -- (A);
\draw [->] (A) to [out = -45, in = 180] (8,-0.5);
}
\end{scope}
\end{tikzpicture}
}
\]
\caption{3D reflection transformation (RHS of \eqref{eq:reflection-RK-eq}).}
\label{fig:reflection-eq2}
\end{figure}

\begin{lem}\label{lem:K-trop}
Let $(J_{123123123},y)$ be a tropical $y$-seed for $J_{123123123}$.
It holds that 
\begin{align}
\label{eq:K-trop-id}
\begin{split}
&\mu_{11} \, \mu_{2,10,2} \, \mu_{3,6,3}\, \mu_{11}\, \mu_7 \,\mu_{2,6,2}\, \mu_{10} (J_{123123123},y) 
\\ &\quad 
= \mu_7 \,\mu_{3,6,3} \,\mu_{10} \,\mu_{11} \,\mu_{2,6,2}\, \mu_{3,7,3} \,\mu_{11}(J_{123123123},y).
\end{split}
\end{align}
The tropical signs corresponding to $\mu_i$ are all positive, and those corresponding to $\mu_{i,j,i}$ are all $(+,+,-)$.
In particular, the final quiver after mutations is $J_{321321321}$.
\end{lem}

\begin{proof}
We calculate the two transformations of tropical $y$-variables using \eqref{eq:K-trop-mutation}. 
We use a shorthand $y_{i_1,i_2,\ldots,i_p}$ for $y_{i_1} y_{i_2} \cdots y_{i_p}$.
\begin{align*}
\text{LHS : ~}&(y_1,y_2,y_3,y_4,y_5,y_6,y_7,y_8,y_9,\underline{y_{10}},y_{11},y_{12})
\\ \displaybreak[0]
&\stackrel{\mu_{10}}{\mapsto} (y_1,\underline{y_2},y_3,y_4,y_5,\underline{y_{6,10}},y_7,y_8,y_{9,10},y_{10}^{-1},y_{11},y_{12})
\\ \displaybreak[0]
&\stackrel{\mu_{2,6,2}}{\mapsto}  (y_1,y_2,y_3,y_4,y_5,y_{2,6,10}^{-1},\underline{y_{2,6,7,10}},y_8,y_{9,10}, y_{6},y_{11},y_{12})
\\ \displaybreak[0]
&\stackrel{\mu_{7}}{\mapsto} (y_1,y_2,y_3,y_4,y_5,y_7,y_{2,6,7,10}^{-1},y_8,y_{9,10}, y_{6},\underline{y_{11}},y_{2,6,7,10,12})
\\ %\displaybreak[0]
&\stackrel{\mu_{11}}{\mapsto} (y_1,y_2,\underline{y_3},y_4,y_5,\underline{y_{7,11}},y_{2,6,7,10}^{-1},y_8,y_{9,10,11}, y_{6},y_{11}^{-1},y_{2,6,7,10,12})
\\ %\displaybreak[0]
&\stackrel{\mu_{3,6,3}}{\mapsto} (y_1,\underline{y_2},y_3,y_4,y_5,y_{3,7,11}^{-1},y_{2,6,7,10}^{-1},y_{3,7,8,11},y_{9,10,11}, \underline{y_{6}},y_{7},y_{2,6,7,10,12})
\\ %\displaybreak[0]
&\stackrel{\mu_{2,10,2}}{\mapsto} (y_1,y_2,y_3,y_4,y_{5,6},y_{3,7,11}^{-1},y_{2,6,7,10}^{-1},y_{3,7,8,11},y_{9,10,11}, y_{2,6}^{-1},\underline{y_{2,6,7}},y_{2,6,7,10,12})
\\ %\displaybreak[0]
&\stackrel{\mu_{11}}{\mapsto} (y_1,y_2,y_3,y_4,y_{5,6},y_{3,7,11}^{-1},y_{10}^{-1},y_{3,7,8,11},y_{9,10,11}, y_7,y_{2,6,7}^{-1},y_{2,6,7,10,12}),
\\ %\displaybreak[0]
\text{RHS : ~}&(y_1,y_2,y_3,y_4,y_5,y_6,y_7,y_8,y_9,,y_{10},\underline{y_{11}},y_{12})
\\ \displaybreak[0]
&\stackrel{\mu_{11}}{\mapsto} (y_1,y_2,\underline{y_3},y_4,y_5,y_{6},\underline{y_{7,11}},y_8,y_{9},y_{10,11},y_{11}^{-1},y_{12})
\\ \displaybreak[0]
&\stackrel{\mu_{3,7,3}}{\mapsto} (y_1,\underline{y_2},y_3,y_4,y_{5},\underline{y_{6}},y_{3,7,11}^{-1},y_{3,7,8,11},y_{9}, y_{10,11},y_{11}^{-1},y_{12})
\\ \displaybreak[0]
&\stackrel{\mu_{2,6,2}}{\mapsto} (y_1,y_2,y_3,y_4,y_{5,6},y_{2,6}^{-1},y_{3,7,11}^{-1},y_{3,7,8,11},y_{9}, y_{10,11},\underline{y_{2,6,7}},y_{12})
\\ \displaybreak[0]
&\stackrel{\mu_{11}}{\mapsto} (y_1,y_2,y_3,y_4,y_{5,6},y_{7},y_{3,7,11}^{-1},y_{3,7,8,11},y_{9}, \underline{y_{10,11}},y_{2,6,7}^{-1},y_{2,6,7,12})
\\ %\displaybreak[0]
&\stackrel{\mu_{10}}{\mapsto} (y_1,y_2,\underline{y_3},y_4,y_{5,6},\underline{y_{7,10,11}},y_{3,7,11}^{-1},y_{3,7,8,11},y_{9,10,11}, y_{10,11}^{-1},y_{2,6,7}^{-1},y_{2,6,7,12}) 
\\ \displaybreak[0]
&\stackrel{\mu_{3,6,3}}{\mapsto} (y_1,y_2,y_3,y_4,y_{5,6},y_{3,7,10,11}^{-1},\underline{y_{10}},y_{3,7,8,11},y_{9,10,11}, y_{7},y_{2,6,7}^{-1},y_{2,6,7,12})
\\ \displaybreak[0]
&\stackrel{\mu_{7}}{\mapsto} (y_1,y_2,y_3,y_4,y_{5,6},y_{3,7,11}^{-1},y_{10}^{-1},y_{3,7,8,11},y_{9,10,11}, y_{7},y_{2,6,7}^{-1},y_{2,6,7,10,12}).
\end{align*}
Thus we have proved \eqref{eq:K-trop-id}.
The tropical sign-sequences for these mutation sequences are obtained from the underlined $y$-variables.
One sees that the final quiver is $J_{321321321}$ in Figure \ref{fig:reflection-eq}.
\end{proof} 

Due to Theorem \ref{thm:period}, Theorem \ref{thm:id-mono-qdilog} and \eqref{eq:K-decomp} 
we obtain the following proposition as a corollary of Lemma \ref{lem:K-trop}.
\begin{prop}
Let $(J_{123123123},Y)$ be a quantum $y$-seed for $J_{123123123}$. 
It holds that
\begin{align}\label{eq:K-q-id}
\begin{split}
&\mu_{11} \, \mu_{2,10,2} \, \mu_{3,6,3}\, \mu_{11}\, \mu_7 \,\mu_{2,6,2}\, \mu_{10} (J_{123123123},Y) 
\\ &\qquad 
= \mu_7 \,\mu_{3,6,3} \,\mu_{10} \,\mu_{11} \,\mu_{2,6,2}\, \mu_{3,7,3} \,\mu_{11} (J_{123123123},Y).
\end{split}
\end{align}
In particular, we have identities
\begin{align}
\label{K-mono-id}
\begin{split}
&\tau_{10,+} \,\tau_{2,6,2,++-} \, \tau_{7,+} \, \tau_{11,+} \, \tau_{3,6,3, ++-} \, \tau_{2,10,2,+ + -} \, \tau_{11,+} 
\\ &\qquad 
= \tau_{11,+} \, \tau_{3,7,3,++-} \,\tau_{2,6,2,++-} \,\tau_{11,+} \,\tau_{10,+}  \,\tau_{3,6,3,++-} \,\tau_{7,+}, 
\end{split}
\end{align}
as a morphism from $\mathcal{T}(J_{321321321})$ to $\mathcal{T}(J_{123123123})$, and  
\begin{align}
\label{K-qdilog-id}
\begin{split}
&\Psi_q(\rY^{e_{10}}) 
\cdot \Psi_{q^2}(\rY^{e_{2}}) \Psi_q(\rY^{e_6+e_{10}}) \Psi_{q^2}(\rY^{e_{2}})^{-1} 
\cdot \Psi_{q}(\rY^{e_{2}+e_{6}+e_{7}+e_{10}}) \cdot \Psi_{q}(\rY^{e_{11}})
\\ &\quad 
\cdot \Psi_{q^2}(\rY^{e_{3}}) \Psi_q(\rY^{e_7+e_{11}}) \Psi_{q^2}(\rY^{e_{3}})^{-1}
\cdot \Psi_{q^2}(\rY^{e_{2}}) \Psi_q(\rY^{e_6}) \Psi_{q^2}(\rY^{e_{2}})^{-1}
\cdot \Psi_{q}(\rY^{e_2+e_6+e_7})
\\
&= \Psi_q(\rY^{e_{11}}) 
\cdot \Psi_{q^2}(\rY^{e_{3}}) \Psi_q(\rY^{e_7+e_{11}}) \Psi_{q^2}(\rY^{e_{3}})^{-1}
\cdot \Psi_{q^2}(\rY^{e_{2}}) \Psi_q(\rY^{e_6}) \Psi_{q^2}(\rY^{e_{2}})^{-1}
\\ &\quad 
\cdot \Psi_{q}(\rY^{e_{2}+e_{6}+e_{7}}) \cdot \Psi_{q}(\rY^{e_{10}+e_{11}}) 
\cdot \Psi_{q^2}(\rY^{e_{3}}) \Psi_q(\rY^{e_7+e_{10}+e_{11}}) \Psi_{q^2}(\rY^{e_{3}})^{-1} 
\cdot \Psi_{q}(\rY^{e_{10}}),
\end{split}
\end{align}
in the completion $\hat{\mathbb{A}}^{\! \times}(J_{123123123})$.
Here we have set $\tau_{i,j,i, + + -} := \tau_{i,+} \tau_{j,+} \tau_{i,-}$.
\end{prop}

Let $S := \{1,2,\ldots,9\}$ be the set of crossings  in the wiring diagrams in Figure \ref{fig:reflection-eq}.
To the crossing $i \in S$ we assign a pair of noncommuting variables $(p_i,u_i)$ satisfying 
\begin{align}\label{pu3}
[p_i,u_j] 
= \begin{cases}
2 \hbar & i=j=3,6,9, \\
\hbar & i=j=1,2,4,5,7,8, \\
0 & \text{otherwise},
\end{cases} 
\qquad 
[p_i,p_j] = [u_i,u_j] = 0.
\end{align}
Let $\mathcal{W}(C_3)$ 
be the algebra over $\C$ generated by the $q$-commuting $q$-Weyl pairs $e^{\pm p_i}, e^{\pm u_i}~(i \in S)$. In the same way as \S 4, we define a group $N(C_3)$ generated by \eqref{N(A3)gen}, and the symmetric group $\mathfrak{S}_9$ generated by permutations $\rho_{ij} ~(i,j \in S)$. The group $\mathfrak{S}_9$ acts on $N(C_3)$ by adjoint action, and a semidirect product $N(C_3) \rtimes \mathfrak{S}_9$ acts on $\mathcal{W}(C_3)$ by adjoint action. 
We use the complex parameters $\lambda_i, \kappa_i~(i \in S)$ and the notation \eqref{para}.

Corresponding to the transformation $R_{ijk}$ of the wiring diagrams, 
we let $\pi_{ijk}$ denote the isomorphism of $\mathcal{W}(C_3)$ 
given by \eqref{eq:R-pi-general}.
Similarly, the isomorphism $\pi^K_{ijkl}$ of $\mathcal{W}(C_3)$ associated with $K_{ijkl}$ is given by 
\begin{align}\label{eq:K-pi-general}
\pi^K_{ijkl}:
\begin{cases}
\;p_i \mapsto p_i+ \lambda_{jl}, 
\quad 
p_j \mapsto p_l+ 2p_i+ \lambda_{jl},
\\
\;p_k  \mapsto p_k-  \lambda_{jl},
\quad
p_l \mapsto p_j-2p_i-  \lambda_{jl},
\\
\;u_i \mapsto u_i+u_j-u_l,
\quad 
u_j  \mapsto u_l,
\\
\; u_k  \mapsto u_k,
\quad 
u_l \mapsto u_j,
\end{cases}
\end{align}
as the extension of \eqref{eq:K-pi}.
Let $P_{ijk}$ be the one defined in \eqref{eq:R-Pop-general}, which is now understood as an element of 
$N(C_3) \rtimes \mathfrak{S}_9$. 
From \eqref{K-Pop} we set 
\begin{equation}\label{eq:K-Pop-general}
\begin{split}
P^K_{ijkl} &= \rho_{jl}\, e^{\frac{1}{\hbar}p_i(u_l-u_j)} 
e^{\frac{\lambda_{jl}}{2\hbar}(2u_k-2u_i+u_l-u_j)}
\in N(C_3) \rtimes \mathfrak{S}_9.
\end{split}
\end{equation}
By combining $\phi$ \eqref{eq:phi-K} and $\phi$ \eqref{eq:phi_kappa},
we introduce an embedding $\phi: \mathcal{Y}(J_{123123123}) \hookrightarrow \mathrm{Frac}\mathcal{W}(C_3)$ by
\begin{align}\label{eq:phi-Y-C3}
\phi: 
\begin{cases}
Y_1 \mapsto \kappa_3^{-1} \,e^{p_3-u_3-2 p_2}, ~~ 
Y_2 \mapsto \kappa_3 \kappa_6^{-1} e^{p_3+u_3+p_6-u_6-2 p_5}, ~~
Y_3 \mapsto \kappa_6 \kappa_9^{-1} e^{p_6+u_6+p_9-u_9-2 p_8},
\\
Y_4 \mapsto \kappa_9 \, e^{p_9+u_9} , ~~
Y_5 \mapsto \kappa_2^{-1} e^{p_2-u_2-p_1}, ~~
Y_6 \mapsto \kappa_2 \kappa_5^{-1} e^{p_2+u_2+p_5-u_5-p_3-p_4},
\\ 
Y_7 \mapsto \kappa_5 \kappa_8^{-1} e^{p_5+u_5+p_8-u_8-p_6-p_7}, ~~ 
Y_8 \mapsto \kappa_8 \,e^{p_8+u_8-p_9}, ~~
Y_9 \mapsto \kappa_1^{-1} e^{p_1-u_1},
\\
Y_{10} \mapsto \kappa_1 \kappa_4^{-1} e^{p_1+u_1+p_4-u_4-p_2},~~
Y_{11} \mapsto \kappa_4 \kappa_7^{-1} e^{p_4+u_4+p_{7}-u_{7}-p_5}, ~~
Y_{12} \mapsto \kappa_7 \,e^{p_7+u_7-p_8}.
\end{cases}
\end{align}
Rules similar to \eqref{fig:A_3-phi} can also be formulated for this $\phi$ involving vertices with weight 2. However, due to increased complexity, we are omitting those details here and providing only the results.

In the same manner as \S 4, we let $\mathcal{W}_{\mathbb{A}}(J_{123123123}) \subset \mathcal{W}(C_3)$ be the image of $\mathbb{A}(J_{123123123})$ by $\phi$, and $\hat{\mathcal{W}}_{\mathbb{A}}(J_{123123123})$ be the completion of $\mathcal{W}_{\mathbb{A}}(J_{123123123})$ with respect to the ideal generated by $\phi(Y_i) ~(i \in I)$. 
The map $\phi$ induces the map from $\hat{\mathbb{A}}(J_{123123123})$ to  $\hat{\mathcal{W}}_{\mathbb{A}}(J_{123123123})$, and we define $\hat{\mathcal{W}}_{\mathbb{A}}^\times(J_{123123123})$ to be the image of $\hat{\mathbb{A}}^{\! \times}(J_{123123123})$ by the induced map. In the remainder of this section, we abbreviate the dependence on $J_{123123123}$. See \eqref{A-W-relation} for the relations among the noncommuting algebras. 

The following is proved by direct calculation using \eqref{eq:R-pi-general} and  \eqref{eq:K-pi-general}.
\begin{lem}\label{le:nK}
The morphisms $\pi_{ijk}$ and $\pi^K_{ijkl}$ on $\mathcal{W}(C_3)$ satisfy the 3D reflection equation:
\begin{align}\label{eq:K-pi-id}
\pi_{457} \pi^K_{4689} \pi^K_{2379} \pi_{258} \pi_{178} \pi^K_{1356} \pi_{124} 
= \pi_{124} \pi^K_{1356} \pi_{178} \pi_{258} \pi^K_{2379} \pi^K_{4689} \pi_{457}.
\end{align}
\end{lem}

Let $\mathcal{R}_{ijk}$ be the one in \eqref{eq:R-ad}.
From \eqref{eq:K-decomp} and \eqref{eq:K-Pop-general}, we define $\mathcal{K}_{ijkl}= \mathcal{K}(\lambda_i,\lambda_j,\lambda_k,\lambda_l)_{ijkl}$ as 
\begin{align}\label{eq:K-ad}
\begin{split}
\mathcal{K}_{ijkl} &= \Psi_{q^2}(e^{p_j+u_j+p_l-u_l-2p_k+ \lambda_{jl}}) 
\Psi_q(e^{p_i+u_i+p_k-u_k-p_j+\lambda_{ik}}) 
\Psi_{q^2}(e^{p_j+u_j+p_l-u_l-2p_k+\lambda_{jl}})^{-1}  
\\
& \qquad   \times  \rho_{jl} \,e^{\frac{1}{\hbar}p_i(u_l-u_j)}
e^{\frac{\lambda_{jl}}{2\hbar}(2u_k-2u_i+u_l-u_j)},
\end{split}
\end{align}
where the factor in the second line is $P^K_{ijkl}$ in \eqref{eq:K-Pop-general}. 
Note that $\mathcal{K}_{ijkl}$ depends on the parameters $\lambda_i$ 
only through their differences similarly to $\mathcal{R}_{ijk}$ \eqref{eq:R-ad}.
Now we present the main result of this section.

\begin{thm}\label{thm:K-reflection}
The operators $P_{ijk}$ \eqref{eq:R-Pop-general} and $P^K_{ijkl}$ \eqref{eq:K-Pop-general} 
satisfy the 3D reflection equation in $N(C_3) \rtimes \mathfrak{S}_9$:
\begin{align}\label{eq:K-monoad-id}
P_{457} P^K_{4689} P^K_{2379} P_{258} P_{178} P^K_{1356} P_{124} 
= P_{124} P^K_{1356} P_{178} P_{258} P^K_{2379} P^K_{4689} P_{457}.
\end{align}
The operators $\mathcal{R}_{ijk}$ \eqref{eq:R-ad} and $\mathcal{K}_{ijkl}$ \eqref{eq:K-ad} 
satisfy the 3D reflection equation in $\hat{\mathcal{W}}^{\times}_{\mathbb{A}} \rtimes G$:
\begin{align}\label{eq:RK-full-id}
\begin{split}
&\mathcal{R}(\lambda_{457})_{457} \mathcal{K}(\lambda_{4689})_{4689} \mathcal{K}(\lambda_{2379})_{2379} \mathcal{R}(\lambda_{258})_{258} \mathcal{R}(\lambda_{178})_{178} \mathcal{K}(\lambda_{1356})_{1356}\mathcal{R}(\lambda_{124})_{124} 
\\
& ~~ = \mathcal{R}(\lambda_{124})_{124} \mathcal{K}(\lambda_{1356})_{1356} \mathcal{R}(\lambda_{178})_{178} \mathcal{R}(\lambda_{258})_{258} \mathcal{K}(\lambda_{2379})_{2379} \mathcal{K}(\lambda_{4689})_{4689} \mathcal{R}(\lambda_{457})_{457}.
\end{split}
\end{align}
where $\lambda_{ijk} = (\lambda_i,\lambda_j,\lambda_k)$ and 
$\lambda_{ijkl} = (\lambda_i,\lambda_j,\lambda_k,\lambda_l)$.
\end{thm} 

\begin{proof}
For simplicity we outline the calculation of the two sides of (\ref{eq:K-monoad-id}) when $\lambda_j=0$ for all $j$.
\begin{align*}
\begin{split}%\label{eq:KP-right}
\text{LHS}
&= \rho_{57}\, e^{\frac{1}{\hbar}p_4(u_7-u_5)}
\rho_{69}\, e^{\frac{1}{\hbar}p_4(u_9-u_6)}
\rho_{39}\, e^{\frac{1}{\hbar}p_2(u_9-u_3)}
\rho_{58}\, e^{\frac{1}{\hbar}p_2(u_8-u_5)}
\\
& \qquad \cdot 
\rho_{78}\, e^{\frac{1}{\hbar}p_1(u_8-u_7)}
\rho_{36}\, e^{\frac{1}{\hbar}p_1(u_6-u_3)}
\rho_{24}\, e^{\frac{1}{\hbar}p_1(u_4-u_2)}
\\ 
&= e^{\frac{1}{\hbar}p_4(u_5-u_7)} e^{\frac{1}{\hbar}p_4(u_6-u_9)}
 e^{\frac{1}{\hbar}p_2(u_3-u_6)} e^{\frac{1}{\hbar}p_2(u_7-u_8)}
e^{\frac{1}{\hbar}p_1(u_5-u_7)} e^{\frac{1}{\hbar}p_1(u_6-u_4)} e^{\frac{1}{\hbar}p_1(u_2-u_4)}
\\
& \qquad \cdot 
\rho_{57} \rho_{69} \rho_{39} \rho_{58} \rho_{78} \rho_{36} \rho_{24},
\end{split}
\\
\begin{split}%\label{eq:KP-left}
\text{RHS}
&= \rho_{24}\, e^{\frac{1}{\hbar}p_1(u_4-u_2)}
\rho_{36}\, e^{\frac{1}{\hbar}p_1(u_6-u_3)}
\rho_{78}\, e^{\frac{1}{\hbar}p_1(u_8-u_7)}
\rho_{58}\, e^{\frac{1}{\hbar}p_2(u_8-u_5)}
\rho_{39}\, e^{\frac{1}{\hbar}p_2(u_9-u_3)}
\\
& \qquad \cdot 
\rho_{69}\, e^{\frac{1}{\hbar}p_4(u_9-u_6)}
\rho_{57}\, e^{\frac{1}{\hbar}p_4(u_7-u_5)}
\\ 
&= e^{\frac{1}{\hbar}p_1(u_2-u_4)} e^{\frac{1}{\hbar}p_1(u_3-u_6)}
e^{\frac{1}{\hbar}p_1(u_7-u_8)} e^{\frac{1}{\hbar}p_4(u_5-u_7)}
e^{\frac{1}{\hbar}p_4(u_6-u_9)} e^{\frac{1}{\hbar}p_2(u_3-u_6)} e^{\frac{1}{\hbar}p_2(u_7-u_8)}
\\
& \qquad \cdot 
\rho_{24} \rho_{36} \rho_{78} \rho_{58}\rho_{39} \rho_{69} \rho_{57}.
\end{split}
\end{align*}
By using the braid relation and the BCH formula similarly to 
the proof of Theorem \ref{thm:R-tetra}, we get \eqref{eq:K-monoad-id}.

Multiplying \eqref{K-qdilog-id} and \eqref{eq:K-monoad-id} side by side, we have an identity in 
$\hat{\mathcal{W}}^{\times}_{\mathbb{A}} \rtimes G$:
\begin{align*}
\begin{split}
&\Psi_q(Y_{10}) 
\cdot \Psi_{q^2}(Y_2) \Psi_q(q^{-1}Y_6 Y_{10}) \Psi_{q^2}(Y_2)^{-1} 
\cdot \Psi_{q}(q^{-2} Y_2 Y_6 Y_7 Y_{10}) \cdot \Psi_{q}(Y_{11})
\\ &\quad 
\cdot \Psi_{q^2}(Y_3) \Psi_q(q^{-1} Y_7 Y_{11}) \Psi_{q^2}(Y_3)^{-1}
\cdot \Psi_{q^2}(Y_2) \Psi_q(Y_6) \Psi_{q^2}(Y_2)^{-1}
\cdot \Psi_{q}(q^{-1} Y_2 Y_6 Y_7)
\\ &\quad
\cdot P_{457} P^K_{4689} P^K_{2379} P_{258} P_{178} P^K_{1356} P_{124}
\\
&= \Psi_q(\rY^{e_{11}}) 
\cdot \Psi_{q^2}(Y_3) \Psi_q(q^{-1} Y_7 Y_{11}) \Psi_{q^2}(Y_3)^{-1}
\cdot \Psi_{q^2}(Y_2) \Psi_q(Y_6) \Psi_{q^2}(Y_2)^{-1}
\\ &\quad 
\cdot \Psi_{q}(q^{-1} Y_2 Y_6 Y_7) \cdot \Psi_{q}(q^{-1} Y_{10} Y_{11}) 
\cdot \Psi_{q^2}(Y_3) \Psi_q(q^{-2} Y_7 Y_{10} Y_{11}) \Psi_{q^2}(Y_3)^{-1} 
\cdot \Psi_{q}(Y_{10})
\\ & \quad
\cdot P_{124} P^K_{1356} P_{178} P_{258} P^K_{2379} P^K_{4689} P_{457},
\end{split}
\end{align*}
where all the $Y_i$'s are understood to be $\phi(Y_i)$. 
From this we obtain \eqref{eq:RK-full-id} by moving $P_{ijk}$ and $P^K_{ijkl}$ to 
the left appropriately in the same manner as the proof of Theorem \ref{thm:R-tetra}. 
\end{proof}

\begin{remark}\label{re:ezk}
Consider the generalizations of 
$\pi^K_{ijkl}$ (\ref{eq:K-pi}), $P^K_{ijkl}$ (\ref{eq:K-Pop-general}) and 
$\mathcal{K}_{ijkl}$ (\ref{eq:K-ad}) including the two complex parameters $\beta$ and $\gamma$ 
as follows:
\begin{align}
\pi^K_{ijkl}:&
\begin{cases}
\;p_i \mapsto p_i+ (1-\beta) \lambda_{jl}, 
\quad 
p_j \mapsto p_l+ 2p_i+(1-\gamma) \lambda_{jl},
\\
\;p_k  \mapsto p_k- (1-\beta)\lambda_{jl},
\quad
p_l \mapsto p_j-2p_i- (1-\gamma) \lambda_{jl},
\\
\;u_i \mapsto u_i+u_j-u_l+ \beta\lambda_{jl},
\quad 
u_j  \mapsto u_l + (\gamma-2\beta)\lambda_{jl},
\\
\; u_k  \mapsto u_k+ (\beta-\gamma) \lambda_{jl},
\quad 
u_l \mapsto u_j+ \gamma \lambda_{jl},
\end{cases}
\\
P^K_{ijkl} &= \rho_{jl}\, e^{\frac{1}{\hbar}p_i(u_l-u_j)} 
e^{\frac{\lambda_{jl}}{2\hbar}(2(1-\beta)(u_k-u_i)+(1-\gamma)(u_l-u_j))}
e^{\frac{\lambda_{jl}}{2\hbar}(2\beta p_i+(\gamma-2\beta)p_j+2(\beta-\gamma)p_k+\gamma p_l)},
\\
\mathcal{K}_{ijkl} &= \Psi_{q^2}(e^{p_j+u_j+p_l-u_l-2p_k+ \lambda_{jl}}) 
\Psi_q(e^{p_i+u_i+p_k-u_k-p_j+\lambda_{ik}}) 
\Psi_{q^2}(e^{p_j+u_j+p_l-u_l-2p_k+\lambda_{jl}})^{-1} P^K_{ijkl}.
\end{align}   
Then Proposition \ref{prop:K-ad}, Lemma \ref{le:nK} and Theorem \ref{thm:K-reflection} remain valid.
In particular, the 3D reflection equations hold with the companion 
$\pi_{ijk}, P_{ijk}$ and $\mathcal{R}_{ijk}$ in Remark \ref{re:Rcd} including a complex parameter $\alpha$.
\end{remark}

%%%%%%%%%%%%%%%%%%%%%%%%%%%%%%%%%%%%%%%%%%%%%%%
\section{Representation with quantum dilogarithm}
%%%%%%%%%%%%%%%%%%%%%%%%%%%%%%%%%%%%%%%%%%%%%%%

We retain the notation $q=e^\hbar$ in (\ref{para}).
Consider an infinite dimensional representation of the $q$-commuting operators $e^{p_i}$ and $e^{u_i}$ 
on the vector space $V = \oplus_{a_i \in \Z} \C |a_i \rangle$:
\begin{align}\label{eq:rep-q}
e^{p_i} |a_i \rangle = |a_i-1 \rangle, \quad 
e^{u_i} |a_i \rangle = 
\begin{cases}
q^{a_i}|a_i \rangle & \text{ if $[p_i,u_i] = \hbar$ (i.e. $i=1,3$)},
\\
q^{2 a_i}|a_i \rangle & \text{ if $[p_i,u_i] = 2 \hbar$ (i.e. $i=2,4$)}.
\end{cases}
\end{align}
This representation will be referred to as the $u$-representation.
We write $|a_1,a_2,\ldots,a_m \rangle$ for a tensor product 
$|a_1 \rangle \otimes |a_2 \rangle \otimes \cdots \otimes |a_m \rangle \in V^{\otimes m}$, 
where $p_i$ and $u_i$ act on the $i$-th component.
We write $\langle a_i |$ for the dual basis, satisfying $\langle a_i ||a_i' \rangle = \delta_{a_i,a_i'}$.

Besides the $u$-representation, we will also consider another one, 
the $p$-representation, which is defined by replacing (\ref{eq:rep-q}) with 
\begin{equation}\label{prep}
e^{p_i} |c_i \rangle = 
\begin{cases}
q^{c_i}|c_i\rangle & \text{if  $[p_i,u_i] = \hbar$ \; (i.e. $i=1,3$)},
\\
q^{2c_i}|c_i\rangle & \text{if  $[p_i,u_i] = 2 \hbar$ \;(i.e. $i=2,4$)},
\end{cases}
\qquad 
e^{u_i} |c_i \rangle = |c_i+1 \rangle. 
\end{equation}
In the following we use the fact $1/(q;q)_n = 0$ for $n <0$,
which follows from \eqref{eq:qq_n}.

\subsection{Representation of $\mathcal{R}_{123}$}

We use the parameters $\lambda_i$ and $\kappa_i$ which are related as in (\ref{para}).
Let us consider $\mathcal{R}_{123}$ in \eqref{eq:R-ad}:
\begin{align}
\mathcal{R}_{123} = \Psi_q( \kappa_1 \kappa_3^{-1} e^{p_1+u_1+p_3-u_3-p_2}) \rho_{23}\, e^{\frac{1}{\hbar}p_1(u_3-u_2)} 
e^{\frac{\lambda_{23}}{\hbar}(u_3-u_1)}.
\end{align}

\begin{prop}
In the $u$-representation, matrix elements of $\mathcal{R}_{123}$ are given by
\begin{align}\label{eq:Rkappa-rep}
\begin{split}
  R_{b_1,b_2,b_3}^{a_1,a_2,a_3}(\lambda_1,\lambda_2,\lambda_3) 
  &:= \langle a_1,a_2,a_3 | \mathcal{R}_{123} |b_1,b_2,b_3 \rangle
  \\
  &= \frac{(-q)^{a_2-b_3}}{(q^2;q^2)_{a_2-b_3}} q^{(b_1-b_3+n_{13})(a_2-b_3-n_{23}) + n_{13}n_{23}}
\, \delta_{b_1+b_2}^{a_1+a_2} \,\delta_{b_2+b_3}^{a_2+a_3},
\end{split}
\end{align}
where $n_{ij} := \frac{\lambda_{ij}}{\hbar}$.
\end{prop}

\begin{proof}
We first calculate the action of $P_{123}= \rho_{23}\, e^{\frac{1}{\hbar}p_1(u_3-u_2)} e^{\frac{\lambda_{23}}{\hbar}(u_3-u_1)}$:
\begin{align*}
\rho_{23}\,& e^{\frac{1}{\hbar}p_1(u_3-u_2)} e^{\frac{\lambda_{23}}{\hbar}(u_3-u_1)} |a,b,c \rangle 
=
e^{\frac{\lambda_{23}}{\hbar}(c-a)} \rho_{23}\, e^{\frac{1}{\hbar}p_1(u_3-u_2)} |a,b,c \rangle
\\
&= q^{n_{23}(c-a)} \rho_{23} |a+b-c,b,c \rangle
= q^{n_{23}(c-a)} |a+b-c,c,b \rangle,
\end{align*}
where the linear version $u_i |m_i \rangle = \hbar m_i |m_i \rangle$ 
of \eqref{eq:rep-q} has been used.
From \eqref{eq:q-dilog-sum}, the action of $\Psi_q(\phi_\kappa(Y_4))$ is calculated as
\begin{align*}
\Psi_q(\kappa_1 \kappa_3^{-1} e^{p_1+u_1+p_3-u_3-p_2}) |a,b,c \rangle
&= \sum_{n = 0}^\infty \frac{(-q)^n}{(q^2;q^2)_n} \left(\frac{\kappa_1}{\kappa_3}\right)^n e^{n p_1} e^{n u_1} e^{-np_2} e^{np_3} e^{-n u_3} |a,b,c \rangle
\\
&= \sum_{n = 0}^\infty \frac{(-q)^n}{(q^2;q^2)_n} q^{\frac{\lambda_{13}}{\hbar}n} q^{(a-c)n} |a-n,b+n,c-n \rangle. 
\end{align*}
Combining them leads to \eqref{eq:Rkappa-rep} as follows: 
\begin{align*}
&\langle a_1,a_2,a_3 | \Psi_q(\kappa_1 \kappa_3^{-1} e^{p_1+u_1+p_3-u_3-p_2}) P_{123} | b_1,b_2,b_3 \rangle
\\
& \quad = \sum_{n \geq 0} \frac{(-q)^n}{(q^2;q^2)_n} 
q^{\frac{ \lambda_{23}}{\hbar}(b_3-b_1)}
q^{\frac{ \lambda_{13}}{\hbar}n} 
q^{(b_1-b_3)n} \,\delta_{b_1+b_2-b_3-n}^{a_1} \,\delta_{b_3+n}^{a_2} \,\delta_{b_2-n}^{a_3}
\\
& \quad = \frac{(-q)^{a_2-b_3}}{(q^2;q^2)_{a_2-b_3}} 
q^{\frac{ \lambda_{23}}{\hbar}(b_3-b_1)}
q^{\frac{\lambda_{13}}{\hbar}(a_2-b_3)} 
q^{(b_1-b_3)(a_2-b_3)} \delta^{a_1+a_2}_{b_1+b_2} \, \delta^{a_2+a_3}_{b_2+b_3}. 
\end{align*}
\end{proof}

The following claim is a corollary of Theorem \ref{thm:R-tetra}. 
We provide an independent and simple proof of it, which may be of separate interest.

\begin{thm}\label{th:teu}
The matrix elements (\ref{eq:Rkappa-rep}) satisfy the tetrahedron equation:
\begin{equation}
\begin{split}
&\sum_{b_1,\ldots,b_6 \in \Z}
R^{a_1,a_2,a_4}_{b_1,b_2,b_4}(\lambda_1, \lambda_2,\lambda_4) 
R^{b_1,a_3,a_5}_{c_1,b_3,b_5}(\lambda_1, \lambda_3,\lambda_5) 
R^{b_2,b_3,a_6}_{c_2,c_3,b_6}(\lambda_2, \lambda_3,\lambda_6)
R^{b_4,b_5,b_6}_{c_4,c_5,c_6}(\lambda_4, \lambda_5,\lambda_6)
\\
=& 
\sum_{b_1,\ldots,b_6  \in \Z}
R^{a_4,a_5,a_6}_{b_4,b_5,b_6} (\lambda_4, \lambda_5,\lambda_6)
R^{a_2,a_3,b_6}_{b_2,b_3,c_6} (\lambda_2, \lambda_3,\lambda_6)
R^{a_1,b_3,b_5}_{b_1,c_3,c_5} (\lambda_1, \lambda_3,\lambda_5) 
R^{b_1,b_2,b_4}_{c_1,c_2,c_4} (\lambda_1, \lambda_2,\lambda_4),
\end{split}
\end{equation}
where $a_i, c_i  \in \Z (i=1,\ldots, 6)$ are arbitrary. 
\end{thm}
\begin{proof}
Upon substituting (\ref{eq:Rkappa-rep}), both sides contain eight Kronecker deltas.
By inspecting them, one finds that the two sides are vanishing unless the following conditions hold:
$a_1 + a_2 + a_3 - c_1 - c_2 - c_3= -a_1 + a_4 + a_5 + c_1 - c_4 - c_5
=a_3 + a_5 + a_6 - c_3 - c_5 - c_6=0$.
In what follows we assume these conditions and understand that
$c_3, c_5, c_6$ have been eliminated by them.
Then the LHS becomes a single sum over $b_2$, where the other $b_i$'s are related as
$b_1=a_1+a_2-b_2,b_3=a_1+a_2+a_3-b_2-c_1,b_4=a_2+a_4-b_2,
b_5=-a_1-a_2+a_5+b_2+c_1,b_6=a_6-b_2+c_2$.
Similarly, the RHS is a single sum over $b_2$, where $b_1, b_3, b_4, b_5, b_6$ are related as
$b_1=-b_2+c_1+c_2,b_3=a_2+a_3-b_2,b_4=-b_2+c_2+c_4,
b_5=a_4+a_5+b_2-c_2-c_4,b_6=-a_4+a_6-b_2+c_2+c_4$.
Note that the summation variable $b_2$ may be freely shifted, as the property 
$1/(q^2;q^2)_n=0 \,(n<0)$ automatically selects the range yielding non-vanishing contributions. 
Thus we change the summation variable $b_2$ into $n$ by setting $b_2 =  n+a_4$ in the LHS
and $b_2=n+c_4$ in the RHS.
After removing a common factor independent of $n$  (a sign and a messy power of $q$)  from both sides,
one is left to show the concise equality:
\begin{align}\label{bd}
\frac{1}{(q^2)_{s+t}}\sum_{n \in \Z} \frac{(-1)^n q^{n(n+1+2s)}}
{(q^2)_{n}(q^2)_{t-n}(q^2)_{n+r}}
= 
\frac{1}{(q^2)_{r+t}}
\sum_{n \in \Z} \frac{(-1)^n q^{n(n+1+2r)}}
{(q^2)_{n}(q^2)_{t-n}(q^2)_{n+s}},
\end{align}
where $(q^2)_n=(q^2;q^2)_n$ for short, and 
$r = -a_1 + 2 a_4 + a_5 - a_6 + c_1 - c_2 - c_4, 
s = a_4 + a_5 - a_6 - c_2, 
t = a_1 + a_2 + a_3 - a_4 - a_5 - c_1$.
In order to prove this, it suffices to show the symmetry of 
$\mathrm{LHS}\times (q^2)_{s+t}(q^2)_{r+t}$
under the interchange of $r$ and $s$.
This is made manifest by applying the $q$-binomial expansion as follows:
\begin{align*}
&\sum_{n=0}^t
\frac{ (-1)^n q^{n(n+1+2s)}(q^{2n+2r+2};q^2)_{t-n}}
{(q^2)_{n}(q^2)_{t-n}}
= \sum_{\substack{n,m\ge 0\\n+m\le t}}
\frac{(-1)^{n+m}q^{n(n+1)+m(m+1)+2nm}}{(q^2)_{n}(q^2)_{m}(q^2)_{t-n-m}}
q^{2r m+2s n}.
\end{align*}
\end{proof}

\begin{remark}\label{re:dual}
The fact that the tetrahedron equation (\ref{eq:Rkappa-rep}) can be distilled into a duality between $r$ and $s$ as shown in (\ref{bd}) is intriguing. A similar feature is also present in the upcoming proof of Theorem \ref{th:teR}.
According to [SY22, TY14], mutation sequences give rise to the partition functions of $3$-dimensional supersymmetric gauge theories, and the tetrahedron
equation should correspond to a duality of gauge theories.
It will be very interesting to develop this correspondence in
connection with gauge-theoretic interpretations
of the Yang-Baxter equation \cite{S12, TY12, Y12, Yag15, BPZ15}. 
\end{remark}

\begin{prop}
Assume that $n_{23} := \frac{\lambda_{23}}{\hbar} \in \Z$. 
In the $p$-representation,  matrix elements of $\mathcal{R}_{123}$ are given by   
\begin{align}\label{eq:Rkappa-p-rep}
\begin{split}
  R_{d_1,d_2,d_3}^{c_1,c_2,c_3}(\lambda_1,\lambda_2,\lambda_3) 
  &:= \langle c_1,c_2,c_3 | \mathcal{R}_{123} |d_1,d_2,d_3 \rangle
  \\
  &= \frac{(-q)^{c_1-d_1+n_{23}}}{(q^2;q^2)_{c_1-d_1+n_{23}}} q^{(c_1-d_1+n_{23})(d_2-c_2+n_{13})}
   \delta_{d_1+d_3}^{c_2} \,\delta_{d_2}^{c_1+c_3}.
\end{split}
\end{align}
\end{prop}
This is derived similarly to the $u$-representation \eqref{eq:Rkappa-rep}, and we omit the proof.

\begin{remark}
The two formulae of $\mathcal{R}_{123}$ in the $u$-representation \eqref{eq:Rkappa-rep} and the $p$-representation \eqref{eq:Rkappa-p-rep} have similarity to the integral kernels of the modular double $R$ \eqref{Rmod} in the momentum representation \eqref{Rkerp} and the coordinate representation \eqref{Rker2} respectively. 
In comparing them we set $\xi = -\zeta = 1$ in the integral kernels. 
\end{remark}

\subsection{Representation of $\mathcal{K}_{1234}$}

In this subsection we consider $\mathcal{K}_{1234}$ \eqref{eq:K-ad} 
with $\forall \lambda_i = 0$ hence $\forall \kappa_i=1$ for simplicity:
\begin{equation}\label{ks}
\begin{split}
\mathcal{K}_{1234} &= \Psi_{q^2}(e^{p_2+u_2+p_4-u_4-2p_3}) \Psi_q(e^{p_1+u_1+p_3-u_3-p_2}) \Psi_{q^2}(e^{p_2+u_2+p_4-u_4-2p_3})^{-1} \rho_{24} \, e^{\frac{1}{\hbar}p_1(u_4-u_2)}.
\end{split}
\end{equation}

\begin{prop}
In the $u$-representation,  matrix elements of $\mathcal{K}_{1234}$ (\ref{ks}) are given by 
\begin{align}\label{eq:K-rep-u}
\begin{split}
&K^{a_1,a_2,a_3,a_4}_{b_1,b_2,b_3,b_4} :=
\langle a_1,a_2,a_3,a_4 |\mathcal{K}_{1234} | b_1,b_2,b_3,b_4 \rangle
\\
& \qquad = q^{(\ast)} (-1)^{a_2-b_4} \frac{(q^{-4(a_2+b_2-a_4-b_4)};q^4)_{b_2-a_4}}{(q^4;q^4)_{b_2-a_4}(q^2;q^2)_{a_2+b_2-a_4-b_4}}
\delta_{b_2+b_3+b_4}^{a_2+a_3+a_4}\,\delta_{b_1+2b_2+b_3}^{a_1+2a_2+a_3},
\end{split}
\end{align} 
where $(\ast) = (a_2+b_2-a_4-b_4)(b_1+2b_2-b_3-2 b_4 +1) + 2(b_2-a_4)(a_2-a_4+1)$.
\end{prop}

\begin{proof}
The actions of the constituent operators on $V^{\otimes 4}$ are computed as follows:
\begin{align}
\label{eq:K-actionP}
&\rho_{24} \, e^{\frac{1}{\hbar}p_1(u_4-u_2)} | a,b,c,d \rangle = |a+2b-2d,d,c,b \rangle,
\\
\begin{split}
&\Psi_q(e^{p_1+u_1+p_3-u_3-p_2})| a,b,c,d \rangle
= \sum_{n = 0}^\infty \frac{(-q)^n}{(q^2;q^2)_n} e^{n p_1} e^{n u_1} e^{-np_2} e^{np_3} e^{-n u_3} |a,b,c,d \rangle
\\
&\hspace{3cm}= \sum_{n = 0}^\infty \frac{(-q)^n q^{n(a-c)}}{(q^2;q^2)_n}|a-n,b+n,c-n,d \rangle, 
\end{split}
 \\
\begin{split}
&\Psi_{q^2}(e^{p_2+u_2+p_4-u_4-2p_3})| a,b,c,d \rangle
= \sum_{n = 0}^\infty \frac{(-q^2)^n}{(q^4;q^4)_n} e^{n p_2} e^{n u_2} e^{-2np_3} e^{np_4} e^{-n u_4} |a,b,c,d \rangle
\\
&\hspace{3cm}= \sum_{n = 0}^\infty \frac{(-q^2)^{n} q^{2n(b-d)}}{(q^4;q^4)_n}|a,b-n,c+2n,d-n \rangle.
\end{split} 
\\
\begin{split}
&\Psi_{q^2}(e^{p_2+u_2+p_4-u_4-2p_3})^{-1}| a,b,c,d \rangle
= \sum_{n = 0}^\infty \frac{q^{2n^2}}{(q^4;q^4)_n} e^{n p_2} e^{n u_2} e^{-2np_3} e^{np_4} e^{-n u_4} |a,b,c,d \rangle
\\
&\hspace{3.3cm}
= \sum_{n = 0}^\infty \frac{q^{2n^2+2n(b-d)}}{(q^4;q^4)_n} |a,b-n,c+2n,d-n \rangle.
\end{split}
\end{align}
By combining them, we get
\begin{align*}
&\langle a_1,a_2,a_3,a_4 |\mathcal{K}_{1234}| b_1,b_2,b_3,b_4 \rangle
\\
& \quad =
\sum_{n_1 \geq 0} \sum_{n_2 \geq 0} \sum_{n_3 \geq 0}
\frac{(-q)^{n_2}(-q^2)^{n_3} q^{(\#)}}{(q^4;q^4)_{n_1}(q^2;q^2)_{n_2}(q^4;q^4)_{n_3}} 
\\
& \qquad \qquad \times \delta^{a_1}_{b_1+2b_2-2b_4-n_2} \delta^{a_2}_{b_4-n_1+n_2-n_3}
\delta^{a_3}_{b_3+2n_1-n_2+2n_3} \delta^{a_4}_{b_2-n_1-n_3}, 
\end{align*}
where $(\#) = 2n_1^2 + 2n_1(b_4-b_2)+n_2(b_1+2b_2-b_3-2b_4-2n_1)+2n_3(b_4-b_2+n_2)$. This is reduced to
\begin{align*}
%\theta(a_2+b_2-a_4-b_4)\,& 
\frac{(-1)^{a_2-b_4} q^{(\ast)}}{(q^2;q^2)_{a_2+b_2-a_4-b_4}} \, \delta_{b_2+b_3+b_4}^{a_2+a_3+a_4}\,\delta_{b_1+b_2-b_4}^{a_1+a_2-a_4}
%\\ 
%& \qquad \cdot 
\sum_{n_1=0}^{b_2-a_4} \frac{q^{2n_1^2-4n_1(a_2+b_2-a_4-b_4)}(-q^2)^{-n_1}}{(q^4;q^4)_{n_1}(q^4;q^4)_{b_2-a_4-n_1}}.
\end{align*}
The sum here can be taken by means of the $q$-binomial formula, leading to 
$$ 
\frac{(q^{-4(a_2+b_2-a_4-b_4)};q^4)_{b_2-a_4}}{(q^4;q^4)_{b_2-a_4}},
$$
and \eqref{eq:K-rep-u} is obtained.
\end{proof}

By a similar computation,
the following proposition for the $p$-representation is proved. 

\begin{prop}
In the $p$-representation,  matrix elements of $\mathcal{K}_{1234}$ (\ref{ks}) are given by  
\begin{align}\label{eq:K-rep-p}
\begin{split}
&K^{c_1,c_2,c_3,c_4}_{d_1,d_2,d_3,d_4} :=
\langle c_1,c_2,c_3,c_4 |\mathcal{K}_{1234}| d_1,d_2,d_3,d_4 \rangle
\\
&\qquad = q^{(\ast \ast)} (-1)^{c_1+c_2-d_4} \frac{(q^{-4(c_1-d_1)};q^4)_{c_2-d_1-d_4}}{(q^4;q^4)_{c_2-d_1-d_4}(q^2;q^2)_{c_1-d_1}}
\delta_{d_1+d_3}^{c_1+c_3}\,\delta_{d_2+d_4}^{c_2+c_4},
\end{split}
\end{align} 
where $(\ast \ast) = (c_1-d_1)(-d_1+d_3-2d_4+1) + 2(c_2-d_1-d_4)(d_2+d_4-c_3+1)$.
\end{prop}

%%%%%%%%%%%%%%%%%%%%%%%%%%%%%%%%%%%%%%%%%%%%%%%%%%%%%%%%%%%%%
\section{Representation with noncompact quantum dilogarithm}
%%%%%%%%%%%%%%%%%%%%%%%%%%%%%%%%%%%%%%%%%%%%%%%%%%%%%%%%%%%%%

In this section we study the representation of $\mathcal{R}_{123}$ \eqref{eq:R-ad} 
and $\mathcal{K}_{1234}$ \eqref{eq:K-ad} involving noncompact quantum dilogarithm function.
We set $i = \sqrt{-1}$ and use the parameters $\bb$ and $\ell_i$ related to 
$q=e^\hbar$ and $\kappa_j = e^{\lambda_j}$ in (\ref{para}) as follows:
\begin{align}\label{para2}
\hbar = i\pi \bb^2, \quad 
\lambda_i = \pi \bb \ell_i,\quad \ell_{ij} = \ell_i-\ell_j.
\end{align} 

\subsection{Modular double $R$} 

For pairs of noncommuting variables $(p_k,u_k)$ for 
$k=1,\ldots,6$ satisfying $[p_k,u_j] = \hbar \delta_{kj}$,
define $\hat{x}_k$ and $\hat{p}_k$ by 
\begin{align}\label{pub}
p_k =  \pi \bb \hat{x}_k,\quad
u_k = \pi \bb \hat{p}_k.
\end{align}
They satisfy the canonical commutation relation
\begin{align*}
[\hat{x}_j, \hat{p}_k] = \frac{i}{\pi} \delta_{j,k}.
\end{align*}
Let $\varphi(z)$ be the noncompact quantum dilogarithm function \eqref{eq:noncomp-dilog}. In view of 
\begin{align}\label{recc}
\frac{\Psi_q(e^{2\pi \bb(z+ i\bb/2)})}
{\Psi_q(e^{2\pi \bb(z- i\bb/2)})}
= 1+e^{2\pi \bb z}
= \frac{\varphi(z-i \bb/2)}{\varphi(z+i \bb/2)},
\end{align}
we introduce the modular double $R$, $\hat{\mathcal{R}}_{123} = \hat{\mathcal{R}}(\ell_1,\ell_2,\ell_3)_{123}$, by formally 
replacing $\Psi_q(e^{\pi \bb\hat{z}})$ by 
$\varphi(\tfrac{1}{2}\hat{z})^{-1}$  in \eqref{eq:R-ad} as
\begin{align}\label{Rmod}
\hat{\mathcal{R}}_{123} = \varphi \bigl(\tfrac{1}{2}(\hat{x}_1+\hat{p}_1+\hat{x}_3-\hat{p}_3-\hat{x}_2+\ell_{13})\bigr)^{-1}\rho_{23}
e^{\pi i \hat{x}_1(\hat{p}_2-\hat{p}_3)}
e^{\pi i \ell_{23}(\hat{p}_1-\hat{p}_3)}.
\end{align}

\subsection{Coordinate representation of $\hat{\mathcal{R}}_{123}$}\label{ss:crR}
Let us regard the modular double $R$ (\ref{Rmod}) as an operator 
on the Hilbert space $L^2(\R^3)$ of functions with variables $x_1,x_2,x_3$,  
where $\hat{x}_k$ acts as the multiplication by $x_k$
and $\hat{p}_k$ as a differential operator 
\begin{align*}
\hat{p}_k = -\frac{i}{\pi} \frac{\partial}{\partial x_k}.
\end{align*}
Formally we write the space as $|x_1,x_2,x_3 \rangle = |x_1 \rangle \otimes |x_2 \rangle \otimes |x_3 \rangle$ where $\hat{x}_k |x_k \rangle = x_k |x_k \rangle$ and 
$\langle x_k | | x_k' \rangle = \delta(x_k-x_k')$.
The integral kernel of $\hat{\mathcal{R}}_{123}$ is given by 
\begin{align}
\begin{split}
&\langle x_1,x_2, x_3| \hat{\mathcal{R}}_{123} |x'_1, x'_2, x'_3\rangle 
\\
&= 
\langle x_1,x_2, x_3| 
\varphi \bigl(\tfrac{1}{2}(\hat{x}_1+\hat{p}_1+\hat{x}_3-\hat{p}_3-\hat{x}_2+\ell_{13})\bigr)^{-1}
\rho_{23}
e^{\pi i \hat{x}_1(\hat{p}_2-\hat{p}_3)}
|x'_1-\ell_{23}, x'_2, x'_3+\ell_{23}\rangle 
\\
&= 
\langle x_1,x_2, x_3| 
\varphi \bigl(\tfrac{1}{2}(\hat{x}_1+\hat{p}_1+\hat{x}_3-\hat{p}_3-\hat{x}_2+\ell_{13})\bigr)^{-1}
\rho_{23}
|x'_1-\ell_{23}, x'_2-x'_1+\ell_{23}, x'_3+x'_1\rangle 
\\
&= 
\langle x_1,x_2, x_3| 
\varphi \bigl(\tfrac{1}{2}(\hat{x}_1+\hat{p}_1+\hat{x}_3-\hat{p}_3-\hat{x}_2+\ell_{13})\bigr)^{-1}
|x'_1-\ell_{23}, x'_3+x'_1, x'_2-x'_1+\ell_{23}\rangle 
\\
&= 
\delta(x_2-x'_1-x'_3)
\langle x_1,x_3| 
\varphi \bigl(\tfrac{1}{2}(\hat{x}_1+\hat{p}_1+\hat{x}_3-\hat{p}_3+\ell_{13}
-x'_1-x'_3)\bigr)^{-1}
|x'_1-\ell_{23}, x'_2-x'_1+\ell_{23}\rangle.
\end{split}
\label{tochu1}
\end{align}
To evaluate (\ref{tochu1}), for a canonical pair $(\hat{x},\hat{p})$ we introduce the diagonalizing kets 
\begin{align*}
\begin{split}
&(\hat{p} + \hat{x}) |\lambda \rangle\!\rangle = \lambda |\lambda\rangle\!\rangle,
\quad
|\lambda \rangle\!\rangle= \int f(\lambda, x)|x\rangle dx, 
\quad f(\lambda, x) =  e^{\pi i(\lambda x - x^2/2)}/\sqrt{2},
\\
&(\hat{p} - \hat{x}) |\mu \rangle\!\rangle = \mu |\mu\rangle\!\rangle,
\quad
|\mu \rangle\!\rangle= \int g(\mu,  x)|x\rangle dx, 
\quad g(\mu, x) = e^{\pi i(\mu x + x^2/2)}/\sqrt{2},
\end{split}
\end{align*}
where $\lambda, \mu \in \R$.
Bra vectors  are similarly defined by
$\langle\!\langle \lambda | = \int dx f(\lambda,x)^\ast \langle x |$
and 
$\langle\!\langle \mu | = \int dx g(\mu,x)^\ast \langle x |$.
Note that 
$\langle\!\langle \lambda | |\lambda' \rangle\!\rangle= \delta(\lambda-\lambda')$
and 
$\langle\!\langle \mu | |\mu' \rangle\!\rangle= \delta(\mu-\mu')$.
Thus inserting 
$1= \int d\lambda |\lambda\rangle\!\rangle \langle\!\langle \lambda |
= \int d\mu |\mu\rangle\!\rangle \langle\!\langle \mu |$, we have 
\begin{align*}
\begin{split}
&\langle x_1,x_3| 
\varphi \bigl((\tfrac{1}{2}(\hat{x}_1+\hat{p}_1+\hat{x}_3-\hat{p}_3+s)\bigr)^{-1}
|z_1, z_3\rangle
\\
&= \int d\lambda d\mu \varphi\bigl(\tfrac{1}{2}(\lambda - \mu+s)\bigr)^{-1}
f(\lambda,z_1)^\ast f(\lambda, x_1)g(\mu,z_3)^\ast g(\mu,x_3)
\\
&=\tfrac{1}{4}\int d\nu \varphi(\nu+\tfrac{s}{2})^{-1}e^{2\pi i \nu (x_1-z_1)}
\int d\mu e^{\pi i\mu (x_1-z_1+x_3-z_3)+\tfrac{\pi i}{2}(-x_1^2+z_1^2+x_3^2-z_3^2)}
\\
&= \tfrac{1}{2K}\delta(x_1+x_3-z_1-z_3)  \varphi(x_1-z_1+i\eta)
e^{-2\pi i(x_1-z_1)(\tfrac{s}{2}+i\eta)+\tfrac{\pi i}{2}(-x_1^2+z_1^2+x_3^2-z_3^2)},
\end{split}
\end{align*}
where the last step is due to (\ref{ram1}).
Applying this result to (\ref{tochu1}), we obtain the following.

\begin{prop}
In the coordinate representation, the integral kernel of $\hat{\mathcal{R}}_{123}$ is 
\begin{equation}\label{Rker}
\begin{split}
&\langle x_1,x_2, x_3| \hat{\mathcal{R}}_{123} |x'_1, x'_2, x'_3\rangle 
\\
&= \varrho_x \delta(x_2-x'_1-x'_3)\delta(x'_2-x_1-x_3) 
\varphi(x_1-x'_1+\ell_{23}+i\eta) e^{\pi i (x_1-x'_1+\ell_{23})(x_2-x'_2-\ell_{13}-2i\eta)},
\end{split}
\end{equation}
where $\varrho_x = (2K)^{-1}$ with $K$ defined after (\ref{rama2}).
\end{prop}

The following claim is a corollary of Theorem \ref{thm:R-tetra} and \cite[Theorem 4.5]{KN11}.
%The following claim is a corollary of the general construction \eqref{eq:R-full-id} in \S \ref{sec:tetra}.
We include it as an independent  confirmation in the coordinate representation of the 
modular double $R$.
\begin{thm}\label{th:teR}
The modular double $R$ 
in the coordinate representation satisfies the tetrahedron equation:
\begin{equation}\label{TElam}
\begin{split}
&\hat{\mathcal{R}}(\ell_1, \ell_2, \ell_4)_{124}
\hat{\mathcal{R}}(\ell_1, \ell_3, \ell_5)_{135}
\hat{\mathcal{R}}(\ell_2, \ell_3, \ell_6)_{236}
\hat{\mathcal{R}}(\ell_4, \ell_5, \ell_6)_{456} 
\\
&= 
\hat{\mathcal{R}}(\ell_4, \ell_5, \ell_6)_{456}
\hat{\mathcal{R}}(\ell_2, \ell_3, \ell_6)_{236}
\hat{\mathcal{R}}(\ell_1, \ell_3, \ell_5)_{135} 
\hat{\mathcal{R}}(\ell_1, \ell_2, \ell_4)_{124}.
\end{split}
\end{equation}
\end{thm}
\begin{proof}
We show that the kernel corresponding to the two sides of the tetrahedron equation coincide.
For the LHS, it is given by 
\begin{align}
\langle  x_1,\ldots, x_6 | 
\hat{\mathcal{R}}_{124}\hat{\mathcal{R}}_{135}\hat{\mathcal{R}}_{236}\hat{\mathcal{R}}_{456} 
|z_1,\ldots, z_6\rangle 
= 
\int R^{x_1 x_2 x_4}_{y_1 y_2 y_4} R^{y_1 x_3 x_5}_{z_1 y_3 y_5}
R^{y_2 y_3 x_6}_{z_2 z_3 y_6} R^{y_4 y_5 y_6}_{z_4 z_5 z_6} dy_1 \cdots dy_6.
\end{align}
Here the dependence on $\ell_i$'s is suppressed 
and (\ref{Rker}) is denoted by $R^{x_1x_2x_3}_{x'_1x'_2x'_3}$ to save the space.
Substitution of $(\ref{Rker})/\varrho_x$ reduces the six-fold integral into 
the {\em single} one due to the delta functions as
\begin{equation*}
\begin{split}
&\delta(x_1+x_4+x_6-z_3)\delta(x_2+x_5-z_2-z_5)\delta(x_3-z_1-z_4-z_6)
\varphi(-x_6-z_2+z_3+\ell_{36}+i\eta) e^{2\pi i \alpha_0}
\\
&\times
\int dy_1 \varphi(x_1-y_1+\ell_{24}+i\eta)
\varphi(y_1-z_1+\ell_{35}+i\eta) \varphi(x_2-y_1-z_4+\ell_{56}+i\eta)
e^{-\pi i y_1^2-2\pi i y_1\alpha_1},
\end{split}
\end{equation*}
where $\alpha_1 = x_6-z_1+z_2-z_3-\ell_{56}-i \eta$ and 
the explicit form of the $y_1$-independent power $\alpha_0$ is omitted.
By inverting $\varphi(y_1-z_1+ \ell_{35}+i\eta)$ by (\ref{pinv}) and setting $y_1=-t$,  
it is rewritten as
\begin{equation}\label{tint1}
\begin{split}
&\delta(x_1+x_4+x_6-z_3)\delta(x_2+x_5-z_2-z_5)\delta(x_3-z_1-z_4-z_6)
\varphi(-z_2+z_3-x_6+\ell_{36}+i\eta) e^{2\pi i \beta_1}
\\
&\times
\int dt \frac{\varphi(t+x_2-z_4+\ell_{56}+i\eta)\varphi(t+x_1+\ell_{24}+i\eta)}
{\varphi(t+z_1-\ell_{35}-i\eta)}
e^{2\pi it(z_2-z_3+x_6-\ell_{36}-2i\eta)},
\end{split}
\end{equation}
where the explicit form of $\beta_1$ is omitted.
A similar calculation of the kernel for the RHS of \eqref{eq:R-full-id}  gives
\begin{equation}\label{tint2}
\begin{split}
&\delta(x_1+x_4+x_6-z_3)\delta(x_2+x_5-z_2-z_5)\delta(x_3-z_1-z_4-z_6)
\varphi(x_2-x_3+z_6+\ell_{36}+i\eta) e^{2\pi i \beta_2}
\\
&\times
\int dt \frac{\varphi(t-z_2+x_4+\ell_{56}+i\eta)\varphi(t-z_1+\ell_{24}+i\eta)}
{\varphi(t-x_1-\ell_{35}-i\eta)}
e^{-2\pi it(x_2-x_3+z_6+\ell_{36}+2i\eta)}
\end{split}
\end{equation}
for some $\beta_2$.
We remark a duality that (\ref{tint1}) and (\ref{tint2}) are exactly transformed to each other 
by the exchange $x_i \leftrightarrow -z_i$ for $i=1,\ldots, 6$.
In particular, the delta functions are the same.
By applying (\ref{heine1}), they are transformed into the identical form
\begin{equation*}
\begin{split}
&\delta(x_1+x_4+x_6-z_3)\delta(x_2+x_5-z_2-z_5)\delta(x_3-z_1-z_4-z_6)
e^{2\pi i \nu}
\\
&\times
\int dt \frac{\varphi(-z_2+z_3-x_6+\ell_{36}+i\eta)
\varphi(x_2-x_3+z_6+\ell_{36}+i\eta)
\varphi(t+i\eta) e^{2\pi i t(x_1-z_1+\ell_{24}+\ell_{35})}}
{\varphi(t-z_2+z_3-x_6+\ell_{36}+i\eta)
\varphi(t+x_2-x_3+z_6+\ell_{36}+i\eta)}
\end{split}
\end{equation*}
for some $\nu$.
\end{proof}

\begin{remark}
The same argument as in the proof of Theorem \ref{th:teR} shows that 
the tetrahedron equation (\ref{TElam}) is actually valid for a slightly more general kernel
\begin{equation}\label{Rker2}
\varrho_x \delta(x_2-x'_1-x'_3)\delta(x'_2-x_1-x_3) 
\varphi(x_1-x'_1+\tfrac{1}{2}(\xi-\zeta)\ell_{23}+i\eta) e^{\pi i (x_1-x'_1+
\xi \ell_{23})(x_2-x'_2+ \zeta\ell_{13}-2i\eta)}
\end{equation}
for any complex numbers $\xi$ and $\zeta$, 
which essentially includes (\ref{Rker}) as the special case 
$\xi+\zeta=0$.
\end{remark}

\subsection{Momentum representation of $\hat{\mathcal{R}}_{123}$}
One may also regard the modular double $R$ as an operator 
on the Hilbert space $L^2(\R^3)$ of functions with variables $p_1,p_2,p_3$,  
where $\hat{p}_k$ acts as the multiplication by $p_k$ and  
\begin{align*}
\hat{x}_k = \frac{i}{\pi} \frac{\partial}{\partial p_k}.
\end{align*}
We write the space as $|p_1,p_2,p_3 \rangle$ in the same manner as the coordinate representation. 
The integral kernel for $\hat{\mathcal{R}}_{123}$ in the momentum representation is just the Fourier transformation 
of (\ref{Rker})  which is a six-fold integral:
\begin{align*}
\begin{split}
&\langle p_1,p_2, p_3| \hat{\mathcal{R}}_{123} |p'_1, p'_2, p'_3\rangle 
\\ 
&\quad = 
\int dx_1\cdots dx'_3 e^{\pi i(p'_1x'_1+p'_2x'_2+p'_3x'_3
-p_1x_1-p_2x_2-p_3x_3)}
\langle x_1,x_2, x_3| \hat{\mathcal{R}}_{123} |x'_1, x'_2, x'_3\rangle.
\end{split}
\end{align*}
A straightforward calculation with the kernel (\ref{Rker2}) 
leads to the following proposition.

\begin{prop}
In the momentum representation, the integral kernel of $\hat{\mathcal{R}}_{123}$ is obtained as  
\begin{equation}\label{Rkerp}
\begin{split}
&\langle p_1,p_2, p_3| \hat{\mathcal{R}}_{123} |p'_1, p'_2, p'_3\rangle 
\\
&= \varrho_p \delta(p_1+p_2-p'_1-p'_2)\delta(p_2+p_3-p'_2-p'_3)
\varphi(p_2-p'_3-\tfrac{1}{2}(\xi+\zeta)\ell_{23}+i \eta)
\\
& \hspace{7cm} \cdot 
e^{\pi i(p_2-p'_3-\xi\ell_{23})(p_3-p_1+ \zeta\ell_{13}-2i\eta)},
\end{split}
\end{equation}
where $\varrho_p = 8\varrho_x e^{\pi i\xi\ell_{23}(\zeta\ell_{13}-2i\eta)}$.
\end{prop} 

The validity of the tetrahedron equation for the above kernel 
is a corollary of Theorem \ref{th:teR}.

\subsection{Modular double $K$}

We keep the parameterization $q=e^\hbar$ with  $\hbar = i \pi \bb^2$. 
See (\ref{para2}). For simplicity we set $\kappa_i=1$ for all $i =1,2,3,4$.
Let us temporarily write $\eta$ in (\ref{qq}) as $\eta_\bb$ and $\varphi(z)$ as $\varphi_\bb(z)$ to signify the 
dependence on the parameter $\bb$.
Here we will also use 
\begin{align*}
\tilde{\eta} = \sqrt{2} \eta_{\sqrt{2}\bb} = \frac{2\bb+\bb^{-1}}{2},
\quad
{\tilde \varphi}(z) = \varphi_{\sqrt{2}\bb}(\frac{z}{\sqrt{2}}).
\end{align*}
We employ the same rescaling as (\ref{pub}):
\begin{align*}
p_k =  \pi \bb \hat{x}_k,\quad
u_k = \pi \bb \hat{p}_k  \quad (k=1,\ldots, 4).
\end{align*}
From (\ref{pu2}), they satisfy the canonical commutation relation
\begin{align*}
[\hat{x}_j, \hat{p}_k] = \begin{cases}
\frac{i}{\pi} \delta_{j,k} & j,k=1,3,
\\
\frac{2i}{\pi} \delta_{j,k} & j,k=2,4.
\end{cases}
\end{align*}
In view of (\ref{recc}) with $\bb\rightarrow \sqrt{2}\bb$ and $z\rightarrow z/\sqrt{2}$,
we introduce the modular double $K$ by formally making the replacement
$\Psi_q(e^{\pi \bb\hat{z}}) \rightarrow \varphi(\tfrac{1}{2}\hat{z})^{-1}$
and 
$\Psi_{q^2}(e^{\pi \bb\hat{z}}) \rightarrow {\tilde \varphi}(\tfrac{1}{2}\hat{z})^{-1}$
in $\mathcal{K}_{1234}$ in (\ref{ks}):
\begin{equation}\label{modK}
\begin{split}
\hat{\mathcal K}_{1234} &= 
\tilde{\varphi}\bigl(\tfrac{1}{2}(\hat{x}_2+\hat{p}_2+\hat{x}_4-\hat{p}_4-2\hat{x}_3)\bigr)^{-1}
\varphi\bigl(\tfrac{1}{2}(\hat{x}_1+\hat{p}_1+\hat{x}_3-\hat{p}_3-\hat{x}_2)\bigr)^{-1}
\\
&\qquad \times \tilde{\varphi}\bigl(\tfrac{1}{2}(\hat{x}_2+\hat{p}_2+\hat{x}_4-\hat{p}_4-2\hat{x}_3)\bigr)
\rho_{24}e^{i\pi \hat{x}_1(\hat{p}_2-\hat{p}_4)}.
\end{split}
\end{equation}

\subsection{Coordinate representation of $\hat{\mathcal K}_{1234}$}
We regard $\hat{\mathcal K}_{1234}$ as an operator in the Hilbert space $L^2(\R^4)$ of functions of 
variables $x_1, \ldots, x_4$, where $\hat{x}_k$ acts as the multiplication by $x_k$ and 
\begin{align*}
\hat{p}_k = \begin{cases} 
-\frac{i}{\pi}\frac{\partial}{\partial x_k} & k=1,3,
\\
-\frac{2i}{\pi}\frac{\partial}{\partial x_k} & k=2,4.
\end{cases}
\end{align*}
By a calculation similar to \S \ref{ss:crR}, we obtain the following result.

\begin{prop}
The integral kernel of $\hat{\mathcal K}_{1234}$ in the coordinate representation is given as 
\begin{align}
&\langle x_1,x_2,x_3, x_4|\hat{\mathcal K}_{1234}|x'_1,x'_2,x'_3,x'_4\rangle
\nonumber
\\
& \quad=\varrho'_x \delta(x_1+x_3-x'_1-x'_3)\delta(x_2+x_4-x'_2-x'_4)
\nonumber \\
& \qquad \quad \times \frac{e^{i \pi \alpha}\varphi(x_1-x'_1+i\eta){\tilde \varphi}(-2x_1+2x'_1-i{\tilde \eta})}
{{\tilde \varphi}(-2x_1+x_2-x'_4-i{\tilde \eta}){\tilde \varphi}(x_4+2x'_1-x'_2-i{\tilde \eta})},
\\
&\alpha =(x_2+x_3+x_4)x'_1-x_1x'_3+\tfrac{1}{2}(x_4x'_4-x_2x'_2)+(x_1-x'_1)(2x'_1+x'_4-2i\eta)
\nonumber \\
&\qquad +\tfrac{1}{2}(x_2-2x'_1-x'_4)(x_2+2x_3-2x'_1-x'_4).
\nonumber
\end{align}
Here, $\varrho'_x$ is a factor that depends on $\bb$ and is independent of $x_k$ and $x'_k$.
\end{prop}

The power $\alpha$ may be simplified further by using the conditions implied by the delta functions.

\subsection{Momentum representation of $\hat{\mathcal K}_{1234}$}
We regard $\hat{\mathcal K}_{1234}$ as an operator in the Hilbert space $L^2(\R^4)$ of functions of 
variables $p_1, \ldots, p_4$, where $\hat{p}_k$ acts as the multiplication by $p_k$ and 
\begin{align*}
\hat{x}_k = \begin{cases} 
\frac{i}{\pi}\frac{\partial}{\partial p_k} & k=1,3,
\\
\frac{2i}{\pi}\frac{\partial}{\partial p_k} & k=2,4.
\end{cases}
\end{align*}

\begin{prop}
The integral kernel of $\hat{\mathcal K}_{1234}$ in the momentum representation  is given as follows:
\begin{align}
&\langle p_1,p_2,p_3, p_4|\hat{\mathcal K}_{1234}|p'_1,p'_2,p'_3,p'_4\rangle
\nonumber \\
& \quad = \varrho'_p
\delta(p_1+p_2+p_3-p'_1-p'_2-p'_3)\delta(p_2+2p_3+p_4-p'_2-2p'_3-p'_4)
\nonumber \\
&\qquad \times
\frac{e^{i \pi \beta}\varphi(p_2+p_3-p'_3-p'_4+i\eta){\tilde
\varphi}(-p_2+p_4-p'_2+p'_4-i{\tilde \eta})}
{{\tilde \varphi}(p_4-p'_2-i{\tilde \eta}){\tilde
\varphi}(p'_4-p_2-i{\tilde \eta})},
\\
&\beta =
\tfrac{1}{4}(p_2^2+p_4^2-5(p'_2)^2-5(p'_4)^2)-p_1^2+p'_1(3p_1-2p'_1)
+p'_2(3p_1-p_4-4p'_1)
\nonumber \\
&\qquad \quad +p'_3(p_2+p_3-p'_3)+p'_4(p_1+3p_2+3p_3-4p'_3)+
2(p_1-p'_1-p'_2+p'_4)i\eta.
\nonumber
\end{align}
Here, $\varrho'_p$ is a factor that depends on $\bb$ and is independent of
$p_k$ and $p'_k$.
\end{prop}

The power $\beta$ may be simplified further by using the conditions implied
by the delta functions.

\appendix

\section{Noncompact quantum dilogarithm}\label{s:ncq}

In this section we summarize the definition 
and basics of the noncompact dilogarithm function. Set the parameters as
\begin{align}\label{qq}
q= e^{i \pi \bb^2}, \quad \bar{q} = e^{-i \pi \bb^{-2}},
\quad \eta = \frac{\bb+\bb^{-1}}{2}.
\end{align}
The noncompact dilogarithm function is defined as 
\begin{align}\label{eq:noncomp-dilog}
\varphi(z) = \exp\left(
\frac{1}{4}\int \frac{e^{-2izw}}{\sinh(w \bb)\sinh(w/\bb)}\frac{dw}{w}\right)
= \frac{(e^{2\pi(z+i\eta)\bb};q^2)_\infty}{(e^{2\pi(z-i\eta)\bb^{-1}};\bar{q}^2)_\infty},
\end{align}
where the infinite product formula is valid in the so-called strong coupling regime $0<\eta < 1$.
It satisfies the relations:
\begin{align}
\varphi(z)\varphi(-z) = e^{i \pi z^2-i \pi (1-2\eta^2)/6},
\label{pinv}\\
\frac{\varphi(z-i \bb^{\pm 1}/2)}{\varphi(z+i \bb^{\pm 1}/2)}
= 1+e^{2\pi z \bb^{\pm1}}.
\label{prec}
\end{align}
The following is known as a modular double analogue of the Ramanujan ${}_1\Psi_1$-sum:
\begin{align}
\int dt \frac{\varphi(t+u)}{\varphi(t+v)} e^{2\pi i wx}
&= \frac{\varphi(u-v-i\eta)\varphi(w+i\eta)}{K\varphi(u-v+w-i\eta)}
e^{-2\pi i w(v+i\eta)}
\label{rama1}\\
&= \frac{K\varphi(v-u-w+i\eta)}{\varphi(v-u+i\eta)\varphi(-w-i\eta)}
e^{-2\pi i w(u-i\eta)},
\label{rama2}
\end{align}
where $K = e^{-i\pi(4\eta^2+1)/12}$.
From $\varphi(u)\vert_{u \rightarrow -\infty} \rightarrow 1$,
their limit $u,v \rightarrow -\infty$ reduces to
\begin{align}
\int dt \frac{e^{2\pi i w t}}{\varphi(t+v)} &= \frac{\varphi(w+i\eta)}{K}e^{-2\pi iw(v+i\eta)},
\label{ram1}
\\
\int dt \varphi(t+u)e^{2\pi i w t}
&= \frac{K}{\varphi(-w-i\eta)} e^{-2\pi i w(u-i\eta)}.
\label{ram2}
\end{align}

The following is a  modular double analogue of the simplest Heine transformation:
\begin{lem}
\begin{align}
\begin{split}
&\int dt \frac{\varphi(t+a)\varphi(t+b)e^{2\pi i td}}{\varphi(t+c)} 
\\
& \quad = 
e^{-2\pi i(b-i\eta)d}\varphi(a-c-i\eta) 
\int dt \frac{\varphi(t+i\eta)e^{2\pi i t(b-c-2i\eta)}}{\varphi(t-d-i\eta)\varphi(t+a-c-i\eta)}.
\end{split}
\label{heine1}
\end{align}
\end{lem}
\begin{proof}
In the LHS,  rewrite 
$\varphi(t+b)$ as an integral using (\ref{ram1}) with $w=t+b-i\eta$ and 
$v= -d-i\eta$. The result reads
\begin{align*}
K e^{-2\pi i(b-i\eta)d}
\int dt \frac{\varphi(t+a)}{\varphi(t+c)} 
\int dz \frac{e^{2\pi iz(t+b-i\eta)}}{\varphi(z-d-i\eta)}.
\end{align*}
Taking the integral over $t$ by means of (\ref{rama1}), we get (\ref{heine1}).
\end{proof}

%%%%%%%%%%%%%%%%%%%%%%%%%%%


\begin{thebibliography}{xx}
\def\cprime{$'$}
%%%%%%%%%%%%%%%%%%%%%%%%%%%

\bibitem[B83]{B83}
R.~ J.~ Baxter,
{\em On Zamolodchikov's solution of the tetrahedron equations}, 
Commun. Math. Phys. {88} (1983) 185--205.

\bibitem[BB92]{BB92}
V.~V.~Bazhanov, R.~J.~Baxter,
{\em New solvable lattice models in three-dimensions},
J. Stat. Phys. {69}  (1992) 453--585.

\bibitem[BS06]{BS06}
V.~V.~Bazhanov, S.~M.~Sergeev,
{\em Zamolodchikov's tetrahedron equation and hidden structure of quantum groups},
J. Phys. A: Math. Theor. 39 (2006) 3295--3310.

\bibitem[BMS08]{BMS08}
V.~V.~Bazhanov, V.~V.~Mangazeev, S.~M.~Sergeev,
{\em Quantum geometry of 3-dimensional lattices and tetrahedron equation},
16th Int. Congr. of Mathematical Physics ed. P.~Exner (Singapore, World
Scientific, 2010) 23--44.

\bibitem[BFZ96]{BFZ96}
A. Berenstein, S. Fomin and A. Zelevinsky,
{\em Parametrizations of canonical bases and totally positive matrices},
Adv. in Math. 122 (1996), no.1, 49--149.

\bibitem[BPZ15]{BPZ15} 
F.~Benini, D.~S.~Park and P.~Zhao, 
{\em Cluster algebras from dualities of $2$d $N = (2, 2)$ quiver gauge theories}, Commun. Math. Phys. 340(1) (2015) 47--104.
 
\bibitem[BV15]{BV15}
A. Bytsko and A. Volkov,
{\em Tetrahedron equation, Weyl group, and quantum dilogarithm},
Lett. Math. Phys. 105 (2015) 45--61.

\bibitem[FG06a]{FG03} 
V. V. Fock and A. B. Goncharov, 
{\em Moduli spaces of local systems and higher Teichm\"uller theory},
Publ. Math. Inst. Hautes \'Etudes Sci. No. 103 (2006) 1--211.

\bibitem[FG06b]{FG06} 
V. V. Fock and A. B. Goncharov, 
 {\em Cluster $\mathcal{X}$-varieties, amalgamation and Poisson-Lie groups},
Algebraic geometry and number theory, volume 253 of Progr. Math.
(Birkh\"auser Boston, Boston, MA, 2006) 27--68.

\bibitem[FG09a]{FG09}
V. V. Fock and A. B. Goncharov,
{\em Cluster ensembles, quantization and the dilogarithm}, 
Ann. Sci. \'Ec. Norm. Sup\'er. (4) 42 (2009) 865--930.

\bibitem[FG09b]{FG09b}
V.~V. Fock and A.~B. Goncharov, 
{\em The quantum dilogarithm and representations of quantum cluster varieties.}
 Invent. Math. 175 (2009) no. 2, 223--286.

\bibitem[FZ07]{FZ07}
S.~Fomin and A.~Zelevinsky,  
{\em Cluster algebras. IV. Coefficients}, Compos. Math. 143, no. 1 (2007) 112--164. 

\bibitem[GHKK14]{GHKK14}
M. Gross, P. Hacking, S. Keel and M. Kontsevich,
 {\em Canonical bases for cluster algebras},
J. Amer. Math. Soc. 31, no. 2 (2018) 497--608. 

\bibitem[IIO21]{IIO21}
R.~Inoue, T.~Ishibashi and H.~Oya, {\em Cluster realizations of Weyl groups and higher Teichmuller theory}, 
Selecta Math. (N.S.) 27 (2021) no. 3, Paper No. 37, 84pp.

\bibitem[IK97]{IK97}
A.~P.~Isaev and P.~P.~Kulish, {\em Tetrahedron reflection equations}, 
Mod. Phys. Lett. A 12 (1997) 427--437.

\bibitem[IKT23]{IKT23}
R.~Inoue, A.~Kuniba and Y.~Terashima, 
{\em Tetrahedron equation and quantum cluster algebras}, in preparation.

\bibitem[KV94]{KV94}
M.~M.~Kapranov, V.~A.~Voevodsky,
{\em 2-Categories and Zamolodchikov tetrahedron equations}, 
Proc. Symposia in Pure Math. {56} (1994) 177--259.

\bibitem[KMS93]{KMS93}
R.~M.~Kashaev, V.~V.~Mangazeev, Yu.~G.~Stroganov,
{\em Star-square and tetrahedron equations in the Baxter-Bazhanov model},
Int. J. Mod. Phys. A {8}  (1993) 1399--1409.


\bibitem[KN11]{KN11}
R.~M.~Kashaev and T.~Nakanishi, 
{\em Classical and quantum dilogarithm identities}, SIGMA Symmetry Integrability Geom. Methods Appl. 7 (2011) Paper 102, 29pp.

\bibitem[Ke11]{Ke11}
B.~Keller, {\em On cluster theory and quantum dilogarithm identities, in Representations of Algebras and Related Topics}, Editors A. Skowro\'nski and K. Yamagata, EMS Series of Congress Reports (European Mathematical Society, 2011) 85--11.

\bibitem[K22]{K22} 
A.~Kuniba, 
\textit{Quantum groups in three-dimensional integrability}, Springer  Singapore (2022)

\bibitem[KMY23]{KMY23} 
A.~Kuniba, S.~Matsuike, A.~Yoneyama,
{\em New solutions to the tetrahedron equation
associated with quantized six-vertex models},
Commun. Math. Phys.10.1007/s00220-023-04711-y (2023). 


\bibitem[KO12]{KO12} A.~Kuniba, M.~Okado,
{\em Tetrahedron and 3D reflection equations
from quantized algebra of functions}, 
J. Phys. A: Math.Theor. {45} (2012) 465206, 27pp. 


\bibitem[KO13]{KO13}
A.~Kuniba and M.~Okado, {\em A solution of the 3D reflection equation from quantized algebra of functions of type B}, 
Nankai Series in Pure, Applied Mathematics and Theoretical Physics,
vol. 11 (World Scientific, 2013) 181--190.


\bibitem[MS97]{MS97}
J.~-M.~Maillard and S.~M.~Sergeev,
{\em Three-dimensional integrable models based on modified tetrahedron equations and quantum dilogarithm}, Phys. Lett. B 405 (1997) 55--63.

\bibitem[N21]{N21}
T.~Nakanishi, T. {\em Synchronicity phenomenon in cluster patterns}, 
J. Lond. Math. Soc. (2) (2021) 103 (3), 1120--1152.

\bibitem[R09]{R09}
M.~Reineke, 
{\em Poisson automorphisms and quiver moduli}, 
J. Inst. Math. Jussieu 9 (2009), 653--667.

\bibitem[S08]{S08}
S.~M.~Sergeev,
{\em Tetrahedron equations and nilpotent subalgebras of $\mathcal{U}_q(sl_n)$},
Lett. Math. Phys. {83} (2008) 231--235.

\bibitem[S12]{S12}
V.~P.~Spiridonov, 
{\em Elliptic beta integrals and solvable models of statistical mechanics},
Contemp. Math. 563 (2012) 181--211.

\bibitem[SMS96]{SMS96}
S.~M.~Sergeev, V.~Mangazeev and Yu.~G.~Stroganov,
{\em The vertex formulation of the Bazhanov-Baxter model},
J. Stat. Phys. {82} (1996) 31--49. 

\bibitem[SY22]{SY22}
X.~ Sun and J.~ Yagi,
{\em Cluster transformations, the tetrahedron equation and three-dimensional gauge theories}, arXiv:2211.10702.

\bibitem[TY12]{TY12}
Y.~Terashima and M.~Yamazaki, 
{\em Emergent $3$-manifolds from four dimensional superconformal indices}, 
Phys. Rev. Lett.  109 (9) (2012).

\bibitem[Yag15]{Yag15}
J.~Yagi, {\em Quiver gauge theories and integrable lattice models},
JHEP, 10:065 (2015).

\bibitem[Yam12]{Y12}
M.~Yamazaki, {\em Quivers, YBE and $3$-manifolds}, JHEP, 05:147 (2012).

\bibitem[Y21]{Y21}
A.~Yoneyama,
{\em Tetrahedron and 3D reflection equation from PBW bases of the nilpotent subalgebra of quantum superalgebras},
Commun. Math. Phys. {387}  (2021) 481--550.

\bibitem[Z80]{Z80}
A.~B.~Zamolodchikov,
{\em Tetrahedra equations and integrable systems in three-dimensional space},
Soviet Phys. JETP {79}  (1980) 641--664.

\bibitem[Z81]{Z81}
A.~B.~Zamolodchikov,
{\em Tetrahedron equations and relativistic 
$S$ matrix of straight strings in $(2+1)$-dimensions},
Commun. Math. Phys. {79} (1981) 489--505.

\end{thebibliography}
\end{document}